\documentclass{article}

\usepackage[a4paper, total={6in, 8in}]{geometry}

\usepackage{amsmath, amsthm, amssymb}
\usepackage{mathtools}
\usepackage{mathrsfs}

\numberwithin{equation}{section}

\usepackage{graphicx}
\usepackage{subcaption}
\usepackage{standalone}
\usepackage{tikz}

\usetikzlibrary{calc}

\usepackage{enumitem}
\usepackage{todonotes}
\usepackage{theoremref}

\usepackage{hyperref}

\hypersetup{
    colorlinks=true,
    linkcolor=blue,
    filecolor=blue,
    urlcolor=blue,
    citecolor=blue,
    pdftitle={From local giants to locality in long-range percolation},
    pdfpagemode=FullScreen
}

\theoremstyle{plain}

\newtheorem{theorem}{Theorem}[section]
\newtheorem{corollary}[theorem]{Corollary}

\newtheorem{proposition}[theorem]{Proposition}
\newtheorem{lemma}[theorem]{Lemma}

\theoremstyle{definition}

\newtheorem{definition}[theorem]{Definition}
\newtheorem{example}[theorem]{Example}

\theoremstyle{remark}

\newtheorem{remark}[theorem]{Remark}

\newcommand{\LRP}{long-range percolation}
\newcommand{\BNNP}{Bernoulli percolation}
\newcommand{\ER}{Erd\H{o}s-R\'enyi}

\renewcommand{\epsilon}{\varepsilon}

\newcommand{\norm}[1]{\|#1\|}
\newcommand{\charf}{\mathbf{1}}
\newcommand{\dd}{\,\mathrm{d}}

\newcommand{\aut}{\mathrm{Aut}}
\newcommand{\Aut}{\mathrm{Aut}}

\newcommand{\diam}{\mathrm{diam}}

\newcommand{\Bin}{\mathrm{Bin}}

\newcommand{\E}{\mathbf{E}}
\renewcommand{\P}{\mathbf{P}}

\newcommand{\Z}{\mathbf{Z}}
\newcommand{\R}{\mathbf{R}}
\newcommand{\N}{\mathbf{N}}

\renewcommand{\L}{\mathbf{L}}

\renewcommand{\AA}{\mathcal{A}}
\newcommand{\CC}{\mathcal{C}}

\newcommand{\GG}{\mathcal{G}}
\newcommand{\HH}{\mathcal{H}}
\newcommand{\NN}{\mathcal{N}}

\newcommand{\radius}{\overline{r}}

\newcommand{\tile}[2]{T_{#1}(#2)}
\newcommand{\Tile}[2]{T_{#1,#2}}

\newcommand{\cell}[2]{\mathrm{Cell}_{#1}(#2)}
\newcommand{\Cell}[2]{\mathrm{Cell}_{#1,#2}}

\newcommand{\vor}[1]{\mathrm{Vor}(#1)}

\newcommand{\numtiles}[1]{\kappa_{#1,j}}

\newcommand{\numsub}[1]{\kappa_{#1,j}}

\newcommand{\netvertex}[2]{x_{#1,#2}}

\newcommand{\cvol}{c_G}

\newcommand{\cvor}{c_{\mathrm{vor}}}
\newcommand{\Cvor}{C_{\mathrm{vor}}}

\newcommand{\cluster}{K}

\newcommand{\kinfty}{\cluster_{\infty}}

\newcommand{\giantone}[1]{\cluster_{#1}^{\scriptscriptstyle (1)}}
\newcommand{\gianttwo}[1]{\cluster_{#1}^{\scriptscriptstyle (2)}}
\newcommand{\giantcomp}[2]{\cluster_{#1,#2}^{\scriptscriptstyle (1)}}

\newcommand{\goodcluster}[2]{\mathcal{C}_{#1,#2}}
\newcommand{\goodsize}[2]{\mathcal{S}_{#1,#2}}

\newcommand{\loccomp}[2]{\cluster_{#1,#2}}

\newcommand{\level}[1]{level-$#1$}

\newcommand{\error}[1]{\mathrm{err}_{#1}}

\newif\ifusecolor
\newif\ifshowlabels
\newif\ifshowsites
\newif\ifshowcoarseguide
\newif\ifshowselectedmerge

\usecolortrue
\showlabelstrue
\showsitestrue
\showcoarseguidetrue
\showselectedmergetrue

\def\r{4.0}
\def\vorRadius{4.35}
\def\M{25}

\def\fineedgewidth{0.22pt}
\def\coarseedgewidth{0.95pt}
\def\mergededgewidth{1.15pt}
\def\selectededgewidth{1.75pt}

\newcommand{\sitecolor}[1]{%
  \ifcase#1 black%
  \or red%
  \or teal%
  \or orange%
  \or cyan%
  \or blue%
  \or violet%
  \or lime!65!black%
  \or pink!85!red%
  \else black%
  \fi
}

\def\finesites{
    -1.351/-3.438/1/5,
    0.473/-3.393/2/5,
    -0.088/-3.366/3/5,
    -0.594/-3.337/4/5,
    1.842/-3.169/5/5,
    -0.222/-2.838/6/5,
    -1.648/-2.813/7/5,
    -0.794/-2.755/8/5,
    2.772/-2.702/9/7,
    -2.721/-2.679/10/2,
    -2.237/-2.648/11/2,
    0.869/-2.615/12/5,
    -3.139/-2.317/13/2,
    -0.084/-2.228/14/5,
    2.076/-2.202/15/7,
    -1.224/-2.196/16/5,
    -2.507/-2.161/17/2,
    1.122/-2.152/18/5,
    -1.973/-2.061/19/2,
    3.173/-2.044/20/7,
    -0.907/-1.843/21/5,
    2.711/-1.794/22/7,
    -0.141/-1.758/23/5,
    1.563/-1.682/24/7,
    -2.882/-1.653/25/2,
    -1.691/-1.598/26/2,
    -3.538/-1.560/27/2,
    1.104/-1.535/28/5,
    0.496/-1.509/29/5,
    -2.313/-1.493/30/2,
    1.941/-1.278/31/7,
    -0.777/-1.209/32/4,
    -0.095/-1.197/33/5,
    -3.032/-1.138/34/2,
    -1.790/-1.117/35/2,
    -1.320/-1.101/36/2,
    2.623/-1.041/37/7,
    3.254/-1.010/38/7,
    1.133/-0.976/39/7,
    0.332/-0.679/40/4,
    -2.935/-0.555/41/2,
    3.349/-0.552/42/7,
    1.595/-0.506/43/7,
    -1.717/-0.505/44/2,
    0.968/-0.489/45/7,
    -3.530/-0.463/46/2,
    2.763/-0.447/47/7,
    -1.092/-0.426/48/4,
    -0.401/-0.409/49/4,
    2.193/-0.389/50/7,
    1.314/-0.140/51/7,
    -2.095/-0.070/52/2,
    0.541/-0.040/53/4,
    -3.713/0.020/54/2,
    -0.748/0.105/55/4,
    -3.253/0.115/56/2,
    3.048/0.116/57/7,
    -0.139/0.126/58/4,
    3.568/0.389/59/7,
    1.130/0.504/60/6,
    -0.330/0.513/61/4,
    2.158/0.538/62/7,
    -1.102/0.545/63/4,
    -2.387/0.575/64/1,
    0.163/0.614/65/4,
    -3.031/0.643/66/1,
    3.050/0.655/67/7,
    -3.475/0.673/68/1,
    -1.697/0.699/69/4,
    1.207/0.990/70/6,
    3.751/1.013/71/8,
    0.568/1.016/72/6,
    -3.033/1.081/73/1,
    2.467/1.157/74/8,
    0.015/1.179/75/4,
    -0.765/1.268/76/4,
    -1.208/1.269/77/4,
    -2.766/1.484/78/1,
    3.570/1.493/79/8,
    2.810/1.527/80/8,
    -2.115/1.634/81/1,
    0.304/1.739/82/6,
    -0.482/1.802/83/3,
    2.242/1.810/84/8,
    1.757/1.820/85/6,
    -1.599/1.823/86/1,
    1.011/1.836/87/6,
    -1.956/2.122/88/1,
    -1.207/2.140/89/3,
    -2.601/2.175/90/1,
    1.997/2.201/91/6,
    -0.637/2.243/92/3,
    0.633/2.379/93/6,
    -1.076/2.605/94/3,
    1.734/2.630/95/6,
    0.963/2.665/96/6,
    2.704/2.671/97/8,
    0.154/2.708/98/3,
    -1.566/2.904/99/3,
    -0.393/2.926/100/3,
    2.091/2.950/101/8,
    -1.984/3.161/102/3,
    -1.198/3.445/103/3,
    1.249/3.462/104/6,
    -0.090/3.492/105/3,
    0.754/3.494/106/6
}

\def\coarsesites{
    -2.750/1.850/1,
    -2.850/-1.150/2,
    -1.050/3.000/3,
    -0.650/0.450/4,
    0.100/-2.750/5,
    1.250/1.950/6,
    2.650/-0.700/7,
    3.000/2.350/8
}

\newcommand{\finecell}[5]{}
\newcommand{\finecells}{%
  \finecell{5}{1}{-1.351}{-3.438}{(-1.004,-3.152) -- (-1.207,-2.986) -- (-2.094,-3.408) -- (-2.090,-3.411) -- (-1.954,-3.490) -- (-1.816,-3.564) -- (-1.675,-3.633) -- (-1.531,-3.696) -- (-1.384,-3.753) -- (-1.236,-3.804) -- (-1.086,-3.850) -- (-0.934,-3.889) -- (-0.905,-3.896) -- cycle}
  \finecell{5}{2}{0.473}{-3.393}{(1.152,-3.249) -- (0.347,-2.839) -- (0.210,-3.010) -- (0.162,-3.997) -- (0.314,-3.988) -- (0.470,-3.972) -- (0.626,-3.951) -- (0.780,-3.923) -- (0.934,-3.889) -- (1.086,-3.850) -- (1.236,-3.804) -- (1.243,-3.802) -- cycle}
  \finecell{5}{3}{-0.088}{-3.366}{(0.210,-3.010) -- (-0.329,-3.146) -- (-0.377,-3.981) -- (-0.314,-3.988) -- (-0.157,-3.997) -- (0.000,-4.000) -- (0.157,-3.997) -- (0.162,-3.997) -- cycle}
  \finecell{5}{4}{-0.594}{-3.337}{(-0.329,-3.146) -- (-0.536,-2.991) -- (-1.004,-3.152) -- (-0.905,-3.896) -- (-0.780,-3.923) -- (-0.626,-3.951) -- (-0.470,-3.972) -- (-0.377,-3.981) -- cycle}
  \finecell{5}{5}{1.842}{-3.169}{(2.212,-2.747) -- (1.578,-2.593) -- (1.490,-2.654) -- (1.152,-3.249) -- (1.243,-3.802) -- (1.384,-3.753) -- (1.531,-3.696) -- (1.675,-3.633) -- (1.816,-3.564) -- (1.954,-3.490) -- (2.090,-3.411) -- (2.222,-3.326) -- (2.351,-3.236) -- (2.428,-3.178) -- cycle}
  \finecell{5}{6}{-0.222}{-2.838}{(0.305,-2.636) -- (-0.460,-2.464) -- (-0.536,-2.991) -- (-0.329,-3.146) -- (0.210,-3.010) -- (0.347,-2.839) -- cycle}
  \finecell{5}{7}{-1.648}{-2.813}{(-1.231,-2.646) -- (-1.641,-2.364) -- (-1.867,-2.462) -- (-2.126,-3.387) -- (-2.094,-3.408) -- (-1.207,-2.986) -- cycle}
  \finecell{5}{8}{-0.794}{-2.755}{(-0.536,-2.991) -- (-0.460,-2.464) -- (-0.605,-2.268) -- (-0.765,-2.288) -- (-1.231,-2.646) -- (-1.207,-2.986) -- (-1.004,-3.152) -- cycle}
  \finecell{7}{9}{2.772}{-2.702}{(2.765,-2.246) -- (2.562,-2.260) -- (2.212,-2.747) -- (2.428,-3.178) -- (2.476,-3.141) -- (2.598,-3.042) -- (2.715,-2.937) -- (2.828,-2.828) -- (2.937,-2.715) -- (3.042,-2.598) -- (3.141,-2.476) -- (3.142,-2.476) -- cycle}
  \finecell{2}{10}{-2.721}{-2.679}{(-2.491,-2.470) -- (-2.797,-2.344) -- (-3.029,-2.612) -- (-2.937,-2.715) -- (-2.828,-2.828) -- (-2.715,-2.937) -- (-2.598,-3.042) -- (-2.476,-3.141) -- (-2.448,-3.163) -- cycle}
  \finecell{2}{11}{-2.237}{-2.648}{(-1.867,-2.462) -- (-2.203,-2.311) -- (-2.491,-2.470) -- (-2.448,-3.163) -- (-2.351,-3.236) -- (-2.222,-3.326) -- (-2.126,-3.387) -- cycle}
  \finecell{5}{12}{0.869}{-2.615}{(1.490,-2.654) -- (0.515,-2.121) -- (0.305,-2.636) -- (0.347,-2.839) -- (1.152,-3.249) -- cycle}
  \finecell{2}{13}{-3.139}{-2.317}{(-2.797,-2.344) -- (-2.873,-2.038) -- (-3.250,-1.892) -- (-3.461,-2.003) -- (-3.411,-2.090) -- (-3.326,-2.222) -- (-3.236,-2.351) -- (-3.141,-2.476) -- (-3.042,-2.598) -- (-3.029,-2.612) -- cycle}
  \finecell{5}{14}{-0.084}{-2.228}{(0.515,-2.121) -- (0.514,-2.117) -- (0.299,-1.943) -- (-0.498,-2.040) -- (-0.605,-2.268) -- (-0.460,-2.464) -- (0.305,-2.636) -- cycle}
  \finecell{7}{15}{2.076}{-2.202}{(2.209,-1.711) -- (2.027,-1.737) -- (1.600,-2.159) -- (1.578,-2.593) -- (2.212,-2.747) -- (2.562,-2.260) -- cycle}
  \finecell{5}{16}{-1.224}{-2.196}{(-1.321,-1.791) -- (-1.573,-1.987) -- (-1.641,-2.364) -- (-1.231,-2.646) -- (-0.765,-2.288) -- cycle}
  \finecell{2}{17}{-2.507}{-2.161}{(-2.536,-1.790) -- (-2.873,-2.038) -- (-2.797,-2.344) -- (-2.491,-2.470) -- (-2.203,-2.311) -- (-2.286,-1.863) -- cycle}
  \finecell{5}{18}{1.122}{-2.152}{(1.600,-2.159) -- (1.260,-1.839) -- (0.786,-1.853) -- (0.514,-2.117) -- (0.515,-2.121) -- (1.490,-2.654) -- (1.578,-2.593) -- cycle}
  \finecell{2}{19}{-1.973}{-2.061}{(-2.029,-1.709) -- (-2.286,-1.863) -- (-2.203,-2.311) -- (-1.867,-2.462) -- (-1.641,-2.364) -- (-1.573,-1.987) -- cycle}
  \finecell{7}{20}{3.173}{-2.044}{(3.157,-1.522) -- (2.765,-2.246) -- (3.142,-2.476) -- (3.236,-2.351) -- (3.326,-2.222) -- (3.411,-2.090) -- (3.490,-1.954) -- (3.564,-1.816) -- (3.633,-1.675) -- (3.681,-1.563) -- cycle}
  \finecell{5}{21}{-0.907}{-1.843}{(-1.111,-1.470) -- (-1.244,-1.545) -- (-1.321,-1.791) -- (-0.765,-2.288) -- (-0.605,-2.268) -- (-0.498,-2.040) -- (-0.548,-1.586) -- cycle}
  \finecell{7}{22}{2.711}{-1.794}{(3.157,-1.522) -- (2.957,-1.384) -- (2.383,-1.450) -- (2.209,-1.711) -- (2.562,-2.260) -- (2.765,-2.246) -- cycle}
  \finecell{5}{23}{-0.141}{-1.758}{(0.124,-1.497) -- (-0.432,-1.452) -- (-0.548,-1.586) -- (-0.498,-2.040) -- (0.299,-1.943) -- cycle}
  \finecell{7}{24}{1.563}{-1.682}{(1.495,-1.240) -- (1.442,-1.272) -- (1.260,-1.839) -- (1.600,-2.159) -- (2.027,-1.737) -- cycle}
  \finecell{2}{25}{-2.882}{-1.653}{(-2.671,-1.312) -- (-3.190,-1.463) -- (-3.250,-1.892) -- (-2.873,-2.038) -- (-2.536,-1.790) -- cycle}
  \finecell{2}{26}{-1.691}{-1.598}{(-1.548,-1.318) -- (-1.979,-1.407) -- (-2.029,-1.709) -- (-1.573,-1.987) -- (-1.321,-1.791) -- (-1.244,-1.545) -- cycle}
  \finecell{2}{27}{-3.538}{-1.560}{(-3.250,-1.892) -- (-3.190,-1.463) -- (-3.567,-1.011) -- (-3.870,-1.009) -- (-3.850,-1.086) -- (-3.804,-1.236) -- (-3.753,-1.384) -- (-3.696,-1.531) -- (-3.633,-1.675) -- (-3.564,-1.816) -- (-3.490,-1.954) -- (-3.461,-2.003) -- cycle}
  \finecell{5}{28}{1.104}{-1.535}{(0.812,-1.240) -- (0.786,-1.853) -- (1.260,-1.839) -- (1.442,-1.272) -- cycle}
  \finecell{5}{29}{0.496}{-1.509}{(0.651,-1.047) -- (0.328,-1.111) -- (0.124,-1.497) -- (0.299,-1.943) -- (0.514,-2.117) -- (0.786,-1.853) -- (0.812,-1.240) -- cycle}
  \finecell{2}{30}{-2.313}{-1.493}{(-2.373,-0.858) -- (-2.482,-0.930) -- (-2.671,-1.312) -- (-2.536,-1.790) -- (-2.286,-1.863) -- (-2.029,-1.709) -- (-1.979,-1.407) -- cycle}
  \finecell{7}{31}{1.941}{-1.278}{(2.383,-1.450) -- (2.180,-0.866) -- (1.964,-0.804) -- (1.596,-0.969) -- (1.495,-1.240) -- (2.027,-1.737) -- (2.209,-1.711) -- cycle}
  \finecell{4}{32}{-0.777}{-1.209}{(-0.739,-0.738) -- (-0.986,-0.838) -- (-1.111,-1.470) -- (-0.548,-1.586) -- (-0.432,-1.452) -- (-0.442,-0.878) -- cycle}
  \finecell{5}{33}{-0.095}{-1.197}{(-0.110,-0.749) -- (-0.442,-0.878) -- (-0.432,-1.452) -- (0.124,-1.497) -- (0.328,-1.111) -- cycle}
  \finecell{2}{34}{-3.032}{-1.138}{(-3.277,-0.797) -- (-3.567,-1.011) -- (-3.190,-1.463) -- (-2.671,-1.312) -- (-2.482,-0.930) -- cycle}
  \finecell{2}{35}{-1.790}{-1.117}{(-2.317,-0.744) -- (-2.373,-0.858) -- (-1.979,-1.407) -- (-1.548,-1.318) -- (-1.564,-0.834) -- cycle}
  \finecell{2}{36}{-1.320}{-1.101}{(-1.374,-0.707) -- (-1.564,-0.834) -- (-1.548,-1.318) -- (-1.244,-1.545) -- (-1.111,-1.470) -- (-0.986,-0.838) -- cycle}
  \finecell{7}{37}{2.623}{-1.041}{(2.927,-0.799) -- (2.450,-0.687) -- (2.180,-0.866) -- (2.383,-1.450) -- (2.957,-1.384) -- cycle}
  \finecell{7}{38}{3.254}{-1.010}{(3.016,-0.722) -- (2.927,-0.799) -- (2.957,-1.384) -- (3.157,-1.522) -- (3.681,-1.563) -- (3.696,-1.531) -- (3.753,-1.384) -- (3.804,-1.236) -- (3.850,-1.086) -- (3.889,-0.934) -- (3.896,-0.904) -- cycle}
  \finecell{7}{39}{1.133}{-0.976}{(1.277,-0.656) -- (0.727,-0.842) -- (0.651,-1.047) -- (0.812,-1.240) -- (1.442,-1.272) -- (1.495,-1.240) -- (1.596,-0.969) -- cycle}
  \finecell{4}{40}{0.332}{-0.679}{(0.599,-0.413) -- (0.128,-0.258) -- (0.055,-0.301) -- (-0.110,-0.749) -- (0.328,-1.111) -- (0.651,-1.047) -- (0.727,-0.842) -- cycle}
  \finecell{2}{41}{-2.935}{-0.555}{(-2.681,-0.024) -- (-3.195,-0.268) -- (-3.277,-0.797) -- (-2.482,-0.930) -- (-2.373,-0.858) -- (-2.317,-0.744) -- (-2.321,-0.648) -- cycle}
  \finecell{7}{42}{3.349}{-0.552}{(3.486,-0.088) -- (3.098,-0.263) -- (3.016,-0.722) -- (3.896,-0.904) -- (3.923,-0.780) -- (3.951,-0.626) -- (3.972,-0.470) -- (3.988,-0.314) -- (3.994,-0.206) -- cycle}
  \finecell{7}{43}{1.595}{-0.506}{(1.816,-0.045) -- (1.282,-0.455) -- (1.277,-0.656) -- (1.596,-0.969) -- (1.964,-0.804) -- cycle}
  \finecell{2}{44}{-1.717}{-0.505}{(-1.494,0.070) -- (-2.321,-0.648) -- (-2.317,-0.744) -- (-1.564,-0.834) -- (-1.374,-0.707) -- (-1.471,0.056) -- cycle}
  \finecell{7}{45}{0.968}{-0.489}{(0.926,-0.101) -- (0.599,-0.413) -- (0.727,-0.842) -- (1.277,-0.656) -- (1.282,-0.455) -- cycle}
  \finecell{2}{46}{-3.530}{-0.463}{(-3.567,-1.011) -- (-3.277,-0.797) -- (-3.195,-0.268) -- (-3.438,-0.152) -- (-3.983,-0.359) -- (-3.972,-0.470) -- (-3.951,-0.626) -- (-3.923,-0.780) -- (-3.889,-0.934) -- (-3.870,-1.009) -- cycle}
  \finecell{7}{47}{2.763}{-0.447}{(3.016,-0.722) -- (3.098,-0.263) -- (2.523,0.028) -- (2.450,-0.687) -- (2.927,-0.799) -- cycle}
  \finecell{4}{48}{-1.092}{-0.426}{(-1.257,0.058) -- (-1.471,0.056) -- (-1.374,-0.707) -- (-0.986,-0.838) -- (-0.739,-0.738) -- (-0.750,-0.270) -- cycle}
  \finecell{4}{49}{-0.401}{-0.409}{(-0.438,-0.059) -- (-0.750,-0.270) -- (-0.739,-0.738) -- (-0.442,-0.878) -- (-0.110,-0.749) -- (0.055,-0.301) -- cycle}
  \finecell{7}{50}{2.193}{-0.389}{(2.450,-0.687) -- (2.523,0.028) -- (2.489,0.086) -- (1.846,0.062) -- (1.816,-0.045) -- (1.964,-0.804) -- (2.180,-0.866) -- cycle}
  \finecell{7}{51}{1.314}{-0.140}{(1.651,0.305) -- (0.953,0.105) -- (0.926,-0.101) -- (1.282,-0.455) -- (1.816,-0.045) -- (1.846,0.062) -- cycle}
  \finecell{2}{52}{-2.095}{-0.070}{(-1.528,0.124) -- (-1.993,0.365) -- (-2.668,0.058) -- (-2.681,-0.024) -- (-2.321,-0.648) -- (-1.494,0.070) -- cycle}
  \finecell{4}{53}{0.541}{-0.040}{(0.634,0.450) -- (0.246,0.225) -- (0.128,-0.258) -- (0.599,-0.413) -- (0.926,-0.101) -- (0.953,0.105) -- cycle}
  \finecell{2}{54}{-3.713}{0.020}{(-3.438,-0.152) -- (-3.536,0.326) -- (-3.970,0.484) -- (-3.972,0.470) -- (-3.988,0.314) -- (-3.997,0.157) -- (-4.000,0.000) -- (-3.997,-0.157) -- (-3.988,-0.314) -- (-3.983,-0.359) -- cycle}
  \finecell{4}{55}{-0.748}{0.105}{(-0.718,0.492) -- (-1.257,0.058) -- (-0.750,-0.270) -- (-0.438,-0.059) -- (-0.447,0.214) -- cycle}
  \finecell{2}{56}{-3.253}{0.115}{(-3.268,0.432) -- (-3.536,0.326) -- (-3.438,-0.152) -- (-3.195,-0.268) -- (-2.681,-0.024) -- (-2.668,0.058) -- (-2.751,0.214) -- cycle}
  \finecell{7}{57}{3.048}{0.116}{(3.238,0.385) -- (2.632,0.387) -- (2.489,0.086) -- (2.523,0.028) -- (3.098,-0.263) -- (3.486,-0.088) -- cycle}
  \finecell{4}{58}{-0.139}{0.126}{(-0.052,0.410) -- (-0.447,0.214) -- (-0.438,-0.059) -- (0.055,-0.301) -- (0.128,-0.258) -- (0.246,0.225) -- cycle}
  \finecell{7}{59}{3.568}{0.389}{(4.000,0.000) -- (3.997,0.157) -- (3.988,0.314) -- (3.972,0.470) -- (3.952,0.615) -- (3.435,0.767) -- (3.238,0.385) -- (3.486,-0.088) -- (3.994,-0.206) -- (3.997,-0.157) -- cycle}
  \finecell{6}{60}{1.130}{0.504}{(0.879,0.793) -- (0.644,0.535) -- (0.634,0.450) -- (0.953,0.105) -- (1.651,0.305) -- (1.639,0.672) -- cycle}
  \finecell{4}{61}{-0.330}{0.513}{(-0.139,0.837) -- (-0.403,0.974) -- (-0.705,0.800) -- (-0.718,0.492) -- (-0.447,0.214) -- (-0.052,0.410) -- cycle}
  \finecell{7}{62}{2.158}{0.538}{(1.835,1.086) -- (1.639,0.672) -- (1.651,0.305) -- (1.846,0.062) -- (2.489,0.086) -- (2.632,0.387) -- (2.590,0.709) -- cycle}
  \finecell{4}{63}{-1.102}{0.545}{(-0.987,0.931) -- (-1.332,0.881) -- (-1.528,0.124) -- (-1.494,0.070) -- (-1.471,0.056) -- (-1.257,0.058) -- (-0.718,0.492) -- (-0.705,0.800) -- cycle}
  \finecell{1}{64}{-2.387}{0.575}{(-2.120,1.071) -- (-2.341,1.127) -- (-2.540,1.044) -- (-2.682,0.863) -- (-2.751,0.214) -- (-2.668,0.058) -- (-1.993,0.365) -- cycle}
  \finecell{4}{65}{0.163}{0.614}{(0.244,0.937) -- (-0.139,0.837) -- (-0.052,0.410) -- (0.246,0.225) -- (0.634,0.450) -- (0.644,0.535) -- cycle}
  \finecell{1}{66}{-3.031}{0.643}{(-3.239,0.861) -- (-3.268,0.432) -- (-2.751,0.214) -- (-2.682,0.863) -- cycle}
  \finecell{7}{67}{3.050}{0.655}{(3.263,1.103) -- (3.174,1.158) -- (2.915,1.087) -- (2.590,0.709) -- (2.632,0.387) -- (3.238,0.385) -- (3.435,0.767) -- cycle}
  \finecell{1}{68}{-3.475}{0.673}{(-3.536,0.326) -- (-3.268,0.432) -- (-3.239,0.861) -- (-3.744,1.408) -- (-3.753,1.384) -- (-3.804,1.236) -- (-3.850,1.086) -- (-3.889,0.934) -- (-3.923,0.780) -- (-3.951,0.626) -- (-3.970,0.484) -- cycle}
  \finecell{4}{69}{-1.697}{0.699}{(-1.748,1.237) -- (-2.120,1.071) -- (-1.993,0.365) -- (-1.528,0.124) -- (-1.332,0.881) -- cycle}
  \finecell{6}{70}{1.207}{0.990}{(1.376,1.475) -- (0.903,1.365) -- (0.879,0.793) -- (1.639,0.672) -- (1.835,1.086) -- (1.823,1.178) -- cycle}
  \finecell{8}{71}{3.751}{1.013}{(3.263,1.103) -- (3.435,0.767) -- (3.952,0.615) -- (3.951,0.626) -- (3.923,0.780) -- (3.889,0.934) -- (3.850,1.086) -- (3.804,1.236) -- (3.782,1.299) -- cycle}
  \finecell{6}{72}{0.568}{1.016}{(0.700,1.474) -- (0.367,1.352) -- (0.244,0.937) -- (0.644,0.535) -- (0.879,0.793) -- (0.903,1.365) -- cycle}
  \finecell{1}{73}{-3.033}{1.081}{(-3.239,0.861) -- (-2.682,0.863) -- (-2.540,1.044) -- (-3.598,1.745) -- (-3.633,1.675) -- (-3.696,1.531) -- (-3.744,1.408) -- cycle}
  \finecell{8}{74}{2.467}{1.157}{(2.450,1.517) -- (1.991,1.358) -- (1.823,1.178) -- (1.835,1.086) -- (2.590,0.709) -- (2.915,1.087) -- cycle}
  \finecell{4}{75}{0.015}{1.179}{(-0.103,1.595) -- (-0.356,1.393) -- (-0.403,0.974) -- (-0.139,0.837) -- (0.244,0.937) -- (0.367,1.352) -- cycle}
  \finecell{4}{76}{-0.765}{1.268}{(-0.964,1.715) -- (-0.985,1.704) -- (-0.987,0.931) -- (-0.705,0.800) -- (-0.403,0.974) -- (-0.356,1.393) -- cycle}
  \finecell{4}{77}{-1.208}{1.269}{(-1.179,1.705) -- (-1.711,1.329) -- (-1.748,1.237) -- (-1.332,0.881) -- (-0.987,0.931) -- (-0.985,1.704) -- cycle}
  \finecell{1}{78}{-2.766}{1.484}{(-2.540,1.044) -- (-2.341,1.127) -- (-2.493,1.784) -- (-3.455,2.014) -- (-3.490,1.954) -- (-3.564,1.816) -- (-3.598,1.745) -- cycle}
  \finecell{8}{79}{3.570}{1.493}{(3.219,2.142) -- (3.174,1.158) -- (3.263,1.103) -- (3.782,1.299) -- (3.753,1.384) -- (3.696,1.531) -- (3.633,1.675) -- (3.564,1.816) -- (3.490,1.954) -- (3.411,2.090) -- (3.327,2.221) -- cycle}
  \finecell{8}{80}{2.810}{1.527}{(2.739,2.098) -- (2.450,1.517) -- (2.915,1.087) -- (3.174,1.158) -- (3.219,2.142) -- cycle}
  \finecell{1}{81}{-2.115}{1.634}{(-2.294,1.962) -- (-2.493,1.784) -- (-2.341,1.127) -- (-2.120,1.071) -- (-1.748,1.237) -- (-1.711,1.329) -- (-1.895,1.833) -- cycle}
  \finecell{6}{82}{0.304}{1.739}{(0.167,2.214) -- (-0.057,2.180) -- (-0.103,1.595) -- (0.367,1.352) -- (0.700,1.474) -- (0.632,1.975) -- cycle}
  \finecell{3}{83}{-0.482}{1.802}{(-0.074,2.192) -- (-0.871,1.913) -- (-0.964,1.715) -- (-0.356,1.393) -- (-0.103,1.595) -- (-0.057,2.180) -- cycle}
  \finecell{8}{84}{2.242}{1.810}{(2.485,2.235) -- (2.002,1.932) -- (1.991,1.358) -- (2.450,1.517) -- (2.739,2.098) -- cycle}
  \finecell{6}{85}{1.757}{1.820}{(1.546,2.219) -- (1.392,2.214) -- (1.376,1.475) -- (1.823,1.178) -- (1.991,1.358) -- (2.002,1.932) -- cycle}
  \finecell{1}{86}{-1.599}{1.823}{(-1.583,2.205) -- (-1.895,1.833) -- (-1.711,1.329) -- (-1.179,1.705) -- cycle}
  \finecell{6}{87}{1.011}{1.836}{(1.331,2.270) -- (1.032,2.253) -- (0.632,1.975) -- (0.700,1.474) -- (0.903,1.365) -- (1.376,1.475) -- (1.392,2.214) -- cycle}
  \finecell{1}{88}{-1.956}{2.122}{(-2.016,2.640) -- (-2.238,2.634) -- (-2.294,1.962) -- (-1.895,1.833) -- (-1.583,2.205) -- (-1.589,2.427) -- cycle}
  \finecell{3}{89}{-1.207}{2.140}{(-1.494,2.472) -- (-1.589,2.427) -- (-1.583,2.205) -- (-1.179,1.705) -- (-0.985,1.704) -- (-0.964,1.715) -- (-0.871,1.913) -- (-0.944,2.318) -- cycle}
  \finecell{1}{90}{-2.601}{2.175}{(-2.493,1.784) -- (-2.294,1.962) -- (-2.238,2.634) -- (-2.718,2.935) -- (-2.828,2.828) -- (-2.937,2.715) -- (-3.042,2.598) -- (-3.141,2.476) -- (-3.236,2.351) -- (-3.326,2.222) -- (-3.411,2.090) -- (-3.455,2.014) -- cycle}
  \finecell{6}{91}{1.997}{2.201}{(2.277,2.546) -- (2.113,2.567) -- (1.546,2.219) -- (2.002,1.932) -- (2.485,2.235) -- cycle}
  \finecell{3}{92}{-0.637}{2.243}{(-0.250,2.490) -- (-0.677,2.642) -- (-0.944,2.318) -- (-0.871,1.913) -- (-0.074,2.192) -- cycle}
  \finecell{6}{93}{0.633}{2.379}{(0.564,2.791) -- (0.167,2.214) -- (0.632,1.975) -- (1.032,2.253) -- cycle}
  \finecell{3}{94}{-1.076}{2.605}{(-0.874,3.063) -- (-1.158,3.022) -- (-1.494,2.472) -- (-0.944,2.318) -- (-0.677,2.642) -- cycle}
  \finecell{6}{95}{1.734}{2.630}{(1.617,3.119) -- (1.363,2.971) -- (1.331,2.270) -- (1.392,2.214) -- (1.546,2.219) -- (2.113,2.567) -- cycle}
  \finecell{6}{96}{0.963}{2.665}{(0.978,3.110) -- (0.576,3.009) -- (0.564,2.791) -- (1.032,2.253) -- (1.331,2.270) -- (1.363,2.971) -- cycle}
  \finecell{8}{97}{2.704}{2.671}{(2.277,2.546) -- (2.485,2.235) -- (2.739,2.098) -- (3.219,2.142) -- (3.327,2.221) -- (3.326,2.222) -- (3.236,2.351) -- (3.141,2.476) -- (3.042,2.598) -- (2.937,2.715) -- (2.828,2.828) -- (2.715,2.937) -- (2.598,3.042) -- (2.528,3.099) -- cycle}
  \finecell{3}{98}{0.154}{2.708}{(0.333,3.194) -- (-0.012,3.086) -- (-0.250,2.490) -- (-0.074,2.192) -- (-0.057,2.180) -- (0.167,2.214) -- (0.564,2.791) -- (0.576,3.009) -- cycle}
  \finecell{3}{99}{-1.566}{2.904}{(-1.597,3.321) -- (-2.016,2.640) -- (-1.589,2.427) -- (-1.494,2.472) -- (-1.158,3.022) -- cycle}
  \finecell{3}{100}{-0.393}{2.926}{(-0.642,3.424) -- (-0.874,3.063) -- (-0.677,2.642) -- (-0.250,2.490) -- (-0.012,3.086) -- cycle}
  \finecell{8}{101}{2.091}{2.950}{(1.617,3.119) -- (2.113,2.567) -- (2.277,2.546) -- (2.528,3.099) -- (2.476,3.141) -- (2.351,3.236) -- (2.222,3.326) -- (2.090,3.411) -- (1.954,3.490) -- (1.870,3.535) -- cycle}
  \finecell{3}{102}{-1.984}{3.161}{(-2.238,2.634) -- (-2.016,2.640) -- (-1.597,3.321) -- (-1.705,3.618) -- (-1.816,3.564) -- (-1.954,3.490) -- (-2.090,3.411) -- (-2.222,3.326) -- (-2.351,3.236) -- (-2.476,3.141) -- (-2.598,3.042) -- (-2.715,2.937) -- (-2.718,2.935) -- cycle}
  \finecell{3}{103}{-1.198}{3.445}{(-1.597,3.321) -- (-1.158,3.022) -- (-0.874,3.063) -- (-0.642,3.424) -- (-0.664,3.944) -- (-0.780,3.923) -- (-0.934,3.889) -- (-1.086,3.850) -- (-1.236,3.804) -- (-1.384,3.753) -- (-1.531,3.696) -- (-1.675,3.633) -- (-1.705,3.618) -- cycle}
  \finecell{6}{104}{1.249}{3.462}{(0.978,3.110) -- (1.363,2.971) -- (1.617,3.119) -- (1.870,3.535) -- (1.816,3.564) -- (1.675,3.633) -- (1.531,3.696) -- (1.384,3.753) -- (1.236,3.804) -- (1.086,3.850) -- (1.027,3.865) -- cycle}
  \finecell{3}{105}{-0.090}{3.492}{(-0.642,3.424) -- (-0.012,3.086) -- (0.333,3.194) -- (0.331,3.986) -- (0.314,3.988) -- (0.157,3.997) -- (0.000,4.000) -- (-0.157,3.997) -- (-0.314,3.988) -- (-0.470,3.972) -- (-0.626,3.951) -- (-0.664,3.944) -- cycle}
  \finecell{6}{106}{0.754}{3.494}{(0.333,3.194) -- (0.576,3.009) -- (0.978,3.110) -- (1.027,3.865) -- (0.934,3.889) -- (0.780,3.923) -- (0.626,3.951) -- (0.470,3.972) -- (0.331,3.986) -- cycle}
}

\newcommand{\mergededge}[4]{}
\newcommand{\mergededges}{%
  \mergededge{-2.126}{-3.387}{-1.867}{-2.462}
  \mergededge{2.212}{-2.747}{2.428}{-3.178}
  \mergededge{1.578}{-2.593}{2.212}{-2.747}
  \mergededge{1.578}{-2.593}{1.600}{-2.159}
  \mergededge{-1.867}{-2.462}{-1.641}{-2.364}
  \mergededge{-1.641}{-2.364}{-1.573}{-1.987}
  \mergededge{-1.573}{-1.987}{-1.321}{-1.791}
  \mergededge{1.260}{-1.839}{1.600}{-2.159}
  \mergededge{1.260}{-1.839}{1.442}{-1.272}
  \mergededge{-1.321}{-1.791}{-1.244}{-1.545}
  \mergededge{-0.548}{-1.586}{-0.432}{-1.452}
  \mergededge{-1.244}{-1.545}{-1.111}{-1.470}
  \mergededge{-1.111}{-1.470}{-0.548}{-1.586}
  \mergededge{-1.111}{-1.470}{-0.986}{-0.838}
  \mergededge{0.812}{-1.240}{1.442}{-1.272}
  \mergededge{0.328}{-1.111}{0.651}{-1.047}
  \mergededge{0.651}{-1.047}{0.812}{-1.240}
  \mergededge{0.651}{-1.047}{0.727}{-0.842}
  \mergededge{-0.442}{-0.878}{-0.432}{-1.452}
  \mergededge{-0.442}{-0.878}{-0.110}{-0.749}
  \mergededge{-0.110}{-0.749}{0.328}{-1.111}
  \mergededge{-1.374}{-0.707}{-0.986}{-0.838}
  \mergededge{0.599}{-0.413}{0.727}{-0.842}
  \mergededge{0.599}{-0.413}{0.926}{-0.101}
  \mergededge{0.926}{-0.101}{0.953}{0.105}
  \mergededge{-1.471}{0.056}{-1.374}{-0.707}
  \mergededge{-2.668}{0.058}{-1.993}{0.365}
  \mergededge{0.953}{0.105}{1.651}{0.305}
  \mergededge{-1.528}{0.124}{-1.494}{0.070}
  \mergededge{-2.751}{0.214}{-2.668}{0.058}
  \mergededge{-3.536}{0.326}{-3.268}{0.432}
  \mergededge{-1.993}{0.365}{-1.528}{0.124}
  \mergededge{-3.268}{0.432}{-2.751}{0.214}
  \mergededge{0.634}{0.450}{0.953}{0.105}
  \mergededge{0.634}{0.450}{0.644}{0.535}
  \mergededge{-3.970}{0.484}{-3.536}{0.326}
  \mergededge{1.639}{0.672}{1.651}{0.305}
  \mergededge{1.639}{0.672}{1.835}{1.086}
  \mergededge{2.590}{0.709}{2.915}{1.087}
  \mergededge{3.435}{0.767}{3.952}{0.615}
  \mergededge{0.244}{0.937}{0.644}{0.535}
  \mergededge{0.244}{0.937}{0.367}{1.352}
  \mergededge{-2.120}{1.071}{-1.993}{0.365}
  \mergededge{-2.120}{1.071}{-1.748}{1.237}
  \mergededge{1.835}{1.086}{2.590}{0.709}
  \mergededge{2.915}{1.087}{3.174}{1.158}
  \mergededge{3.263}{1.103}{3.435}{0.767}
  \mergededge{3.174}{1.158}{3.263}{1.103}
  \mergededge{1.823}{1.178}{1.835}{1.086}
  \mergededge{1.823}{1.178}{1.991}{1.358}
  \mergededge{-1.748}{1.237}{-1.711}{1.329}
  \mergededge{-1.711}{1.329}{-1.179}{1.705}
  \mergededge{1.991}{1.358}{2.002}{1.932}
  \mergededge{-0.356}{1.393}{-0.103}{1.595}
  \mergededge{-0.103}{1.595}{0.367}{1.352}
  \mergededge{-0.103}{1.595}{-0.057}{2.180}
  \mergededge{-1.179}{1.705}{-0.985}{1.704}
  \mergededge{-0.964}{1.715}{-0.356}{1.393}
  \mergededge{2.002}{1.932}{2.485}{2.235}
  \mergededge{-0.057}{2.180}{0.167}{2.214}
  \mergededge{-1.583}{2.205}{-1.179}{1.705}
  \mergededge{0.167}{2.214}{0.564}{2.791}
  \mergededge{-1.589}{2.427}{-1.583}{2.205}
  \mergededge{2.277}{2.546}{2.485}{2.235}
  \mergededge{2.113}{2.567}{2.277}{2.546}
  \mergededge{-2.238}{2.634}{-2.016}{2.640}
  \mergededge{-2.016}{2.640}{-1.589}{2.427}
  \mergededge{0.564}{2.791}{0.576}{3.009}
  \mergededge{-2.718}{2.935}{-2.238}{2.634}
  \mergededge{1.617}{3.119}{2.113}{2.567}
  \mergededge{1.617}{3.119}{1.870}{3.535}
  \mergededge{0.333}{3.194}{0.576}{3.009}
  \mergededge{0.331}{3.986}{0.333}{3.194}
}

\newcommand{\selectededge}[4]{}
\newcommand{\selectededges}{%
  \selectededge{-0.548}{-1.586}{-0.432}{-1.452}
  \selectededge{-1.111}{-1.470}{-0.548}{-1.586}
  \selectededge{-1.111}{-1.470}{-0.986}{-0.838}
  \selectededge{0.328}{-1.111}{0.651}{-1.047}
  \selectededge{0.651}{-1.047}{0.727}{-0.842}
  \selectededge{-0.442}{-0.878}{-0.432}{-1.452}
  \selectededge{-0.442}{-0.878}{-0.110}{-0.749}
  \selectededge{-0.110}{-0.749}{0.328}{-1.111}
  \selectededge{-1.374}{-0.707}{-0.986}{-0.838}
  \selectededge{0.599}{-0.413}{0.727}{-0.842}
  \selectededge{0.599}{-0.413}{0.926}{-0.101}
  \selectededge{0.926}{-0.101}{0.953}{0.105}
  \selectededge{-1.471}{0.056}{-1.374}{-0.707}
  \selectededge{-1.528}{0.124}{-1.494}{0.070}
  \selectededge{-1.993}{0.365}{-1.528}{0.124}
  \selectededge{0.634}{0.450}{0.953}{0.105}
  \selectededge{0.634}{0.450}{0.644}{0.535}
  \selectededge{0.244}{0.937}{0.644}{0.535}
  \selectededge{0.244}{0.937}{0.367}{1.352}
  \selectededge{-2.120}{1.071}{-1.993}{0.365}
  \selectededge{-2.120}{1.071}{-1.748}{1.237}
  \selectededge{-1.748}{1.237}{-1.711}{1.329}
  \selectededge{-1.711}{1.329}{-1.179}{1.705}
  \selectededge{-0.356}{1.393}{-0.103}{1.595}
  \selectededge{-0.103}{1.595}{0.367}{1.352}
  \selectededge{-1.179}{1.705}{-0.985}{1.704}
  \selectededge{-0.964}{1.715}{-0.356}{1.393}
}

\newcommand{\drawball}{%
  \fill[white] (0,0) circle (\vorRadius);
  \fill[black!2] (0,0) circle (\r);
}

\newcommand{\drawfinesites}{%
  \ifshowsites
    \foreach \x/\y/\idx/\owner in \finesites {
      \fill[black, fill opacity=0.55] (\x,\y) circle (0.030);
    }
  \fi
}

\newcommand{\drawcoarsesites}{%
  \ifshowsites
    \foreach \x/\y/\cidx in \coarsesites {
      \fill[white] (\x,\y) circle (0.105);
      \fill[\sitecolor{\cidx}] (\x,\y) circle (0.075);
      \draw[black, line width=0.35pt] (\x,\y) circle (0.105);
    }
  \fi
}

\newcommand{\drawfinecells}{%
  \begin{scope}
    \clip (0,0) circle (\r);
    \renewcommand{\finecell}[5]{%
      \filldraw[fill=black!3, draw=black!52, line width=\fineedgewidth] ##5;
    }
    \finecells
  \end{scope}
  \drawfinesites
}

\newcommand{\drawfinemesh}[1]{%
  \begin{scope}
    \clip (0,0) circle (\r);
    \renewcommand{\finecell}[5]{%
      \draw[#1] ##5;
    }
    \finecells
  \end{scope}
}

\newcommand{\drawmergedcells}{%
  \begin{scope}
    \clip (0,0) circle (\r);
    \renewcommand{\finecell}[5]{%
      \ifusecolor
        \ifnum##1=4
          \fill[\sitecolor{##1}!30] ##5;
        \else
          \fill[\sitecolor{##1}!14] ##5;
        \fi
      \else
        \ifnum##1=4
          \fill[black!22] ##5;
        \else
          \fill[black!7] ##5;
        \fi
      \fi
      \draw[black!35, line width=0.16pt] ##5;
    }
    \finecells
  \end{scope}
}

\newcommand{\drawmergededges}{%
  \begin{scope}
    \clip (0,0) circle (\r);
    \renewcommand{\mergededge}[4]{%
      \draw[black, line width=\mergededgewidth] (##1,##2) -- (##3,##4);
    }
    \mergededges
    \ifshowselectedmerge
      \renewcommand{\selectededge}[4]{%
        \draw[\sitecolor{4}!75!black, line width=\selectededgewidth]
          (##1,##2) -- (##3,##4);
      }
      \selectededges
    \fi
  \end{scope}
}

\newcommand{\drawcoarsefills}{%
  \foreach[count=\k] \x/\y/\cidx in \coarsesites {
    \begin{scope}
      \clip (0,0) circle (\r);
      \foreach[count=\m] \xx/\yy/\ccidx in \coarsesites {
        \ifnum\k=\m\else
          \pgfmathsetmacro{\mx}{(\x+\xx)/2}
          \pgfmathsetmacro{\my}{(\y+\yy)/2}
          \pgfmathsetmacro{\nx}{\xx-\x}
          \pgfmathsetmacro{\ny}{\yy-\y}
          \pgfmathsetmacro{\tx}{-\ny}
          \pgfmathsetmacro{\ty}{\nx}
          \clip
            ({\mx+\M*\tx},{\my+\M*\ty}) --
            ({\mx-\M*\tx},{\my-\M*\ty}) --
            ({\mx-\M*\tx-\M*\nx},{\my-\M*\ty-\M*\ny}) --
            ({\mx+\M*\tx-\M*\nx},{\my+\M*\ty-\M*\ny}) -- cycle;
        \fi
      }
      \ifusecolor
        \fill[\sitecolor{\cidx}!18] (-\M,-\M) rectangle (\M,\M);
      \else
        \fill[black!8] (-\M,-\M) rectangle (\M,\M);
      \fi
    \end{scope}
  }
}

\newcommand{\drawcoarseedges}[1]{%
  \foreach[count=\k] \x/\y/\cidx in \coarsesites {
    \foreach[count=\m] \xx/\yy/\ccidx in \coarsesites {
      \ifnum\k<\m
        \begin{scope}
          \clip (0,0) circle (\r);
          \pgfmathsetmacro{\mx}{(\x+\xx)/2}
          \pgfmathsetmacro{\my}{(\y+\yy)/2}
          \pgfmathsetmacro{\dx}{\xx-\x}
          \pgfmathsetmacro{\dy}{\yy-\y}
          \pgfmathsetmacro{\tx}{-\dy}
          \pgfmathsetmacro{\ty}{\dx}
          \foreach[count=\n] \xxx/\yyy/\cccidx in \coarsesites {
            \ifnum\n=\k\else
            \ifnum\n=\m\else
              \pgfmathsetmacro{\mxa}{(\x+\xxx)/2}
              \pgfmathsetmacro{\mya}{(\y+\yyy)/2}
              \pgfmathsetmacro{\nxa}{\xxx-\x}
              \pgfmathsetmacro{\nya}{\yyy-\y}
              \pgfmathsetmacro{\txa}{-\nya}
              \pgfmathsetmacro{\tya}{\nxa}
              \clip
                ({\mxa+\M*\txa},{\mya+\M*\tya}) --
                ({\mxa-\M*\txa},{\mya-\M*\tya}) --
                ({\mxa-\M*\txa-\M*\nxa},{\mya-\M*\tya-\M*\nya}) --
                ({\mxa+\M*\txa-\M*\nxa},{\mya+\M*\tya-\M*\nya}) -- cycle;
              \pgfmathsetmacro{\mxb}{(\xx+\xxx)/2}
              \pgfmathsetmacro{\myb}{(\yy+\yyy)/2}
              \pgfmathsetmacro{\nxb}{\xxx-\xx}
              \pgfmathsetmacro{\nyb}{\yyy-\yy}
              \pgfmathsetmacro{\txb}{-\nyb}
              \pgfmathsetmacro{\tyb}{\nxb}
              \clip
                ({\mxb+\M*\txb},{\myb+\M*\tyb}) --
                ({\mxb-\M*\txb},{\myb-\M*\tyb}) --
                ({\mxb-\M*\txb-\M*\nxb},{\myb-\M*\tyb-\M*\nyb}) --
                ({\mxb+\M*\txb-\M*\nxb},{\myb+\M*\tyb-\M*\nyb}) -- cycle;
            \fi\fi
          }
          \draw[#1] ({\mx-\M*\tx},{\my-\M*\ty}) -- ({\mx+\M*\tx},{\my+\M*\ty});
        \end{scope}
      \fi
    }
  }
}

\title{From local giants to locality in long-range percolation}
\author{Yago Moreno Alonso and J\'ulia Komj\'athy}
\date{\today}
\begin{document}

\maketitle

\begin{abstract} 
    We prove the analogue of Schramm's locality conjecture for \LRP{} on transitive graphs of polynomial growth with $\alpha \in (0,2)$.
    In this setting, we also prove the joint continuity of the percolation probability $\theta$ with respect to three parameters: the underlying graph with respect to the local topology, the connectivity kernel, and the percolation parameter $\beta$ for all values of $\beta \in \R_+$, including the critical parameter $\beta_c$.
    We also prove a number of results related to the supercritical sharpness of \LRP{}: the long-range order decay of the distribution of finite clusters, the truncation problem, the anchored isoperimetric dimension and the transience of the infinite percolation cluster, and the smoothness of the percolation characters.
    We obtain these results from proving the local existence-and-uniqueness of the linear-sized (giant) cluster.
    As an immediate corollary of the local existence-and-uniqueness of the giant we obtain the law of large numbers, which answers a special case of a question of Nekrashevych and Pete \cite[Question 1.3]{nekrashevych_scale-invariant_2011}.
    The main technical contribution is the construction of a renormalisation scheme combining iteratively merged Voronoi tiles with scale-invariant nets, related to the scale-invariant groups of Benjamini.
\end{abstract}

{\footnotesize
    \hspace{1em}
    
    \hspace{1em}
    MSC2020 Class: 82B43, 20F65.

    \hspace{1em}
    E-mail: ymamorenoalonso@tudelft.nl, j.komjathy@tudelft.nl

    \hspace{1em}
    Address: Delft Institute of Applied Mathematics, Delft University of Technology.
}

\section{Introduction}
\label{sec:intro}
\subsection{Long-range percolation on transitive graphs of polynomial growth}\label{sec:intro-model-def}
Let $G$ be a transitive graph of polynomial growth equipped with a fixed origin $o$.
We write $B(r)$ for the ball of radius $r$ centred at the origin with respect to the graph metric $d_G$.
It is a consequence of the theorems of Gromov, Trofimov, Bass and Guivarc'h \cite{gromov_groups_1981,trofimov_graphs_1985,bass_degree_1972,guivarch_groupes_1970} that there exists a unique positive integer $d$, called the \textbf{dimension} of $G$ and denoted by $\dim(G)$, and positive reals $c_G,C_G$ such that 
\begin{equation}
\label{eq:polynomial_volume_growth}
    c_G 
    r^d 
    \leq 
    \# B(r)
    \leq 
    C_G 
    r^d
\end{equation}
for all $r \in \N$.
Let $J : V \times V \to \R_+$ be \textbf{transitive}, meaning that for any graph automorphism $\gamma$ of $G$ we have $J(\gamma(x), \gamma(y)) = J(x,y)$.
For $\beta \in \R_+=[0, \infty)$, \textbf{\LRP{}} on $G$ with kernel $J$ is the random graph with vertex set $V$ where we include an edge between the vertices $x,y \in V$ independently at random with probability $1 - \exp(-\beta J(x,y))$.
We are mostly interested in the case where $J(x,y) = \Theta(d_G(x,y)^{-d \alpha})$ for $\alpha > 0$, meaning that there exist positive reals $c_J,C_J$, and $R_J$ such that 
\begin{equation}
\label{eq:polynomial_kernel}
    c_J 
    d_G(x,y)^{-d \alpha}
    \leq 
    J(x,y) 
    \leq 
    C_J d_G(x,y)^{-d\alpha}
\end{equation}
for all $x,y$ with $d_G(x,y) \geq R_J$. 
We write $\P_{\beta}$ for the law of the resulting random graph and $\E_{\beta}$ for the expectation operator.
We call the connected components of this random graph \textbf{clusters}, and we write $K$ for the cluster of the origin. 
The model undergoes a phase transition at the \textbf{critical value}
\begin{equation}
    \label{eq:critical_parameter}
    \beta_c
    =
    \beta_c(G,J)
    =
    \inf
    \{
        \beta \geq 0 : \P_{\beta}(\# K = \infty) > 0 
    \}
\end{equation}
which satisfies $0 < \beta_c < \infty$ if $d \geq 2$ and $\alpha > 1$ or $d = 1$ and $1 < \alpha \leq 2$ \cite{schulman_long_1983,newman_one-dimensional_1986}.
When $\alpha > 1$ the expected degree of the origin 
\begin{equation}
    \E_{\beta} 
    \left[ 
        \deg(o) 
    \right]
    =
    \E_{\beta}
    \left[
        \# 
        \{
            y \in V
            :
            o 
            \sim 
            y
        \}
    \right]
\end{equation}
is finite and we say that the pair $(G,J)$ is \textbf{integrable}.
When $0 < \alpha \leq 1$ the expected degree of the origin is infinite, in which case $\beta_c = 0$ and we say that the pair $(G,J)$ is \textbf{non-integrable}.
In this paper we study the supercritical regime when $\beta > \beta_c$ for values of $\alpha \in (0,2)$, where the long-range nature of the model is most apparent.

\subsection{Local topology for \LRP{}}
\label{subsec:local_topology}
Let $(G_n)_{n \in \N}$ be a sequence of transitive graphs, and let $G$ be a transitive graph. 
The sequence $(G_n)_{n \in \N}$ \textbf{converges locally} to $G$ (in the sense of Benjamini and Schramm) if and only if, for each $r \in \N$, the balls of radius $r$ in $G_n$ and $G$ are isomorphic as rooted graphs for all sufficiently large $n$.
\textbf{Schramm's locality conjecture}, stated in \cite[Conjecture 1.2]{benjamini_is_2011} and recently solved in full generality by Easo and Hutchcroft \cite{easo_critical_2023}, states that if $(G_n)_{n \in \N}$ is a sequence of infinite transitive graphs converging locally to $G$ an infinite transitive graph and $\limsup_{n \to \infty} p_c(G_n) < 1$ then $p_c(G_n) \to p_c(G)$ as $n \to \infty$.
Let $(G_n,J_n)_{n \in \N}$ be a sequence of pairs of a transitive graph $G_n=(V_n, E_n)$ and a transitive kernel $J_n : V_n \times V_n \to \R_+$, and let $(G,J)$ be a pair of a transitive graph $G=(V,E)$ and a transitive kernel $J : V \times V \to \R_+$.
To state an analogue of Schramm's locality conjecture for \LRP{},
for $\epsilon > 0 $ we define the \textbf{coupling radius} $R_n = R_n(G_n,G,\epsilon)$ to be the largest integer such that both 
\begin{equation}
    \label{eq:coupling_radius_definition}
    B_{G_n}(2 R_n) 
    \cong 
    B_{G}(2 R_n) 
    \quad 
    \text{ and }
    \sum_{x,y \in B(R_n)} 
    \lvert 
        J_n(x,y)
        -
        J(x,y)
    \rvert
    \leq 
    \epsilon.
\end{equation}
When $B_{G_n}(2 R_n) \cong B_{G}(2 R_n)$, we drop the reference to the graph and we write $x \in B(2R_n)$ for both the vertex in $B_{G}(2R_n)$ and its isomorphic copy in $B_{G_n}(2R_n)$.
Further, the graph metrics $d_{G_n}$ and $d_G$ agree on $B(R_n)$, in the sense that $d_{G_n}(x,y) = d_G(x,y)$ for all $x,y \in B(R_n)$.
We say that $(G_n,J_n)_{n \in \N}$ \textbf{converges locally} to $(G,J)$ if for all $\epsilon > 0$ the coupling radius $R_n(G_n,G,\epsilon)$ tends to infinity as $n$ tends to infinity.
This amounts to requiring that the sequence of graphs $(G_n)_{n \in \N}$ converges locally to $G$ and that the kernels $(J_n)_{n \in \N}$ converge to $J$ in $L^1$ inside $B(R_n)$. 
We say that $(J_n)_{n \in \N}$ are \textbf{uniformly integrable} if there exists $M > 0$ such that
\begin{equation}
\label{eq:tight-decay-Jn}
    \sum_{x \in V_n \setminus \{o_n\}} 
    J_n(o_n,x) 
    \leq 
    M
\end{equation}
for all $n$ sufficiently large.
\thref{lem:kernel_example} shows that the definition of local convergence and the assumption of uniform integrability are satisfied for a natural class of examples.
\subsection{Main results}
Our first result is the 
locality of \LRP{} with $\alpha \in (0,2)$ on 
transitive graphs of polynomial growth.

\begin{theorem}[Locality]
\thlabel{thm:lrp_locality}
    Let $G$ be a transitive graph of polynomial growth with $d \geq 2$, and suppose that $J : V \times V \to \R_+$ is a transitive kernel satisfying $J(x,y) = \Omega(d_G(x,y)^{-d \alpha})$ with $\alpha \in (0, 2)$. 
    Let $(G_n,J_n)_{n \in \N}$ be a sequence of transitive graphs and kernels converging locally to $(G,J)$ with $\liminf_{n \to \infty} \dim(G_n) \geq 2$.
    If $\alpha \leq 1$, or if $\alpha > 1$ and $(J_n)_{n \in \N}$ is uniformly integrable, then 
    \begin{equation}
        \lim_{n \to \infty}
        \beta_c(G_n,J_n) 
        =
        \beta_c(G,J).
    \end{equation}
\end{theorem}

This is somewhat surprising since we might expect long-range edges to have a greater impact on the geometry of the infinite cluster. 
For instance, in this setting (and on $\Z^d$) Biskup \cite{biskup_scaling_2004} proves that the chemical distance is polylogarithmic in the Euclidean distance.
\thref{thm:lrp_locality} may be intuitively understood as a consequence of proving the truncation problem for \LRP{} on transitive graphs of polynomial growth, which is the content of the next result.

\begin{theorem}[Truncation]
    \thlabel{thm:truncation_problem}
    Let $G$ be a transitive graph of polynomial growth with dimension $d \geq 2$ and suppose that $J : V \times V \to \R_+$ is a transitive kernel satisfying $J(x,y) = \Omega (d_G(x,y)^{-d \alpha})$ with $\alpha \in (1,2)$. 
    Let $\beta > \beta_c(J)$. 
    Then there exists $N=N(\beta) \in \N$ such that the truncated kernel 
    \begin{equation}
        \label{eq:truncated_kernel}
        \widetilde{J} 
        = 
        J(x,y) 
        \charf(
            d_G(x,y) \leq N
        )
    \end{equation}
    satisfies $ \beta_c(\widetilde{J})<\beta$.
\end{theorem}

Our results for the supercritical phase imply the continuity of the phase transition for a fixed $G$.

\begin{theorem}[Continuity of $\theta$]
\thlabel{thm:continuity_percolation_density}
    Let $G$ be a transitive graph of polynomial growth with $d \geq 1$ and suppose that $J : V \times V \to \R_+$ is a transitive kernel satisfying $J(x,y) = \Omega (d_G(x,y)^{-d \alpha})$ with $\alpha \in (1,2)$. 
    Then the set $\{\beta \in \R_+ : \theta(\beta) > 0 \}$ is open. 
    In particular, $\theta(\beta_c) = 0$.
\end{theorem}

For \LRP{} on $\Z^d$ this result is due to Berger \cite{berger_transience_2002}.
For \LRP{} on transitive \textit{unimodular} graphs, which include transitive graphs of polynomial growth, a result of Hutchcroft \cite[Theorem 1.2]{hutchcroft_power-law_2021} implies that there is no percolation at $\beta_c$.
Our next result is the \textit{joint} continuity of the percolation density $\theta$ with respect to three parameters: the underlying graph with respect to the local topology, the connectivity kernel, and the percolation parameter $\beta$ for all values of $\beta \in \R_+$, including the critical parameter $\beta_c$.

\begin{theorem}[Joint continuity of $\theta$]
\thlabel{thm:theta_locality}
    Let $G$ be a transitive graph of polynomial growth with $d \geq 2$ and suppose that $J : V \times V \to \R_+$ is a transitive kernel satisfying $J(x,y) = \Theta(d_G(x,y)^{-d\alpha})$ with $\alpha \in (0, 2)$. 
    Let $\beta \in \R_+$ and let $\beta_n \to \beta$. 
    Let $(G_n,J_n)_{n \in \N}$ be a sequence of transitive graphs and kernels converging locally to $(G,J)$ with $\liminf_{n \to \infty} \dim(G_n) \geq 2$ and $(J_n)_{n \in \N}$ uniformly integrable.
    If either $\alpha \leq 1$ and $\beta>0$, or if $\alpha > 1$, then
    \begin{equation}
        \label{eq:full-continuity}
        \lim_{n\to \infty}
        \theta(\beta_n,G_n,J_n)
        = 
        \theta(\beta,G,J).
    \end{equation}
\end{theorem}

The techniques we develop are also suitable to obtain finer quantitative results about supercritical \LRP{}.
Our next result, which was known for $\Z^d$ and more general inhomogeneous random graphs \cite{jorritsma_cluster-size_2025} and is new for arbitrary transitive graphs of polynomial growth, is the stretched-exponential decay of the distribution of finite clusters. 

\begin{theorem}[Cluster-size decay]
\thlabel{thm:cluster_size_decay}
    Let $G$ be a transitive graph of polynomial growth with $d \geq 1$ and suppose that $J : V \times V \to \R_+$ is a transitive kernel satisfying $J(x,y) = \Omega(d_G(x,y)^{-d \alpha})$ with $\alpha \in (0,2)$. 
    Let $\beta > \beta_c$.
    Then there exists $A = A(G,J) > 0$ such that for all $k \in \N$
    \begin{equation}
        \label{eq:cluster_size_decay}
        \P_{\beta}
        (
            k
            \leq
            \# \cluster 
            < 
            \infty
        )
        \leq
        \exp
        (
            -\beta k^{\min(2 - \alpha,1)}
            / 
            A
        ).
    \end{equation}
\end{theorem}

We also prove the following lower bound on the distribution of finite clusters.

\begin{theorem}
\thlabel{thm:cluster_size_lower_bound}
    Let $G$ be a transitive graph of polynomial growth with $d \geq 1$, and suppose that $J : V \times V \to \R_+$ is a transitive kernel satisfying $J(x,y) = \Theta(d_G(x,y)^{-d \alpha})$ with $\alpha \in (1,2)$.
    Let $\beta > \beta_c$.
    %
    %
    Then there exists $A = A(G,J) > 0$ such that for all $k \in \N$
    \begin{equation}
    \label{eq:cluster-lower-bound-alt}
        \exp
        (
            -  
            A 
            \beta k^{
                \max(2-\alpha, (d-1)/d)
            }(\log k)^{
                \charf
                \left(
                    \alpha=1+1/d
                \right)}
        )
        \leq 
        \P_{\beta}
        \left(
            k 
            < 
            \# K 
            < 
            \infty
        \right).
    \end{equation}
\end{theorem} 

The lower bound in \thref{thm:cluster_size_lower_bound} matches the upper bound in \thref{thm:cluster_size_decay} when $\alpha \in (1,1+1/d)$. 
For $\alpha > 1 + 1/d$, we prove the matching upper bound in \cite{alonso_supercritical_2026} under the assumption that $\beta>\beta_c$ is sufficiently large. 
\thref{thm:cluster_size_lower_bound} improves on an earlier result \cite[Theorem 1.4]{alonso_supercritical_2026} which contained a correction term in the exponent.
and required $\beta > \beta_c$ sufficiently large. 

Our next results determine the anchored isoperimetric dimension of the infinite cluster in \LRP{}. 
We say that a locally finite connected graph $G = (V,E)$ satisfies an \textbf{anchored $d$-dimensional isoperimetric inequality} if for some (and hence every) vertex $v \in V$ there exists $c = c(v) > 0$ such that $\# \partial_{E} W \geq c \left(\# W \right)^{(d-1)/d}$ for every finite connected set of vertices $W \subset V$ that contains $v$, where $\partial_E W$ denotes the edge boundary of $W$ in $G$.
The \textbf{anchored isoperimetric dimension} of $G$ is defined to be the supremal value of $d$ for which $G$ satisfies an anchored $d$-dimensional isoperimetric inequality.

\begin{corollary}
    \thlabel{cor:anchored_isop_dim}
    Let $G$ be a transitive graph of polynomial growth with $d \geq 1$, and suppose that $J: V \times V \to \R_+$ is a transitive kernel with $J(x,y) = \Omega(d_G(x,y)^{-d\alpha})$ with $\alpha \in (1,2)$.
    Let $\beta > \beta_c$. 
    Then the infinite cluster $\cluster_{\infty}$ has anchored isoperimetric dimension at least $1/(\alpha-1)$ almost surely.
\end{corollary}
 
\thref{cor:anchored_isop_dim} follows from \thref{thm:cluster_size_decay} and \cite[Theorem 1.5]{alonso_supercritical_2026}, which in turn extends arguments of Pete and Hutchcroft \cite{pete_note_2008,hutchcroft_transience_2023} to \LRP{}.
Whereas in \BNNP{} the anchored isoperimetric dimension is at most the dimension of the underlying graph, this is not the case for \LRP{}.
With the results of this paper, we prove a matching \textit{upper} bound on the anchored isoperimetric dimension, which is new for \LRP{}.

\begin{theorem}
    \thlabel{thm:anch_dim_upper}
    Let $G$ be a transitive graph of polynomial growth with $d \geq 1$, and suppose that $J: V \times V \to \R_+$ is a transitive kernel with $J(x,y) = \Theta(d_G(x,y)^{-d\alpha})$ with $\alpha \in (1,2)$.
    Let $\beta>\beta_c$.
    Then the infinite cluster $K_\infty$ has anchored isoperimetric dimension at most $\max(1/(\alpha-1), d)$ almost surely. 
\end{theorem}

Our next result is that several percolation characters, such as the percolation probability $\theta(\beta) = \P_{\beta}(o \leftrightarrow \infty)$, the truncated susceptibility $\chi^f = \E_{\beta} [ \# K \charf(\# K < \infty)]$, the free energy $\kappa(\beta) = \E_{\beta} [ (\# K)^{-1}]$, or the truncated two point function $\tau_{\beta}^f(u,v) = \P_{\beta}(u \leftrightarrow v, \# K < \infty)$, are all smooth functions of $\beta$ on the interval $(\beta_c,\infty)$.
This is new for \LRP{}.

\begin{theorem}
\thlabel{cor:characters_smoothness}
    Let $G$ be a transitive graph of polynomial growth with $d \geq 1$, and suppose that $J : V \times V \to \R_+$ is a transitive kernel satisfying $J(x,y) = \Omega (d_G(x,y)^{-d \alpha})$ with $\alpha \in (1,2)$.
    Let $F$ be a function of $K$ so that $|F(K)|$ can be bounded from above by a polynomial of $\# \cluster$.
    Then $\E_{\beta} \left[ F(\cluster) \charf(\# \cluster < \infty) \right]$ is a smooth function of $\beta$ for $\beta \in (\beta_c,\infty)$.
\end{theorem}

Most of the above results crucially rely on the local existence-and-uniqueness of the linear-sized (\textbf{giant}) component, which is one of the main contribution of this paper.
We write $\giantone{r}$ and $\gianttwo{r}$ for the largest and second-largest clusters in $B(r)$, respectively. 

\begin{theorem}[Existence of the giant]
\thlabel{prop:biskup}
    Let $G$ be a transitive graph of polynomial growth with $d \geq 1$, and suppose that $J : V \times V \to \R_+$ is a transitive kernel satisfying $J(x,y) = \Omega(d_G(x,y)^{-d \alpha})$ with $\alpha \in (0,2)$. 
    Let $\beta > \beta_c$.
    Then there exist $\nu = \nu(G, J) > 0$ and $A = A(G,J) > 0$ such that for all $r \in \N$
    \begin{equation}
    \label{eq:linear_giant_11}
        \P_{\beta}
        \left(
            \# \giantone{r}
            \leq 
            \nu
            \# B(r)
        \right) 
        \leq 
        \exp(
            -
           \beta 
           (\# B(r))^{\min(2 - \alpha,1)}
           /
           A
        ).
    \end{equation}
\end{theorem}

\begin{theorem}[Uniqueness of the giant]
    \thlabel{prop:second_largest}
    Let $G$ be a transitive graph of polynomial growth with $d \geq 1$, and suppose that $J : V \times V \to \R_+$ is a transitive kernel satisfying $J(x,y) = \Omega(d_G(x,y)^{-d \alpha})$ with $\alpha \in (0,2)$. 
    Let $\beta > \beta_c$.
    Then there exists $A=A(G,J) > 0$ such that for all $r \in \N$ and $k \in \N$
    \begin{equation}\label{eq:second-largest-in-theorem}
        \P_{\beta}
        \left( 
            \# \gianttwo{r}
            > 
            k
        \right)
        \leq 
        \frac{\#B(r)}{k}
        \exp
        \left(
            -
            \beta 
            k^{\min(2-\alpha,1)}/A
        \right).
    \end{equation} 
    In particular, $ \# \gianttwo{r}=O_{\P}((\log r)^{\max(1, 1/(2-\alpha))})$.
\end{theorem}
Note that \thref{prop:biskup,prop:second_largest} allow for $\alpha \in (0,1]$; we use this to prove the non-integrable part of \thref{thm:lrp_locality}.
One corollary to \thref{prop:biskup,prop:second_largest} is the following result which may be understood as saying that the giant is a local witness to the percolation density.
This is one of the two main inputs to prove \thref{thm:theta_locality}.

\begin{corollary}
    \thlabel{prop:o_in_kinfty}
    Let $G$ be a transitive graph of polynomial growth with $d \geq 1$ and suppose that $J : V \times V \to \R_+$ is a transitive kernel satisfying $J(x,y) = \Omega(d_G(x,y)^{-d \alpha})$ with $\alpha \in (0,2)$.
    Let $\beta > \beta_c$.
    Then there exists $c = c(G,J) > 0$ such that 
    \begin{equation}
        \P_{\beta}
        \left(
            o 
            \not\in 
            \giantone{r},
            o 
            \in 
            \kinfty
        \right)
        \leq 
        \frac{1}{(\#B(r))^{c}}
    \end{equation}
    and in particular, for all $r \in \N$
    \begin{equation}
        \P_{\beta}
        \left(
            o 
            \in 
            \giantone{r}
        \right)
        \geq 
        \theta(\beta) 
        -
        \frac{1}{(\#B(r))^{c}}.
    \end{equation}
\end{corollary}

As a second corollary to \thref{prop:biskup,prop:second_largest}, we verify the conditions of \cite[Theorem 2.2]{van_der_hofstad_giant_2023} to obtain the law of large numbers for \LRP{}.

\begin{corollary}[Law of large numbers]
    \thlabel{thm:lln}
    Let $G$ be a transitive graph of polynomial growth with $d \geq 1$, and suppose that $J : V \times V \to \R_+$ is a transitive kernel satisfying $J(x,y) = \Omega(d_G(x,y)^{-d \alpha})$ with $\alpha \in (0,2)$. 
    Let $\beta > \beta_c$.
    Then as $r\to \infty$
    \begin{equation}
        \frac{\# \giantone{r}}{\# B(r)}
        \overset{\P}{\longrightarrow}
        \theta(\beta, G, J)
        \quad
        \text{ and }
        \quad
        \frac{\# \gianttwo{r}}{\# B(r)}
        \overset{\P}{\longrightarrow}
        0.
    \end{equation}
\end{corollary}

Nekrashevych and Pete ask \cite[Question 1.3]{nekrashevych_scale-invariant_2011} whether for any amenable transitive graph there exists a connected F{\o}lner sequence exhausting the graph such that the law of large numbers holds along the sequence.
\thref{thm:lln} gives a positive answer to this question in the setting of \LRP{} with $\alpha \in (1,2)$ on transitive graphs of polynomial growth.

Our last results concerns the random walk on the infinite cluster $\kinfty$.
We say that $\kinfty$ is transient if the simple random walk on it is transient, and recurrent otherwise.

\begin{theorem}
\thlabel{thm:random_walks}
    Let $G$ be a transitive graph of polynomial growth with $d \geq 1$, let $J : V \times V \to \R_+$ be transitive and suppose that $J(x,y) = \Theta (d_G(x,y)^{-d \alpha})$ with $\alpha \in (1,2)$. Let $\beta > \beta_{c}$.
    Then the infinite cluster $\kinfty$ is transient almost surely.
\end{theorem}

\begin{theorem}
    \thlabel{thm:two_dim_recurrence}
    Let $G$ be a transitive graph of polynomial growth with $d = 2$, let $J : V \times V \to \R_+$ be transitive and integrable, and suppose that $J(x,y) = O (d_G(x,y)^{-d \alpha})$ with $\alpha \geq 2$. 
    Let $\beta > \beta_c$. 
    Then the infinite cluster $\kinfty$ is recurrent almost surely.
\end{theorem}

\thref{thm:random_walks,thm:two_dim_recurrence} extend results of Berger \cite{berger_transience_2002} to transitive graphs of polynomial growth.
For $\alpha\in (1, 3/2)$, \thref{cor:anchored_isop_dim} provides a new proof of \thref{thm:random_walks}: in this case the infinite cluster $\kinfty$ has anchored isoperimetric dimension at least $2$ and transience follows from Thomassen's criterion \cite{thomassen_isoperimetric_1992}.
To prove transience in the entire parameter range $\alpha \in (1,2)$, we still require a renormalisation argument.
The proof of \thref{thm:two_dim_recurrence} uses different techniques from the rest of the paper and relies on a mass-transport argument involving random electrical networks.

\subsection{Discussion}
\subsubsection{Locality}
\paragraph{Previous locality results.}
Schramm's locality conjecture was established in a number of special cases \cite{benjamini_is_2011,martineau_locality_2017,hutchcroft_locality_2020,hutchcroft_nonuniqueness_2020,hermon_no_2021,contreras_locality_2023,contreras_supercritical_2024}, and has recently been solved in full generality by Easo and Hutchcroft \cite{easo_critical_2023}.
A historical account of the locality conjecture is given in \cite[Section 1.1]{easo_critical_2023}.
In the setting of long-range percolation, a type of locality result is when the graph is fixed and only the kernels vary.
B\"aumler \cite{baumler_continuity_2026} proves that for \LRP{} on $\Z^d$ and any kernel $J$ which is known to satisfy the truncation problem, if $(J_n)_{n \in \N}$ is a sequence of kernels converging to $J$ in $L^1$ of $\Z^d$, then $\beta_c(J_n) \to \beta_c(J)$ as $n \to \infty$. 
The analogous result in the setting of a \textit{unimodular} transitive graphs should follow from \cite{duminil-copin_new_2016,hutchcroft_locality_2020,hutchcroft_power-law_2021}, as remarked in \cite[Remark 1.3]{easo_critical_2023}.
The truncation problem was first considered by Meester and Steif \cite{meester_continuity_1996} and Berger \cite{berger_transience_2002}.
It has also been considered by B{\"a}umler for $\alpha > 2$ \cite{baumler_continuity_2026} and for non-integrable kernels \cite{baumler_truncation_2024}. 
\thref{thm:lrp_locality} proves the analogue of Schramm's locality conjecture for \LRP{} on transitive graphs of polynomial growth with $\alpha \in (0,2)$.

\paragraph{The role of dimension in locality.}
In \thref{thm:lrp_locality}, we replaced the assumption from Schramm's locality conjecture that $\limsup_{n\to\infty} p_c(G_n) < 1$ with the assumption that $\liminf_{n \to \infty} \dim(G_n) \geq 2$.
For \BNNP{}, these conditions are equivalent: for graphs of polynomial growth this follows from the theorems of Gromov and Trofimov
(see for instance \cite[pp.~3542--3543]{duminil-copin_existence_2020}), and for arbitrary transitive graphs this is due to \cite{duminil-copin_existence_2020} (see also \cite{easo_counting_2025}). 
The analogous equivalence is not true for \LRP{}.
Indeed, one-dimensional \LRP{} has a non-trivial phase transition precisely when $\alpha \in (1,2] $ \cite{newman_one-dimensional_1986}, so $\liminf_{n\to\infty} \beta_c(G_n,J_n) < \infty$ does not imply $\limsup_{n \to \infty} \dim(G_n) \geq 2$.
Further, Aizenman and Newman \cite{aizenman_discontinuity_1986} {\color{blue}(see also \cite{duminil-copin_long-range_2024})} prove that in one dimension the phase transition is \textit{discontinuous} when $\alpha = 2$.
For this and other reasons, \LRP{} is \textit{not} local in one dimension, neither when $\alpha \in (1,2)$ nor when $\alpha = 2$.
We study the failure of locality in one dimension in our upcoming paper \cite{moreno-alonso_long-range_2026}.
The assumption that $\liminf_{n \to \infty} \dim(G_n) \geq 2$ is thus both necessary and sufficient for locality.

\paragraph{The role of uniform integrability in locality.}
\thref{thm:lrp_locality} establishes the locality of $\beta_c$ for both $\alpha \in (1,2)$ and $\alpha \in (0,1]$.
For $\alpha \in (1,2)$, the limit pair $(G,J)$ is integrable, and we assume that the kernels $(J_n)_{n \in \N}$ are uniformly integrable in the sense of \eqref{eq:tight-decay-Jn}.
This information is not \textit{local}, as it also involves edges to vertices outside of the coupling radius, and hence it cannot be read off from $(G,J)$ alone. 
The proof of the lower semi-continuity of $\beta_c$ uses the ``$\phi_{\beta}(S)$'' argument of \cite{duminil-copin_new_2016}, and it will be apparent that uniform integrability is a necessary assumption and that obvious counterexamples arise without it, see \thref{ex:no_unif_integrability}.
For $\alpha \in (0,1]$, the pair $(G,J)$ is non-integrable so that $\beta_c(G,J) = 0$ and locality holds without further assumptions.
Non-integrable pairs naturally arise as local limits of integrable pairs, see \thref{lem:non_integrable_limits}. 

\paragraph{Plentiful bulk-to-bulk connections for locality.}
Long range percolation in $d$ dimensional spaces with $\alpha \in (0,2)$ behaves qualitatively long-range. 
We make use of this fact in locality related proofs via \emph{bulk-to-bulk connections}.  
To illustrate this, let $A$ and $B$ be two sets of vertices, each containing order $r^d$ many vertices, with pairwise distance at most order $r$.
Then the expected number of edges from $A$ to $B$ and the probability of a connection between $A, B$ satisfies 
\begin{align}
   \label{eq:bulk-to-bulk-expectation}
   &\E_{\beta}
   \left[
       \# 
       \{
           (x,y) \in A \times B
           : 
           x 
           \sim
           y
       \}
   \right]
   = \Omega(
   r^{d(2-\alpha)}),\\
   & \P_{\beta}
   \left(
       A 
       \not\sim
       B
   \right)
   =
   \exp
   \left(
       - 
       \beta
       J(A,B)
   \right)
   =
   \exp
   \big(
       -\Theta(
               \# A 
               \# B/
          r^{d \alpha})
            \big)
   =
    \exp
    \big(
       -
      \Omega(r^{d(2 - \alpha)})
    \big).\label{eq:bulk-to-bulk-P}
\end{align}
For $\alpha \in (0,2)$ the connection ensured by \eqref{eq:bulk-to-bulk-P} together with the polynomial growth \eqref{eq:polynomial_volume_growth} makes it possible to develop multiscale renormalisation in transitive graphs of polynomial growth, since bulk-to-bulk connections, as opposed to nearest neighbour percolation paths, are less sensitive to the local geometry of the underlying space.

\subsubsection{Sharp phase transitions in percolation}
\paragraph{Supercritical sharpness.} 
Supercritical sharpness generally refers to the fast decay of tail events involving finite clusters, such as the truncated one-arm event or the tail distribution of finite clusters, in the \textit{entire} supercritical regime. 
For \BNNP{} on $\Z^d$, supercritical sharpness was proved by \cite{kesten_critical_1980} for $d=2$ and \cite{chayes_bernoulli_1987, grimmett_supercritical_1990} for $d \geq 3$. 
Hermon and Hutchcroft \cite{hermon_supercritical_2021} develop a
theory of supercritical percolation on transitive non-amenable graphs. 
Relevant to our work, an entirely new proof of supercritical sharpness for transitive graphs of polynomial growth is obtained by Contreras, Martineau, and Tassion \cite{contreras_supercritical_2024}. 
These results are made \textit{local} by Martineau and Panagiotis \cite{martineau_percolation_2025}.
Supercritical sharpness has also been proven for kernel-based spatial random graphs \cite{jorritsma_cluster-size_2025,jorritsma_large_2025}, Voronoi percolation \cite{dembin_supercritical_2025}, and for the $\varphi^4$ model \cite{gunaratnam_supercritical_2025}.
Recently, supercritical sharpness has been proven for \BNNP{} on arbitrary transitive graphs \cite{diskin_supercritical_2026}.
The results of this paper and our recent paper \cite{alonso_supercritical_2026} prove supercritical sharpness for \LRP{} on transitive graphs of polynomial growth with $\alpha \in (1,2)$: 
\thref{thm:cluster_size_decay} proves the stretched-exponential cluster-size decay, \thref{cor:anchored_isop_dim,thm:anch_dim_upper} identify the anchored isoperimetric dimension of the infinite cluster, \thref{thm:random_walks} proves the transience of the infinite cluster, and \cite[Theorem 1.8]{alonso_supercritical_2026} proves the polynomial decay of the truncated one-arm event.

\paragraph{Long-range percolation beyond $\Z^d$.}
The sharpness of the phase transition for \LRP{} was proved for transitive graphs of subexponential growth by Aizenman and Barsky \cite{aizenman_discontinuity_1986}.
This is extended to \LRP{} on arbitrary transitive graphs in \cite{antunovic_sharpness_2008,duminil-copin_new_2016} and to the random cluster model on weighted transitive graphs in \cite{hutchcroft_new_2020}.
Gandolfi, Keane, and Newman's \cite{gandolfi_uniqueness_1992} proved uniqueness of the infinite cluster for transitive graphs of subexponential growth.
Long-range percolation has previously been considered on finite cycles \cite{benjamini_diameter_2001,benjamini_long-range_2008}, wedges of $\Z^3$ \cite{berger_transience_2006}, weighted transitive graphs \cite{hutchcroft_power-law_2021,hutchcroft_derivation_2022}, and on the hierarchical lattice \cite{koval_long-range_2012,georgakopoulos_percolation_2020,hutchcroft_critical_2024,hutchcroft_critical_2025-3,baumler_discontinuous_2025}.
The related model of spread-out percolation has been considered on transitive graphs of polynomial growth \cite{spanos_spread-out_2024}.
Supercritical inhomogeneous long-range models have also been studied in \cite{jorritsma_cluster-size_2025,jorritsma_large_2025}.
In \cite{alonso_supercritical_2026} we consider \LRP{} on transitive graphs of polynomial growth with $\alpha > 1 + 1/d$, and in \cite{moreno-alonso_long-range_2026} we consider \LRP{} on transitive graphs of \textit{linear} growth with $\alpha \in (1,2]$.

\subsubsection{Scale-invariant groups and structure theory}
\label{sec:antal}

\paragraph{Scale-invariant groups and the uniqueness of the giant.}
Renormalisation arguments are a cornerstone of supercritical percolation theory, and are often driven by the notion of a \textit{crossing cluster}.
The classical Antal–Pisztora renormalisation lemma \cite{antal_chemical_1996} (see also \cite[Section 7.4]{grimmett_percolation_1999}) roughly says that for \BNNP{} on $\Z^d$, $d \geq 2$, for every $p>p_c$ and $\epsilon>0$ , there exists $n$ sufficiently large such that, with probability at least $1-\epsilon$, a box of side-length $n$ is ``good’’: it contains a cluster connecting all $2d$ faces, while all other clusters in the box have diameter at most $\epsilon n$.
Such a lemma gives a qualitative local existence-and-uniqueness statement for a macroscopic component, and serves as a building block for renormalisation.
Indeed, if two neighbouring boxes of side-length $n$ are enlarged from their centres by a factor $5/4$, then, provided both enlarged boxes are good and $\epsilon<5/8$, their crossing clusters must connect in the union of the two boxes. 
A robust analogue on transitive graphs of polynomial growth is the local existence-and-uniqueness of \textit{arms} of \cite[Proposition 1.3]{contreras_supercritical_2024}.
This type of procedure can then be bootstrapped to obtain many statements about supercritical percolation, but one thing it does not give is the existence of a linear-sized \textbf{giant} component.
In a different direction, Benjamini \cite[Chapter 1.8]{benjamini_coarse_2013} introduced scale-invariant groups as a possible structure for renormalisation arguments.
These groups are further investigated by Nekrashevych and Pete \cite{nekrashevych_scale-invariant_2011} and remain an active area of research \cite{bartholdi_self-similar_2020, hurder_cantor_2021,dere_strongly_2022,grigorchuk_liftable_2025,limbeek_structure_2021,wardell_elementary_2025,dere_automorphism_2026} with many open questions \cite[Chapter 9]{sapir_group_2007} in geometric group theory.
The goal of these arguments, as stated in \cite[Question 1.3]{nekrashevych_scale-invariant_2011}, is to establish a law of large numbers for percolation, which is essentially equivalent to local existence-and-uniqueness of the giant component and is related to uniqueness of the giant in finite transitive graphs \cite{alon_percolation_2004}.
In our setting of \LRP{} on transitive graphs of polynomial growth with $\alpha \in (0,2)$, we replace the crossing-cluster mechanism by the \textit{plentiful bulk-to-bulk} connections of the previous section.
These connections are less sensitive to local geometry and, combined with coarse graining, allow us to run a multi-scale renormalisation argument yielding the local existence-and-uniqueness of the giant in \thref{prop:biskup,prop:second_largest} and the law of large numbers in \thref{thm:lln}.

\paragraph{Structure theory of transitive graphs.}
To prove locality, we use the structure theory of transitive graphs of polynomial growth.
A consequence of Gromov’s theorem is that a finitely generated group of polynomial growth satisfies the following alternative: it is either virtually $\Z$, or contains a subgroup isomorphic to $\Z^2$ (see \cite[Theorem 7.18]{lyons_probability_2016}).
Trofimov’s theorem gives the analogous alternative for transitive graphs of polynomial growth: either it is quasi-isometric to $\Z$, or to a graph containing $\Z^2$ as a subgraph.
This classical structure theory is packaged in \thref{prop:fixed_unif_control_nets}.
Recently, an extensive structure theory was developed in \cite{breuillard_structure_2012} and \cite{tessera_balls_2024} to obtain \textit{finitary} versions of Gromov’s and Trofimov’s theorems.
We refer to \cite{breuillard_brief_2014,tointon_brief_2020,easo_uniform_2025} for surveys.
This finitary structure theory is packaged in \thref{prop:unif_control_nets}, and is part of the strategy introduced by \cite{contreras_locality_2023} to prove locality for \BNNP{} on transitive graphs of polynomial growth.

\subsection{Our new proof techniques}
The main contribution of our paper is to develop a robust multi-scale renormalisation scheme for \LRP{} on transitive graphs of polynomial growth which allows for a complete understanding of the supercritical regime.
Another challenge inherent to long-range percolation is that estimates for long-range events depend on volume bounds for spheres in transitive graphs of polynomial growth, and it remains an open problem to prove sharp bounds for these. 
We overcome this issue through the \textit{sphere calculus}.

\paragraph{Coarse-graining via scale-invariant nets.}
The first part of our scheme is a coarse-graining of the space combining iteratively merged Voronoi tiles with scale-invariant nets. 
A net is a \textit{separated} subset of vertices of the graph $G$ which is also \textit{dense} in the graph.
The parameter regulating the separation and the density is the \textit{scale} of the net.
For a sequence of increasing scales, we consider the Voronoi tiling associated to a net of each scale.
Voronoi tilings at different scales are not necessarily nesting, and instead we inductively redefine a level-$i+1$ tile to be the \textit{union} of those level-$i$ tiles whose centre is contained in the level-$i+1$ Voronoi tile.
This iteratively defined construction gives a partition of the space into increasingly large tiles with a nesting structure.
This procedure is illustrated in Figure \ref{fig:voronoi_iterative_merge}.
The price to pay for this nesting structure is that the tiles are no longer genuine Voronoi tiles, and their volumes fluctuate.
Crucial to the success of the renormalisation scheme, in \thref{prop:tile_and_cell_volume_bounds} we prove that for any sufficiently fast growing sequence of scales we can maintain uniform control over the fluctuations in volume of tiles and also over the number of subtiles in each tile.
This gives a robust alternative to the scale-invariant box structure available in $\Z^d$.

\begin{figure}
    \centering
    \begin{subfigure}[t]{0.31\textwidth}
        \centering
        \resizebox{\linewidth}{!}{\begin{tikzpicture}[scale=0.68]
  \drawball
  \drawfinecells
  \draw[line width=1.05pt] (0,0) circle (\r);
\end{tikzpicture}}
        \caption{}
        \label{fig:voronoi_level_1}
    \end{subfigure}
    \hfill
    \begin{subfigure}[t]{0.31\textwidth}
        \centering
        \resizebox{\linewidth}{!}{\begin{tikzpicture}[scale=0.68]
  \drawball
  \drawcoarsefills
  \drawfinemesh{black!10, line width=0.12pt}
  \drawcoarseedges{black, line width=\coarseedgewidth}
  \drawcoarsesites
  \draw[line width=1.05pt] (0,0) circle (\r);
\end{tikzpicture}}
        \caption{}
        \label{fig:voronoi_level_2}
    \end{subfigure}
    \hfill
    \begin{subfigure}[t]{0.31\textwidth}
        \centering
        \resizebox{\linewidth}{!}{\begin{tikzpicture}[scale=0.68]
  \drawball
  \drawmergedcells
  \ifshowcoarseguide
    \drawcoarseedges{black!35, line width=0.45pt, densely dashed}
  \fi
  \drawmergededges
  \drawcoarsesites
  \drawfinesites
  \draw[line width=1.05pt] (0,0) circle (\r);
\end{tikzpicture}}
        \caption{}
        \label{fig:voronoi_merged}
    \end{subfigure}
    \caption{the sketches in (a) and (b) show a metric ball $B(r)$ with two Voronoi tilings at different scales, say $r_0$ and $r_1$ with $r_0 < r_1 < r$. 
    In (c), a tile at scale $r_1$ is redefined to be the union of those scale $r_0$ tiles whose centre is contained in the scale $r_1$ tile.}
    \label{fig:voronoi_iterative_merge}
\end{figure}
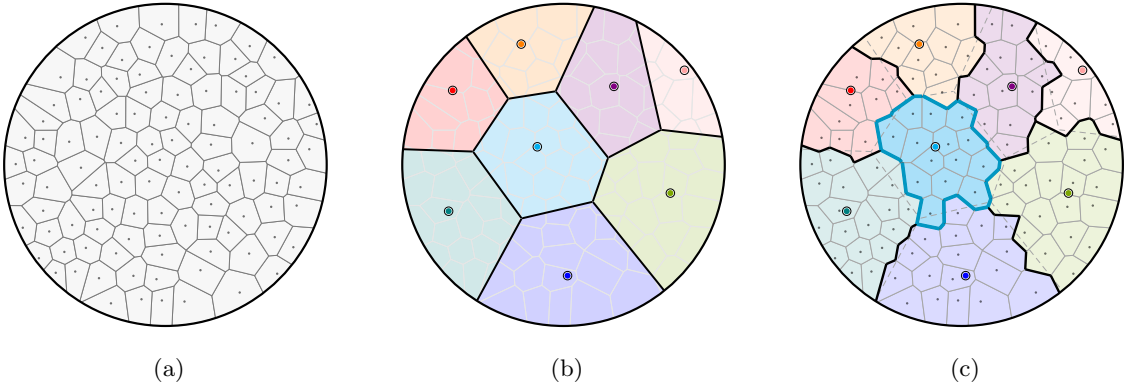

\paragraph{Local existence-and-uniqueness of the giant.}
Armed with the scale-invariant coarse-graining of the space, 
we run a multi-scale renormalisation argument to establish the
 local existence-and-uniqueness of the giant in \thref{prop:biskup,prop:second_largest}. 
We use ergodicity and the uniqueness of the infinite cluster at the base level for the existence of a local linear-sized cluster, and we use the \textit{plentiful bulk-to-bulk connections} in \eqref{eq:bulk-to-bulk-P} to iteratively construct a local linear-sized cluster at all scales by connecting nearby local giants.
Since the volume of the tiles fluctuates, one technical challenge is to maintain uniform control on the \textit{density} of the local giant built along the renormalisation argument, different from earlier similar arguments for $\Z^d$ \cite{biskup_scaling_2004}.
Given the local existence of the giant, we combine a revealment scheme with the \textit{plentiful bulk-to-bulk connections} to bound the size of small components and hence obtain the uniqueness of the local giant.
When compared to the discussion about crossing clusters in Section \ref{sec:antal}, we replace the boxing structure by the coarse-graining with nesting tiles, and to move between levels, we replace crossing clusters with the \textit{plentiful bulk-to-bulk connections}.
Note that crossing clusters are not the right notion for \LRP{} since any long edge can be a crossing cluster and long edges may jump over each other.

\paragraph{Natural coupling on the coupling radius.} 
The isomorphism between balls, the agreement of the graph metrics, and the local $L^1$ convergence of the kernels allows us to locally couple the \LRP{} configurations on $G_n$ and $G$ through the \textit{natural coupling on the coupling radius}, stated in \thref{lem:coupling-on-radius}.  
This can be thought of as taking a ``cookie cutter'' picture of the percolation configuration in the limit graph to couple it with those in the graph sequence.
As a consequence of the local existence of the giant we identify a \textit{finite-size criterion} for $\beta_c$ in \LRP{}, stated in \thref{prop:finite_size_criterion}. 
Together with the strcuture theory of transitive graphs in \thref{prop:unif_control_nets} and the ``$\phi_{\beta}(S)$ argument'' of Duminil-Copin and Tassion \cite{duminil-copin_new_2016} we obtain the locality of $\beta_c$.
With the stronger local existence-and-uniquenss of the giant we establish \thref{prop:o_in_kinfty}. 
In a similar spirit to the finite-size criterion, \thref{prop:o_in_kinfty} is a \emph{local witness} to the percolation density $\theta$ and through the natural coupling on the coupling radius we prove the joint continuity of $\theta$ in \thref{thm:theta_locality}.

\paragraph{Sphere calculus.}
Many computations in \LRP{} are closely related to the volume of spheres, typically in the form of weighted sums of spheres.
While in $\Z^d$ it is standard that $\# S(r) \asymp r^{d-1}$, it remains open to prove the analogous result for transitive graphs of polynomial growth.
If we used the currently available upper bounds \cite{colding_liouville_1998,tessera_volume_2007,breuillard_rate_2013,gianella_asymptotics_2017,bodart_intermediate_2025}, their error terms would enter our results (specifically
\thref{thm:cluster_size_lower_bound,thm:anch_dim_upper,thm:two_dim_recurrence,cor:anchored_isop_dim}).
To obtain sharp results, we avoid upper bounds for sphere volumes altogether and instead develop a \textit{sphere calculus} for sums of weighted spheres.
The sphere calculus relies on three crucial arguments: the observation that ``spheres partition balls’’ (a sequence of disjoint spheres in $B(r)$ have total volume at most $\#B(r)$), averaging techniques to show that \textit{on average} spheres satisfy the expected bounds, and the fact that the volume growth $\#B(r)$ and the kernel $J$ have regularly varying bounds.
The sphere calculus may be of independent interest for other models with long-range interactions.

\subsubsection{Organisation}
In Section \ref{sec:preliminaries} we collect notation for the paper, results in the structure theory of transitive graphs, and bounds for the \LRP{} kernel.
In Section \ref{section:sphere_calculus} we introduce the sphere calculus for transitive graphs of polynomial growth.
In Section \ref{sec:renormalisation} we introduce our renormalisation scheme via nets and we prove \thref{prop:finite_size_criterion} on the finite-size criterion. 
In Section \ref{sec:truncation_and_continuity} we  prove \thref{thm:truncation_problem} on the truncation problem and \thref{thm:continuity_percolation_density} on the continuity of the phase transition.
In Section \ref{sec:locality} we prove \thref{thm:lrp_locality} on the locality of $\beta_c$ and \thref{thm:theta_locality} on the joint continuity of $\theta$.
In Section \ref{sec:smoothness} we prove \thref{cor:characters_smoothness} on the smoothness of percolation characters.
In Section \ref{subsec:coarse_graining} we introduce our coarse-graining of the space using Voronoi tiles and scale-invariant nets.
Sections \ref{sec:cluster_decay_below} and \ref{sec:transience} are by far the most analytic parts of the paper where we prove \thref{prop:biskup} on the local existence of a giant, and the transience of the infinite cluster \thref{thm:random_walks} using the multi-scale renormalisation. The less interested reader may jump to the other sections harvesting the consequences of this method.
In Section \ref{sec:local_uniqueness} we prove \thref{prop:second_largest} on the local uniqueness of the giant.
In Section \ref{sec:consequences} we prove \thref{thm:cluster_size_decay} on the stretched-exponential decay of the distribution of finite clusters, \thref{prop:o_in_kinfty} saying that the giant is a local witness to the percolation density, \thref{cor:anchored_isop_dim} and \thref{thm:anch_dim_upper} on the anchored isoperimetric dimension of the infinite cluster, \thref{thm:lln} on the law of large numbers, and \thref{thm:cluster_size_lower_bound} on the lower bound for the distribution of finite clusters.
In Section \ref{sec:recurrence} we prove \thref{thm:two_dim_recurrence} on the recurrence of the infinite cluster $\kinfty$ when $d = 2$ and $\alpha \geq 2$.

\section{Preliminaries}
\label{sec:preliminaries}

\paragraph{Notation.}
Throughout, $G = (V,E)$ is an infinite, locally finite, transitive graph of polynomial growth with $d \geq 1$ equipped with the graph metric $d_G$ and a fixed origin $o$. 
We write $B(x,r)$ for the ball of radius $r$ centred at $x \in V$, and $S(x,r)$ for the sphere of radius $r$ centred at $x$. 
We write $B(r)$ and $S(r)$ for the ball and sphere of radius $r$ centred at the origin.
For a set $A \subset V$, we write $\# A$ for the cardinality of $A$. 
We let $\R_+ = \{x \in \R : x \geq 0\}$.
The kernel $J : V \times V \to \R_+$ will generally be assumed to satisfy $J(x,y) = \Theta(d_G(x,y)^{-d \alpha})$ with $\alpha > 1$, as defined in \eqref{eq:polynomial_kernel}.
Without loss of generality we fix the constants $c_J,C_J$, and $R_J$ in \eqref{eq:polynomial_kernel}, and we write $J(x,y) = \Omega(d_G(x,y)^{-d \alpha})$ and $J(x,y) = O(d_G(x,y)^{-d \alpha})$ to mean that the kernel satisfies either the lower bound or the upper bound, respectively.
For two sets $A,B \subset V$, we write $d_G(A,B) = \min_{x \in A,y \in B} d_G(x,y)$ for their Hausdorff distance.
For a set of vertices $H \subseteq V$, we write $\partial_E(H) = \{(x,y) \in E : x \in H, y \not\in H\}$ for the the edge boundary of $H$.
For two disjoint sets $A,B \subset V$, we write $J(A,B) = \sum_{x \in A} \sum_{y \in B} J(x,y)$.
We write $\norm{J} = \sum_{x \in V \setminus \{o\}} J(o,x)$. 
We write $\P_{\beta}$ for the law of \LRP{} on $G$ with kernel $J$ and parameter $\beta \geq 0$, and $\E_{\beta}$ for the associated expectation operator.
We write $x \sim y$ if $x,y \in V$ are connected by an edge, and we write $\cluster(x)$ for the component of the vertex $x \in V$ and $\cluster$ for the component of the origin.
We write $\giantcomp{x}{r}$ for the largest cluster in $B(x,r)$ and $\giantone{r}$ for the largest cluster in $B(r)$. 
Similarly, we write $\gianttwo{r}$ for the second-largest component in $B(r)$.  
The function $\theta(\beta) = \P_{\beta}(\# \cluster = \infty)$ is the percolation probability, and $\beta_c = \inf \{ \beta \geq 0 :\P_{\beta}(\# K = \infty) > 0 \}$ is the critical parameter.
We also deal with sequences of graphs and kernels.
In this setting we write $B_G(x,r)$, $B_G(r)$, $S_G(x,r)$, and $S_G(r)$ to emphasize the graph metric. 
Depending on which terms need to be emphasized we write $\P_{\beta,J}$ and $\P_{\beta,G,J}$, and $\E_{\beta,J}$ and $\E_{\beta,G,J}$.
Similarly, we write $\theta(\beta,J)$ and $\theta(\beta,G,J)$, and $\beta_c(J)$ and $\beta_c(G,J)$.
\paragraph{Kernel bounds.}
Here we gather kernel bounds that we use throughout the paper.
Since $1 - \exp(-t) \leq t$ for all $t \in \R$, then for all $x,y \in V$ we have
\begin{equation}    
    \label{eq:standard_kernel_bound}
    \P_{\beta}
    \left(
        x 
        \sim
        y
    \right)
    =
    1 
    -
    \exp 
    \left(
        -
        \beta
        J(x,y)
    \right)
    \leq
    \beta J(x,y).
\end{equation} 
We will repeatedly use the following bound, which relates to the ``bulk-to-bulk'' connections discussed in Section \ref{sec:intro}. 
For a transitive kernel $J : V \times V \to \R_+$ satisfying $J(x,y) = \Omega(d_G(x,y)^{-d \alpha})$, and for $A,B \subset V$ finite disjoint subsets with $d_G(A,B) \geq R_J$ and $A, B \subseteq B(R)$ for some $R > 0$, we have
\begin{equation}
    \P_{\beta}
    \left(
        A 
        \not\sim
        B
    \right)
    =
    \exp
    \left(
        - 
        \beta
        J(A,B)
    \right)
    \leq 
    \exp
    \left(
        -
        \frac{
                c_J
                \beta 
                \# A 
                \# B
            }{
                (2R)^{d \alpha}
            }
    \right).
\end{equation}  

\paragraph{Structure theory: polynomial volume implies polynomial growth.}
An important piece of the structure theory we use is the following immediate consequence of a result of Tessera and Tointon {\cite[Corollary 1.17]{tessera_balls_2024}}.
\begin{proposition}
\thlabel{prop:prop:polynomial_volume_implies_polynomial_growth}
    Suppose that $(G_n)_{n \in \N}$ is a sequence of transitive graphs converging locally to $G$ a transitive graph of polynomial growth with $d \geq 1$. 
    Then there exists a constant $C = C(d,G) > 0$ such that $\# B_{G_n}(r) \leq C r^d$ for all $n$ and $r$ sufficiently large.
\end{proposition}

This result says that witnessing polynomial volume at a sufficiently large scale implies polynomial growth \textit{at all scales}.
A well-known consequence of this theorem is that for each $d \geq 1$ the set of transitive graphs of polynomial growth with dimension at most $d$ is an open set in the local topology.
\section{Sphere calculus}
\label{section:sphere_calculus}
A real non-negative function $f$ is said to be \textbf{regularly varying with index $\gamma\in \R$} if $f(\lambda x)/f(x) \to \lambda^{\gamma}$ for all $\lambda > 0$ as $x \to \infty$. 
It follows from Karamata's theorem \cite[Theorem 1.5.11]{bingham_regular_1989} that if $f$ is regularly varying with index $\gamma > - 1$, then $F(x) = \int_0^{x} f(x) \dd x$ is itself regularly varying with index $\gamma + 1$.
It is immediate that a transitive graph of polynomial growth is a \textbf{doubling metric space}, in the sense that there exists $C_G = C_G(G) > 0$ such that $\#B(2r) \leq C_G \# B(r)$ for all $r \in \N$.

\begin{lemma}[Sphere calculus]
    \thlabel{lem:sphere_calculus}
    Let $G$ be a transitive graph of polynomial growth with $d \geq 1$.
    Let $f$ be a non-negative regularly varying function with index $\gamma\in \R$. 
    Then there exist $c = c(\gamma, G) > 0$ and $C = C(\gamma, G) > 0$ such that
    for all $r_0,r_1 \in \N \cup \{\infty\}$ with $r_0 < r_1$, %
    \begin{equation}
        \label{eq:sphere-easy-1}
        c
        \int_{r_0}^{r_1}
        t^{d-1}
        f(t)
        dt
        \leq 
        \sum_{s = r_0}^{r_1}
        \# 
        S(s)
        f(s)
        \leq 
        C
        \int_{r_0}^{r_1}
        t^{d-1}
        f(t)
        dt.
    \end{equation}
    Let $A>1$ be fixed. 
    Then there exists $c = c(\gamma,G) > 0$ such that for all $r_1 \in \N$ 
    \begin{equation}
        \label{eq:sphere-harder-1}
        \sum_{r_1/A \leq s \leq r_1}
        \# 
        S(s)
        f(r_1-s)
        \geq 
        c
        r_1^{d-1}
        \int_{1}^{r_1}
        f(t)
        dt.
    \end{equation}
    Let $k \in \N$ and let $\epsilon>0$.
    Then there exists $C = C(G,A, \gamma) > 0$ such that at least a $(1-\epsilon)$ proportion of the radii $r \in [A^{k-1},A^k)$ satisfy 
    \begin{equation}
        \label{eq:sphere-harder-upper}
        \sum_{r/A \leq s \leq r}
        \# 
        S(s)
        f(r-s)
        \leq 
        \frac{C}{\epsilon} 
        r^{d-1}
        \int_{1}^{r}
        f(t) 
        dt.
    \end{equation}
\end{lemma}

\begin{proof}
    Since $f(t)$ is regularly varying with index $\gamma$ and $t^{d-1}f(t)$ is regularly varying with index $\gamma + d - 1$, the Uniform Convergence Theorem for regularly varying functions \cite[Theorem 1.5.2]{bingham_regular_1989} gives that there exist $c_{\gamma} = c(\gamma,d), C_{\gamma} = C_{\gamma}(\gamma,d) > 0$, and $R > 0$ such that 
    \begin{align}
        \label{eq:unit-interval-bound}
        c_\gamma 
        r^{d-1}
        f(r)
        \leq 
        t^{d-1}f(t)
        \leq 
        C_{\gamma} 
        r^{d-1}f(r)
        \quad 
        \text{ and }
        \quad
        c_\gamma 
        f(r)
        \leq 
        f(t)
        \leq 
        C_{\gamma} 
    \end{align}
    for all $r \geq R$ and for all $t \in [r,2r]$, and in particular for all $t \in [r,r+1]$.
    We begin by proving \eqref{eq:sphere-easy-1}. 
    By the known lower bounds for the volume of spheres \cite{breuillard_rate_2013} together with the bounds above,
    \begin{equation}
        \sum_{s = r_0}^{r_1}
        \# 
        S(s)
        f(s) 
        \geq 
        c_G  
        \sum_{s=r_0}^{r_1}  
        s^{d-1} 
        f(s)
        \geq 
        \frac{c_G}{c_\gamma}
        \sum_{s=r_0}^{r_1}
        \int_s^{s+1} 
        t^{d-1} 
        f(t)
        \dd t
        \geq 
        \frac{c_G}{c_\gamma}
        \int_{r_0}^{r_1} 
        t^{d-1} 
        f(t)
        \dd t 
    \end{equation}
    for some $c_G > 0$, where in the last inequality we use that $f$ is non-negative.
    We now prove the upper bound in \eqref{eq:sphere-easy-1}.
    Writing $N = \lfloor \log_2(r_1/r_0) \rfloor$, we bound the sum using diadic blocks
    \begin{equation}
        \label{eq:sphere-factorise-1}
        \sum_{s = r_0}^{r_1}
        \# 
        S(s)
        f(s)
        \leq 
        \sum_{i = 0}^{N}
        \max_{2^i r_0 \leq s \leq (2^{i+1} r_0 - 1) \wedge r_1}
        f(s)
        \sum_{s = 2^i r_0}^{(2^{i+1} r_0 - 1) \wedge r_1}
        \# S(s).
    \end{equation}
    Since spheres partition balls and since $G$ is a doubling metric space, for all $0 \leq i \leq N$,
    \begin{equation}
        \sum_{s = 2^i r_0}^{(2^{i+1} r_0 - 1) \wedge r_1}
        \# S(s)
        \leq
        \# 
        B(2^{i+1} r_0)
        \leq
        c_1
        \#
        B(2^i r_0)
        \leq 
        c_2
        (2^i r_0)^d
    \end{equation}
    for some $c_1 = c_1(G), c_2 = c_2(G) > 0$.
    Since $f$ is regularly varying, there exists $\widetilde{C}_\gamma > 0$ such that 
    \begin{equation}
        \max_{2^i r_0 \leq s \leq (2^{i+1} r_0 - 1) \wedge r_1}
        f(s) 
        \leq 
        \widetilde{C}_\gamma
        f(2^i r_0).
    \end{equation}
    Applying the two bounds above to \eqref{eq:sphere-factorise-1} we bound further using \eqref{eq:unit-interval-bound} with $t \in [r,r+1]$
    \begin{align}
        \label{eq:all_together_1}
        \sum_{s = r_0}^{r_1}
        \# 
        S(s)
        f(s)
        \leq 
        c_2
        \widetilde{C}_\gamma
        \sum_{i = 0}^{N}
        (2^i r_0)^d
        f(2^i r_0)  
        & \leq   
        \frac{c_2 \widetilde{C}_\gamma}{c_\gamma}  
        \sum_{i = 0}^{N} 
        \int_{2^i r_0}^{2^{i+1}r_0} 
        t^{d-1} 
        f(t)
        \dd t
        \\
        \label{eq:all_together_2}
        & \leq 
        \frac{c_2 \widetilde{C}_\gamma}{c_\gamma}
        \int_{r_0}^{2r_1}
        t^{d-1} f(t)\mathrm d t,
    \end{align}
   where for \eqref{eq:all_together_2} we use that $2^{i+1}r_0 \leq 2r_1$.  
    Writing $F(x) = \int_{r_0}^x t^{d-1}f(t) \dd t$, Karamata's theorem gives that $F(x)$ is regularly varying, so there exists $c > 0$ such that $F(2r_1) \leq c F(r_1)$ for all $r_1 \in \N$ with $r_1 > r_0$.
    Hence we can conclude the proof of the upper bound in \eqref{eq:sphere-easy-1}, as for some $C = C(G,\gamma) > 0$
    \begin{equation}
        \sum_{s = r_0}^{r_1}
        \# 
        S(s)
        f(s)
        \leq 
        C
        \int_{r_0}^{r_1}
        t^{d-1} f(t)\mathrm d t.
    \end{equation}
    We now prove \eqref{eq:sphere-harder-1} on the lower bound for the truncated convolution.
    Since $i \leq r_1/A$ it follows that $r_1-i \geq (1-1/A)r_1$, and together with the lower bound for the volume of spheres we have
    \begin{align}
        \sum_{r_1/A \leq s \leq r_1}
        \# 
        S(s)
        f(r_1-s)
        =
        \sum_{0 \leq i \leq r_1 - r_1/A}
        \# 
        S(r_1 - i)
        f(i)
        \geq
        c_G
        r_1^{d-1}
        \sum_{0 \leq i \leq r_1 - r_1/A}
        f(i)
    \end{align}
    for some $c(G, A) > 0$.
    The bounds in \eqref{eq:unit-interval-bound} give for some $c = c(A,\gamma) > 0$ that
    \begin{equation}
       \sum_{0 \leq i \leq r_1 - r_1/A}
       f(i)
       \geq 
       \sum_{1 \leq i \leq r_1 - r_1/A}
       f(i)
       \geq 
       C_\gamma^{-1}
       \int_{1}^{r_1 - r_1/A} 
       f(t)
       \dd t 
       \geq 
       c
       C_\gamma^{-1}
       \int_{1}^{r_1} 
       f(t)
       \dd t,
    \end{equation}
    where in the last inequality we use that by Karamata's theorem the integral $F(x) = \int_{1}^{x} f(t) \dd t$ is itself a regularly varying function and hence, regardless of whether the integral is convergent or divergent, the integrals from $r_1 - r_1/A$ to $r_1$ can be absorbed into the constant $c$.
    We now prove \eqref{eq:sphere-harder-upper}. 
    Let $k \in \N$ and $\epsilon > 0$.
    We begin by computing the average of the sum in \eqref{eq:sphere-harder-1} for $r \in [A^{k-1}, A^k)$, namely, with $N_A(k) = \# \{r \in \N: r \in [A^{k-1}, A^k)\}$,
    \begin{align}
        \overline R_f(k):
        =
        \frac{1}{N_A(k)}
        \sum_{A^{k-1} \leq r < A^k} 
        \sum_{r/A \leq s \leq r}
        \# 
        S(s)
        f(r - s)\\
        \leq 
        \frac{1}{N_A(k)}
        \sum_{A^{k-2} \leq s < A^k}
        \# 
        S(s)
        \sum_{
            s \leq r \leq As
        }
        f(r - s),\label{eq:exchanged_sums}
    \end{align}
    where we exchanged the summation order to obtain the second row.
    Then, for all $s \in [A^{k-2},A^k)$, 
    \begin{equation}
        \sum_{
            s \leq r \leq As
        }
        f(r - s)
        \leq 
        \sum_{0 \leq j \leq (A-1)A^k}
        f(j)
        \leq 
        c_{\gamma} 
        \int_{1}^{(A-1)A^k}
        f(t)
        \dd t
        \leq 
        c_{\gamma}
        \int_{1}^{A^{k-1}}
        f(t)
        \dd t,
    \end{equation}
    where the second inequality follows from \eqref{eq:unit-interval-bound} and the third inequality follows as before.
    Upon applying this bound to \eqref{eq:exchanged_sums} and by the ``spheres partition balls'' observation together with the fact that $\# B(A^k)/N_A(k) \leq c A^{(k-1)(d-1)}$ for some $c > 0$, we have 
    \begin{multline}
        \overline R_f(k)
        \leq 
        \frac{1}{N_A(k)}
        \bigg(
            \sum_{A^{k-2} \leq s < A^k}
            \# 
            S(s)
        \bigg)
        \bigg(
            c_{\gamma}
            \int_{1}^{A^{k-1}}
            f(t)
            \dd t
        \bigg)
        \\
        \leq
        \frac{
            c_{\gamma} 
            \# B(A^k)
        }{
            N_A(k)
        }
        \int_{1}^{A^{k-1}}
        f(t)
        \dd t
        \leq 
        c 
        r^{d-1}
        \int_{1}^{r}
        f(t)
        \dd t,
    \end{multline}
    where in the last step we use monotonicity and that $A^{k-1} \leq r$ for all $r \in [A^{k-1},A^k]$.
    We conclude the result by the pigeonhole principle: for at least a $(1 - \epsilon)$ proportion of the radii $r \in [A^{k-1},A^k)$,
    \begin{equation}
        \sum_{r/A \leq s \leq r}
        \# 
        S(s)
        f(r-s)
        \leq 
        \overline R_f(k)/\epsilon.
    \end{equation}
    \vskip-1em
\end{proof}

We use a slightly different sphere calculus to prove the divergence of a series involving the reciprocal of weighted spheres. 
This result is specific to two-dimensional transitive graphs, and we use it to establish the effective resistance criterion for recurrence in \thref{thm:two_dim_recurrence}.

\begin{lemma}
    \thlabel{lem:log_divergence}
    Let $G$ be a transitive graph of polynomial growth with $d = 2$. 
    Then 
    \begin{equation}
        \sum_{r \ge 1}
        \frac{1}{\# S(r) \log(\# S(r))}
        =
        \infty.
    \end{equation}
\end{lemma}

\begin{remark}
    Let $S_{E}(r) = \{(x,y) \in E : d_G(o,\{x,y\} = r \}$ denote the edge sphere of radius $r$ centred at the origin.
    Note that $\# S(r) \leq \# S_E(r) \leq \deg(G) \# S(r)$ for all $r \in \N$.
    The statement of \thref{lem:log_divergence} also holds with the edge sphere $S_{E}(r)$ in place of the sphere $S(r)$.
\end{remark}

\begin{proof}
    For $k \in \N$, the average volume of a sphere with radius $r \in (2^{k-1},2^k]$ satisfies 
    \begin{equation}
        \label{eq:sphere_vol_average}
        \frac{1}{2^{k-1}}
        \sum_{2^{k-1} < r \leq 2^k}
        \# 
        S(r)
        =
        \frac{
            \# B(2^k)
            -
            \# B(2^{k-1})
        }{
            2^{k-1}
        }
        \leq 
        c 
        2^{k-1} 
    \end{equation}
    for some $c > 0$, where we use the volume bounds for balls.
    Let $f(x) = 1 / (x \log x)$, which is both convex and decreasing for $x > 1$.
    Applying $f$ to the average volume of a sphere we have
    \begin{equation}
        \label{eq:averaged_bound_with_f}
        \frac{1}{2^{k-1}}
        \sum_{2^{k-1} < r \leq 2^k}
        \frac{1}{\# S(r) \log(\# S(r))}
        \geq
        f
        \bigg(
            \frac{1}{2^{k-1}}
            \sum_{2^{k-1} < r \leq 2^k}
            \#S(r)
        \bigg)
        \geq
        f(c 2^{k-1})
    \end{equation}
    where the first inequality follows from the convexity of $f$, and the second inequality follows from the monotonicity of $f$ together with \eqref{eq:sphere_vol_average}.
    Rearranging \eqref{eq:averaged_bound_with_f} and by the definition of $f$ we conclude
    \begin{equation}
        \sum_{r \ge 1}
        \frac{1}{\# S(r) \log(\#S(r))}
        =
        \sum_{k \ge 1}
        \sum_{2^k-1 < r \leq 2^k}
        \frac{1}{\# S(r) \log(\#S(r))}
        \geq
        \sum_{k \ge 1}
        \frac{1}{c(k-1) \log 2}
        = 
        \infty.
    \end{equation} 
    \vskip-1em
\end{proof}

\subsection{Integrability and non-integrability}
The first of the sphere calculus is to recover the standard calculations for $\Z^d$ involving the \LRP{} kernel and the expected degree.

\begin{lemma}
    \thlabel{lem:integrability}
    Let $G$ be a transitive graph of polynomial growth with $d \geq 1$, and suppose that $J : V \times V \to \R_+$ is a transitive kernel.
    Suppose that $J$ satisfies $J(x,y) = \Omega(d_G(x,y)^{-d \alpha})$ with $\alpha > 0$.
    If $\alpha > 1$ then there exists $c = c(G,J) > 0$ such that     for all $x \in V$ and $r \in \N$
    \begin{equation}
        \label{eq:kernel_lower_bound}
        c
        r^{d(1-\alpha)}
        \leq
        \sum_{L = r}^{\infty}
        \sum_{y \in S(x,L)}
        J(x,y).
    \end{equation}
    If $\alpha \leq 1$ then $\sum_{L = r}^{\infty} \sum_{y \in S(x,L)} J(x,y) = \infty$ and hence $\sum_{L = r}^{\infty} \E_{\beta} \left[ \deg(x,L) \right] = \infty$ for all $x \in V$ and $r \in \N$.
    Suppose that $J$ satisfies $J(x,y) = O(d_G(x,y)^{-d \alpha})$ with $\alpha > 1$.
    Then there exists $c = c(G,J) > 0$ such that     for all $x \in V$ and $r \in \N$
    \begin{equation}
        \label{eq:kernel_upper_bound}
        \sum_{L = r}^{\infty}
        \sum_{y \in S(x,L)}
        J(x,y)
        \leq 
        c 
        r^{d(1-\alpha)}.
    \end{equation}
    In particular, $\sum_{y \in V \setminus \{x\}} J(x,y) < \infty$ and $\E_{\beta} \left[ \deg(x) \right] < \infty$ for all $x \in G$.
\end{lemma}

\begin{proof}
    By the transitivity of $J$ it is sufficient to consider $J(o,y)$.
    Let $m = \max(r,R_J)$, so that  
    using the known lower bounds for the volume of spheres we have
    \begin{align}
        \sum_{L = r}^{\infty} 
        \sum_{y \in S(x,L)} 
        J(x,y)
        \geq 
        c_J
        \sum_{L=m}^{\infty}
        \# S(x,L) 
        L^{-d \alpha}
        \geq 
        c_J
        \sum_{L=m}^{\infty} 
        L^{d-1-d \alpha}
    \end{align}
    for some $c = c(G,J) > 0$.
    When $\alpha > 1$ the series converges and gives \eqref{eq:kernel_lower_bound}.
    When $\alpha \leq 1$ the series diverges and so does $\E_{\beta} \left[ \deg(x,L) \right]$. Now we prove \eqref{eq:kernel_upper_bound}. We bound
    \begin{equation}
        \sum_{L = r}^{\infty} 
        \sum_{y \in S(x,L)} 
        J(x,y)
        \leq
        \charf
        \left(
            r
            < 
            R_J
        \right)
        \sum_{L = r}^{R_J - 1}
        \sum_{y \in S(x,L)} 
        J(x,y)
        +
        c_J
        \sum_{L = m}^{\infty}
        \# 
        S(x,L)
        L^{- d \alpha}.
    \end{equation}
    The first sum is finite. The second sum, using  the sphere calculus \eqref{eq:sphere-easy-1} can be bounded by
    \begin{equation}
        \sum_{L = m}^{\infty}
        \# 
        S(x,L)
        L^{- d \alpha}
        \leq
        c_1 
        \int_{\max(r,R_J)}^{\infty}
        t^{d-1-d\alpha}
        dt
        \leq
        c_2
        r^{d(1-\alpha)} 
    \end{equation}
    for some $c_1 = c_1(G) > 0$ and $c_2 = c_2(G,J) > 0$, where we use that $\alpha > 1$. 
    This gives \eqref{eq:kernel_upper_bound}.
\end{proof}

We also use the sphere calculus to compute the order of the expected number of edges between $B(r)$ and its complement $B(r)^c$.

\begin{lemma}
    \thlabel{lem:averaged_ball_connection}
    Let $G$ be a transitive graph of polynomial growth with $d \geq 1$, and suppose that $J : V \times V \to \R_+$ is a transitive kernel with $J(x,y) = O(d_G(x,y)^{-d \alpha})$ with $\alpha>1$.
     Let $\epsilon \in(0,1)$.
    Then there exists $c = c(G,J)$ such that for all $k\in \N$, at least $(1-\varepsilon)$ proportion of the radii $r \in [2^{k-1},2^k]$ satisfy
    \begin{equation}
    \label{eq:Jrr-sharper}
        J
        \left(
            B(r), 
            B(r)^c
        \right)
        \leq 
         \frac{c}{\varepsilon}
        r^{
        \max
        \left(
            d(2-\alpha),
            d - 1
        \right)  
        }
        \log(r)^{
            \charf
            \left(
                \alpha=1+1/d
            \right)
        }.
    \end{equation}
\end{lemma}

\begin{proof}[Proof of \thref{lem:averaged_ball_connection}]
    For $x>0$, we write $B(x):=B(\lfloor x\rfloor)$.
    \begin{align}
        J(
            B(r),
            B(r)^c
        )
        =
        J
        (
            B(r/2)
            ,
            B(r)^c
        )
        + 
        J
        (
            B(r) \setminus B(r/2),
            B(r)^c
        )
    \end{align}
    and decomposing the sum over $x \in B(r/2)$ according to the distance of $x$ from $B(r/2)^c$ we have
    \begin{align}
    \label{eq:J-Br-connect}
        J
        (
            B(r/2)
            ,
            B(r)^c
        )
        = 
        \sum_{z = 0}^{\lfloor r/2\rfloor } 
        \sum_{x \in S(r-z)} 
        \sum_{y \in B(r)^c} 
        J(x,y)
        \leq 
        \sum_{z = 0}^{\lfloor r/2\rfloor } 
        \# S(r - z)
        \sum_{y \in B(z)^c} 
        J(o,y),
    \end{align}
    where we use the transitivity of the kernel $J$ and the fact that for $x \in S(r-z)$ we have $d_G(x,B(r)^c) \geq z$.
    For $z \in \N$, it follows from \thref{lem:integrability} that there exists $c_1 = c_1(G,J) > 0$ such that
    \begin{equation}\label{eq:yet-another-unlabeled-eq}
        \sum_{y \in B(z)^c} 
        J(o,y)
        =
        \sum_{L = z + 1}^{\infty}
        \sum_{y \in S(L)}
        J(o,y)
        \leq 
        c_1 
        (z + 1)^{d(1 - \alpha)}.
    \end{equation}
    Note that $(z+1)^{d(1-\alpha)} \leq (r/2)^{d(1-\alpha)}$ for $z \geq r/2$ and $\alpha>1$, and by the spheres partition balls trick
    \begin{equation}
        \label{eq:bulk_edges}
        J
        (
            B(r/2),
            B(r)^c
        )
        \leq 
        c_1
        (r/2)^{d(1-\alpha)}
        \sum_{z = 0}^{r/2} 
        \# S(r - z)
        \leq
        c_1
        (r/2)^{d(1-\alpha)}
        \# B(r/2)
        \leq 
        c_2 r^{d(2-\alpha)}
    \end{equation}
    for some $c_2 = c_2(G,J) > 0$.
    The same argument as above yields
    \begin{equation}
        J
        (
            B(r) \setminus B(r/2),
            B(r)^c
        )
        \leq 
        c_1 
        \sum_{z = r/2 + 1}^{r}
        \# 
        S(r-z)
        (z + 1)^{d(1-\alpha)}.
    \end{equation}
    By the sphere calclus, at least $(1-\varepsilon)$ proportion of the radii $r \in [2^{k-1},2^k]$ satisfy
    \begin{equation}
        \sum_{z = r/2 + 1}^{r}
        \# 
        S(r-z)
        (z + 1)^{d(1-\alpha)}
        \leq  
        \frac{c}{\varepsilon} 
        r^{d-1}
        \int_{1}^{r}
        (t + 1)^{d(1-\alpha)}
        dt.
    \end{equation}
    Evaluating the integral, we have that at least $(1-\varepsilon)$ proportion of the radii $r \in [2^{k-1},2^k]$ satisfy
    \begin{equation}
        \label{eq:boundary_effects}
        J
        (
            B(r) \setminus B(r/2),
            B(r)^c
        )
        \leq 
        \begin{cases} 
            c_4  
            r^{d-1}
            & 
            \text{if } 
            d(1 - \alpha) 
            <
            -1,
            \\
            c_4 
            r^{d-1}
            \log r
            & 
            \text {if } 
            d(1 - \alpha)
            =
            -1,
            \\
            c_4 
            r^{d(2 - \alpha)} 
            & \text {if } 
            d(1 - \alpha)
            >
            - 1.
        \end{cases} 
    \end{equation}
    The result follows by combining \eqref{eq:bulk_edges} and \eqref{eq:boundary_effects}.
\end{proof}

Equipped with the structure theory from the previous section and the sphere calculus of this section, we can now prove that a natural class of examples satisfy the definitions of local convergence and uniform integrability in Section \ref{subsec:local_topology}.

\begin{lemma}
\thlabel{lem:kernel_example}
    Let $G$ be a transitive graph of polynomial growth with $d \in \N$ and let $J :V \times V \to \R_+$ be the transitive kernel given by $J(x,y) = d_G(x,y)^{- d \alpha}$ for $\alpha > 0$.
    Let $(G_n)_{n \in \N}$ be a sequence of transitive graphs converging to $G$, let $(s_n)_{n \in \N}$ be a sequence converging to $d \alpha$, and let $(J_n)_{n \in \N}$ be the sequence of kernels $J : V_n \times V_n \to \R_+$ given by $J_n(x,y) = d_{G_n}(x,y)^{- s_n}$. 
    Then $(G_n,J_n)_{n \in \N}$ converges locally to $(G,J)$. Further, if $\alpha>1$, then the kernels $(J_n)_{n \in \N}$ are uniformly integrable.
\end{lemma}

\begin{proof}[Proof of \thref{lem:kernel_example}]
    Let $\epsilon > 0$ and $R>0$. 
    By the definition of local convergence we have $B_{G_n}(2R) \cong B_{G}(2R)$ for all $n$ sufficiently large, so that the graph metrics $d_{G_n}$ and $d_G$ agree in $B(R)$, and we may choose $n$ sufficiently large such that
    \begin{equation}
        \label{eq:kernel_l1_convergence}
        \sum_{x,y \in B(R)} 
        \lvert 
            J_n(o,x)
            -
            J(o,x)
        \rvert
        =
        \sum_{x,y \in B(R)} 
        \lvert 
            d_{G_n}(o,x)^{-s_n}
            -
            d_{G}(o,x)^{-d \alpha}
        \rvert
        \leq 
        \epsilon.
    \end{equation}
    Since $\epsilon$ and $R$ are arbitrary, this shows that for all $\epsilon > 0$ the coupling radius $R_n(G_n,G,\epsilon)$ tends to infinity and the sequence $(G_n,J_n)_{n \in \N}$ converges locally to $(G,J)$. 
    We now show that the kernels $(J_n)_{n \in \N}$ are uniformly integrable. 
    Since $(s_n)_{n \in \N}$ converges to $d \alpha$ and $\alpha > 1$, for $0 < \sigma < d - s$ we have $s_n \geq d + \sigma$ and hence $J_n(x,y) \leq d_{G_n}(x,y)^{-(d + \sigma)}$ for all $n$ sufficiently large.
    Then
    \begin{equation}
    \label{eq:sum_over_spheres}
        \sum_{x \not\in B(R)} 
        J_n(o,x)    
        =
        \sum_{r \geq R}
        \sum_{x \in S_{G_n}(r)}
        J_n(o,x)
        \leq 
        \sum_{r \geq R}
        \# S_{G_n}(r) 
        r^{-(d + \sigma)}
    \end{equation}
    for all $n$ sufficiently large, where recall that $S_{G_n}(r)$ denotes the sphere of radius $r$ in $G_n$. 
    By \thref{prop:prop:polynomial_volume_implies_polynomial_growth} there exists $C > 0$ such that $\# B_{G_n}(r) \leq C r^d$ for all $n$ and $r$ sufficiently large, and together with the sphere calculus in \eqref{eq:sphere-easy-1}, for all $n$ sufficiently large, it holds that 
    \begin{equation}
        \sum_{r \geq R}
        \# S_{G_n}(r) 
        r^{-(d + \sigma)}
        \leq 
        C
        \int_{R}^{\infty}
        t^{d-1}
        t^{-(d + \delta)}
        \dd t
        =
        C R^{-\delta}/\delta.
    \end{equation}
    If $\alpha > 1$, the kernel $J$ is integrable. 
    Writing $R_n = R_n(G_n,G,\epsilon)$ and by \eqref{eq:kernel_l1_convergence} we have
    \begin{multline}
        \sum_{x \in V_n \setminus \{o_n\}}
        J_n(o,x)
        =
        \sum_{x \in B_{G_n}(R) \setminus \{o_n\}}
        J_n(o,x)
        +
        \sum_{x \in G_n \setminus B_{G_n}(R)}
        J_n(o,x)
        \\
        \leq 
        \Bigg(
            \epsilon
            +
            \sum_{x \in B_{G}(R) \setminus \{o\}}
            J(o,x)
        \Bigg)
        +
        \frac{C R^{-\delta}}{\delta}
        <
        M
    \end{multline}
    for some $M > 0$, for all $n$ sufficiently large.
    Hence $(J_n)_{n \in \N}$ is uniformly integrable as required.
\end{proof}

It is also a consequence of the structure theory of transitive graphs and the definition of local convergence that non-integrable pairs naturally arise as local limits of integrable pairs.

\begin{lemma}
    \thlabel{lem:non_integrable_limits}
    Let $G$ be a transitive graph of polynomial growth with $d \geq 2$, and let $J : V \times V \to \R_+$ be the transitive kernel given by $J(x,y) = d_G(x,y)^{-d \alpha}$ for $\alpha > 0$.
    Let $(G_n)_{n \in \N}$ be a sequence of transitive graphs and let $(J_n)_{n \in \N}$ be the sequence of transitive kernels given by $J_n(x,y) = d_G(x,y)^{-\dim(G_n) \alpha_n}$ with $\alpha_n > 1$.
    If the sequence $(G_n,J_n)_{n \in \N}$ converges locally to $(G,J)$, then $\alpha = \lim_{n \to \infty} \alpha_n\dim(G_n)/ d$.
    In particular, if $\limsup_{n \to \infty} \dim(G_n)<d$ and $\limsup_{n\to \infty} \alpha_n< d/\limsup_{n \to \infty} \dim(G_n)$ then $\alpha<1$ and the pairs $(G_n,J_n)$ are integrable for all $n$, but not uniformly integrable, and the pair $(G,J)$ is not integrable.
\end{lemma}

\begin{proof}
    Let $\epsilon > 0$ and let $R_n = R_n(G_n,G,\epsilon)$ be the coupling radius. 
    By the definition of local convergence $R_n$ tends to infinity and $d_{G_n}$ agrees with $d_G$ on $B(R_n)$.
    For $x,y \in V$, we have that $x,y \in B(R_n)$ for all $n$ sufficiently large and hence 
    \begin{equation}
    \begin{aligned}
        \lvert
            d_G(x,y)^{-\dim(G_n) \alpha_n}
            -
            d_G(x,y)^{-d \alpha}
        \rvert
        =
        \lvert
            J_n(x,y)
            -
            J(x,y)
        \rvert
        \leq 
        \sum_{x,y \in B(R_n)}
        \lvert
            J_n(x,y)
            -
            J(x,y)
        \rvert,
        \end{aligned}
    \end{equation}
    which is at most $\varepsilon$ by \eqref{eq:kernel_l1_convergence}. Taking logarithms gives that $\dim(G_n) \alpha_n$ tends to $d \alpha$ with $n$. 
\end{proof}

The structure theory of transitive graphs in \thref{prop:prop:polynomial_volume_implies_polynomial_growth} implies that $\limsup_{n \to \infty} \dim(G_n) \leq d$.
However, it may well be that $\limsup_{n \to \infty} \dim(G_n)$ is \textit{strictly} less than $d$, as in the obvious example of the standard Cayley graphs of $\Z^2 \times (\Z / n \Z)$ converging locally to the standard Cayley graph of $\Z^3$.
This shows that the integrability of the limit kernel is not solely determined by the locally determined decay of the kernel, but that dimension also plays a role and a non-integrable limit kernel may arise when there is a dimension increase in the limit.
It is worth noting that dimension \textit{is} locally determined, in a different sense.
By \cite[Theorem 1.9]{tao_inverse_2017} and \cite[Corollary 1.22]{tessera_balls_2024}, the ``local'' growth rate of a transitive graph of polynomial growth can change boundedly many times, and there is a finite scale after which the dimension stabilises.
\section{Renormalisation via nets and the finite-size criterion}
\label{sec:renormalisation}
Let $G = (V,E)$ be a graph. For $a\ge 1,b \geq \lceil a\rceil$, an \textbf{$(a,b)$-net} is a subset $V_0 \subseteq V$ which is both $a$-separated and $b$-dense, meaning that it satisfies 
\begin{equation}
\label{eq:a_sep_b_dense}
    \min \{ d_G(x,y) : x,y \in V_0, x \neq y \}
    \geq 
    a,
    \qquad 
    \max \{ d_G(x,y) : x \in V, y \in V_0 \} 
    \leq 
    b.
\end{equation}
For an $(a,b)$-net $V_0$, we define the \textbf{net graph} $N(V_0)$ to be the graph with vertex set $V_0$ and where two vertices $x,y \in V_0$ are connected in $N(V_0)$ if $d_G(x,y) \leq 4b$. 
The choice of \textit{four} for the connection threshold is somewhat arbitrary, it ensures that the net graph is connected.
For $A \geq 1$, we say that $G$ admits \textbf{$A$-controlled nets} if for every $a \geq 1$ there exists an $(a,aA)$-net $V_0$ with the net graph $N(V_0)$ containing an $\L^2$ subgraph.
It follows from Zorn's lemma that for $a \geq 1$ we can take a maximal  $a$-separated subset of $V$ to obtain an $(a,\lceil a
\rceil)$-net. 
Here, maximality is with respect to inclusion.
To prove the existence and uniqueness of the giant in \thref{prop:biskup,prop:second_largest} respectively, we use maximally separated nets as the basis of the coarse-graining argument.
To prove the truncation problem and the continuity of the phase transition in \thref{thm:truncation_problem,thm:continuity_percolation_density} respectively, we run a renormalisation argument on a \textit{fixed} graph, for which we require the existence of $A$-controlled nets.

\begin{proposition}
    \thlabel{prop:fixed_unif_control_nets}
    Let $G$ be a transitive graph of polynomial growth with $d \geq 2$.
    Then there exists $A \in \N$ such that $G$ admits $A$-controlled nets.
\end{proposition}

To prove locality and joint continuity in \thref{thm:lrp_locality,thm:theta_locality} we run renormalisation arguments on a \textit{sequence} of graphs. 
    For this we need the existence of \textit{uniformly} $A$-controlled nets over $(G_n)_{n\ge 1}$.

\begin{proposition}[{\cite[Proposition 1.4]{contreras_locality_2023}}]
    \thlabel{prop:unif_control_nets}
    Let $(G_n)_{n \in \N}$ be a sequence of transitive graphs with $\liminf_{n \to \infty} \dim(G_n) \geq 2$. 
    If $(G_n)_{n \in \N}$ converges locally to $G$ a transitive graph of polynomial growth, then there exists $A \geq 1$ such that for all $n$ sufficiently large $G_n$ admits $A$-controlled nets.
\end{proposition}

\thref{prop:unif_control_nets} is proved in \cite[Proposition 1.4]{contreras_locality_2023} using the finitary structure theory of transitive graphs \cite{breuillard_structure_2012,tessera_finitary_2021}.
\thref{prop:fixed_unif_control_nets} is a special case of \thref{prop:unif_control_nets}, and it also follows from the classical structure theory of transitive graphs.
%
When the graph has superlinear growth, we use the structure theory of transitive graphs to identify a finite-size criterion for supercriticality.

\begin{proposition}[Finite-size criterion for $d \geq 2$]
\thlabel{prop:finite_size_criterion}
    Let $G$ be a transitive graph of polynomial growth with $d \geq 2$ and let $J : V \times V \to \R_+$ be a transitive kernel.
    Suppose that $G$ admits $A$-controlled nets for some $A \geq 1$.
    Let $\delta > 0$ and let $\beta \in \R_+$.
    Suppose that there exists $\nu > 0$ such that  for some $R \in \N$ we have
    \begin{equation}
    \label{eq:finite_size_criterion_1}
        \P_{\beta}
        \left(
            \# 
            \giantone{R}
            \geq
            \nu
            \# B(R)
        \right) 
        \geq 
        1
        -
        \delta,
    \end{equation}
    and for any two sets $S_1, S_2 \subseteq B(10AR)$ with $d_G(S_1,S_2) \geq 2R$ and $\#S_1, \#S_2 \geq \nu \# B(R)$, 
    \begin{equation}
    \label{eq:finite_size_criterion_2}
        \P_\beta
        \left(
            S_1 
            \sim 
            S_2 
        \right) 
        \geq 
        1
        -
        \delta
        .
    \end{equation}      
    If $\delta = c_G/(18 (8A)^d C_G)$ then $\beta > \beta_c(G,J)$.
    Further, $\P_{\beta} \left(\giantone{R} \leftrightarrow \infty \right) \to 1$ as $\delta \to 0$.
\end{proposition}

Note that the criterion with $\delta = c_G/(18 (8A)^d C_G)$ is genuinely \textit{finite} in size since there are no assumptions on the kernel concerning edges with endpoints outside of a finite ball: a finite-range kernel with range $20AR$ can satisfy the criterion.

\begin{proof}[Proof of \thref{prop:finite_size_criterion}]
    Let $V_0$ be a $(4R,4AR)$-net with the net graph $N(V_0)$ containing an $\L^2$ subgraph.
    We define a mixed site-bond percolation process $\eta$ on $N(V_0)$ as follows.
    Given $\omega$ a \LRP{} configuration on $G$, we say that a site $x \in V_0$ is \textbf{open} in $\eta$ if $\giantcomp{x}{R} \geq \nu \#B(x,R)$ in $\omega$, and a bond $(x,y)$ in $N(V_0)$ is \textbf{open} in $\eta$ if both $x$ and $y$ are open in $\eta$ and there exists an open edge between $\giantcomp{x}{R}$ and $\giantcomp{y}{R}$ in $\omega$. 
    We write $\P_{\eta}$ for the law associated to $\eta$.
    Since the net $V_0$ is $4R$-separated, $B(x,R)$ and $B(y,R)$ are disjoint for any $x, y \in V_0$.
    By the independence of edges in \LRP{}, by the transitivity of the graph and the kernel, and by the assumption in \eqref{eq:finite_size_criterion_1}, a site $x \in V_0$ is open independently with
    \begin{equation}
        \label{eq:finite_size_criterion_proof_3}
        \P_{\eta}
        \left(
            x \text{ is open}
        \right)
        \geq
        1 
        - 
        \delta.
    \end{equation}
    Suppose that $x,y \in V_0$ are net vertices with $d_G(x,y) \leq 16AR$, so that there is an edge between $x$ and $y$ in $N(V_0)$. 
    Since $V_0$ is $4R$-separated, we have that $d_G(B(x,R),B(y,R)) \geq 2R$.
    Further, there exists some $z \in V$ such that $B(x,R),B(y,R) \subseteq B(z, 10AR)$.
    By the independence of edges in \LRP{}, by the transitivity of the graph and the kernel, by the assumption in \eqref{eq:finite_size_criterion_2}, and by the tower rule, an edge between two open sites $x \in V_0$ and $y \in V_0$ is open conditionally independently given $\giantcomp{x}{R}$ and $\giantcomp{y}{R}$ with
    \begin{multline}
    \label{eq:finite_size_criterion_proof_4}
        \P_{\eta}
            \left(
                (x,y) 
                \text{ is open}
            \mid
            x,y 
            \text{ are open}
        \right)
        =
        \P_{\beta}
        \left(
            \giantcomp{x}{R}
            \sim
            \giantcomp{y}{R}
            \mid
            \# 
            \giantcomp{x}{R}, 
            \#    
            \giantcomp{y}{R}
            \geq 
            \nu 
            \# 
            B(r)
        \right)
        \\ 
        =
        \E_{\beta}
        \left[
            \P_{\beta}
            \left(
                \giantcomp{x}{R}
                \sim
                \giantcomp{y}{R}
                \mid
                \giantcomp{x}{R}, 
                \giantcomp{y}{R}
            \right)
            \mid 
            x,y 
            \text{ are open}
        \right]
        \geq 
        1 
        -
        \delta.
    \end{multline}
    It follows that there exists an independent site-bond percolation process $\hat{\eta}$ on $N(V_0)$ where a site is \textbf{open} in $\hat{\eta}$ independently with probability $1 - \delta$ and a bond between two open sites is \textbf{open} in $\hat{\eta}$ conditionally independently with probability $1 - \delta$, and $\eta$ stochastically dominates $\hat{\eta}$.
    Let $\Delta$ denote the maximum degree of the net graph $N(V_0)$, so that $\Delta \leq \# B(16AR) / \#B(2R) \leq C_G (8A)^d/ c_G$ and in particular $\delta< 1/(18 \Delta)$.  
    By standard arguments in percolation \cite{liggett_domination_1997}, $\hat{\eta}$ stochastically dominates {\color{blue} $\tilde{\eta}$} an independent bond percolation process on $N(V_0)$ with parameter $q = 1 - (2 \Delta)(3\delta)$. 
    By the definition of $A$-controlled nets, the net graph $N(V_0)$ contains an $\L^2$ subgraph.
    When $\delta = c_G/(18 (8A)^d C_G)$ then $q > p_c(\L^2) = 1/2 \geq p_c(N(V_0))$ and by Strassen's theorem on stochastic domination \cite{strassen_existence_1965,lindvall_strassens_1999}, this implies that $\beta > \beta_c(G,J)$.
    The stochastic domination also gives 
    \begin{equation}
        \P_{\beta}
        \left(
            \giantone{R} 
            \leftrightarrow 
            \infty 
            \right)
        \geq 
        \P_{{\eta}}
        \left(
            o
            \leftrightarrow
            \infty
        \right)
        \geq 
        \P_{\hat{\eta}}
        \left(
            o
            \leftrightarrow
            \infty
        \right)
        \geq 
        \P_{\tilde{\eta}}
        \left(
            o
            \leftrightarrow
            \infty
        \right)
    \end{equation}
    and again since the net graph contains an $\L^2$ subgraph a standard Peierls argument gives that $\P_{\tilde{\eta}} \left(o \leftrightarrow \infty \right) \to 1$ as $\delta \to 0$ and hence $q \to 1$.
\end{proof}
\section{\texorpdfstring{Truncation and continuity of $\theta$}{Truncation and continuity of θ}}
\label{sec:truncation_and_continuity}
In this section we assume \thref{prop:biskup} on the local existence of the giant to prove \thref{thm:truncation_problem} on the truncation problem and to prove \thref{thm:continuity_percolation_density} on the continuity of the phase transition.
The proofs follow from the finite-size criterion identified in \thref{prop:finite_size_criterion}.

\begin{proof}[Proof of \thref{thm:truncation_problem} assuming \thref{prop:biskup}]
    By \thref{prop:fixed_unif_control_nets} there exists $A \geq 1$ such that the graph $G$ admits $A$-controlled nets. 
    For some $R \in \N$, consider the truncated kernel $\widetilde{J}$ defined by
    \begin{equation}
        \label{eq:truncated_kernel_proof}
        \widetilde{J}
        =
        J(x,y) 
        \charf
        \left(
            d_G(x,y) \leq 20 AR
        \right).
    \end{equation}
    Let $\delta$ be as in \thref{prop:finite_size_criterion}.
    By \thref{prop:biskup} there exist $\nu > 0$ and $R$ sufficiently large such that 
    \begin{equation}
    \label{eq:biskup_applied_1}
        \P_{\beta,\widetilde{J}}
        \left(
            \# \giantone{R}
            \geq 
            \nu
            \# 
            B(R)
        \right)
        =
        \P_{\beta,J}
        \left(
            \# \giantone{R}
            \geq 
            \nu
            \#
            B(R)
        \right)
        \geq 
        1 
        - 
        \delta.
    \end{equation}
    Since $J$ and $\widetilde{J}$ agree on $B(10AR)$, for $S_1,S_2 \subseteq B(10AR)$ with $d_G(S_1,S_2) \geq 2R$ and $\#S_1, \#S_2 \geq \nu \# B(R)$ we have
    \begin{equation}
        \P_{\beta,\widetilde{J}}
        \left(
            S_1 
            \sim 
            S_2 
        \right) 
        =
        \P_{\beta,J}
        \left(
            S_1 
            \sim 
            S_2 
        \right) 
        =
        1
        -
        \P_{\beta,J}
        \left(
            S_1 
            \not\sim 
            S_2 
        \right) 
        =
        1
        -
        \exp
        \left(
            - 
            \frac{
                c_J 
                \beta
                (
                    \nu \# B(R)
                )^2
            }{
                (20AR)^{d \alpha} 
            }
        \right)
        \geq 
        1
        -
        \delta
    \end{equation}
    for $R$ sufficiently large.
    It follows from \thref{prop:finite_size_criterion} that $\beta > \beta_c(\widetilde{J})$.
\end{proof}

\begin{proof}[Proof of \thref{thm:continuity_percolation_density}]
    Let $\beta \in \R_+$ with $\theta(\beta) > 0$ and let $\delta$ be as in \thref{prop:finite_size_criterion}.
    By \thref{prop:biskup} there exist $\nu > 0$ and $R$ sufficiently large such that 
    \begin{equation}
        \label{eq:finite_criterion_check}
        \P_{\beta}
        \left(
            \# \giantone{R}
            \geq 
            \nu
            \# 
            B(R)
        \right)
        \geq 
        1
        -
        \delta/2.
    \end{equation}
    The event $\# \giantone{R} \geq \nu \# B(R)$ depends only on the finite set of edges with both endpoints in $B(R)$, so its probability is a continuous function of $\beta$.
    We may then choose $\epsilon > 0$ sufficiently small so that
    \begin{equation}
        \label{eq:coupling_disagreement}
        \P_{\beta - \epsilon}
        \left(  
             \# \giantone{R}
            \geq 
            \nu
            \# 
            B(R)
        \right)
        \geq 
        \P_{\beta}
        \left(  
             \# \giantone{R}
            \geq 
            \nu
            \# 
            B(R)
        \right)
        -
        \delta/2
        \geq 
        1 
        - 
        \delta.
    \end{equation}
    Suppose first that $d \geq 2$. 
    For any two sets $S_1, S_2 \subseteq B(10AR)$ with $d_G(S_1,S_2) \geq 2R$ and $\#S_1, \#S_2 \geq \nu \# B(R)$ we have that
    \begin{equation}
        \P_{\beta-\varepsilon}
        \left(
            S_1 
            \sim 
            S_2 
        \right) 
        =
        1
        -
        \P_{\beta-\varepsilon}
        \left(
            S_1 
            \not\sim 
            S_2 
        \right) 
        =
        1
        -
        \exp
        \left(
            - 
            \frac{
                c_J 
                (\beta-\varepsilon)
                (
                    \nu \# B(R)
                )^2
            }{
                (20AR)^{d \alpha} 
            }
        \right)
        \geq 
        1
        -
        \delta
    \end{equation}      
    for $R$ sufficiently large.
    It follows from \thref{prop:finite_size_criterion} that $\beta-\epsilon > \beta_c$, so that $\theta(\beta-\epsilon) > 0$. 
    If $d = 1$, the bound in \eqref{eq:coupling_disagreement} is enough to initialise the induction in the proof of \thref{prop:almost_linear} to obtain $\beta-\epsilon > \beta_c$, so that again $\theta(\beta-\epsilon) > 0$.
     By the natural coupling for different values of $\beta$, the percolation density is monotone in $\beta$ and hence $\theta(\beta+\epsilon) > 0$. 
    This shows that the set $\{\beta \in \R_+ : \theta(\beta) > 0 \}$ is open for all $d\ge 1$. 
    In particular, the critical parameter $\beta_c$ is a boundary point of this set, which implies that $\theta(\beta_c) = 0$.
\end{proof}

\section{\texorpdfstring{Locality of $\beta_c$ and joint continuity of $\theta$}{Locality of \textbackslash{}beta\textunderscore{}c and joint continuity of \textbackslash{}theta}}
\label{sec:locality}
In this section we assume \thref{prop:biskup,prop:second_largest} on the local existence-and-uniqueness of the giant to prove locality of $\beta_c$ and the joint continuity of $\theta$ (\thref{thm:lrp_locality,thm:theta_locality}).
We describe the natural coupling on the coupling radius, allowing us to locally couple \LRP{} configurations. 

Let $(G_n,J_n)_{n \in \N}$ converge locally to $(G,J)$.
Recall from \eqref{eq:coupling_radius_definition} that, for $\epsilon > 0 $, the coupling radius $R_n = R_n(G_n,G,\epsilon)$ is the largest integer such that both 
\begin{equation}
\label{eq:coupling-recall}
    B_{G_n}(2 R_n) 
    \cong 
    B_{G}(2 R_n) 
    \quad 
    \text{ and }
    \sum_{x,y \in B(R_n)} 
    \lvert 
        J_n(x,y)
        -
        J(x,y)
    \rvert
    \leq 
    \epsilon.
\end{equation}
This definition allows for a coupling between the \LRP{} configurations inside the ball $B(R_n)$, as follows. 
We refer to this coupling as the \textbf{natural coupling on the coupling radius}.
For a configuration $\omega$ and a set $S\subset V$, we write $\omega[S]$ for the spanned subgraph of $\omega$ on the set $S$.

\begin{lemma}[Coupling on the coupling radius]
\thlabel{lem:coupling-on-radius}
Let $(G,J)$ be a transitive graph and a transitive kernel, and let $(G_n,J_n)_{n \in \N}$ be a sequence of transitive graphs and kernels converging locally to $(G,J)$.
Let $\beta \in \R_+$ and let $(\beta_n)_{n \in \N}$.
Let $\omega_n$ be a configuration sampled from $\P_{\beta_n,G_n,J_n}$ and let $\omega$ be a configuration sampled from $\P_{\beta,G,J}$.
Then there exists a coupling $\P$ of $\omega_n$ and $\omega$ on $B(R_n)$ so that for any $R\le R_n$ satisfying \eqref{eq:coupling-recall},
\begin{equation}
\label{eq:configuration_disagree_6}
    \P
    \Big(
    \omega_n\big[B(R)\big]
        \neq 
        \eta\big[B(R)\big]
    \Big)
    \leq
    \beta \varepsilon + |\beta_n-\beta|\sum_{x,y \in B(R)} J(x,y).
\end{equation}
Suppose that $R \in \N$ is fixed and that $\beta_n \to \beta$.
Then for $\xi > 0$ and for any event $\AA_{R}$ which is locally determined by $B(R)$, for all $n$ sufficiently large
\begin{equation} \label{eq:configuration_disagree_7}
   \left\vert
       \P_{\beta_n,G_n,J_n}
        \left(
            \AA_{R}
        \right)-\P_{\beta,G,J}
        \left(
            \AA_{R}
        \right) 
    \right\vert
    \le
    \xi.
\end{equation}
\end{lemma}

\begin{proof}
    We let $\widetilde{\omega}_n$ be a configuration sampled from $\P_{\beta_n,G,J}$, and 
    we let $\P$ be the coupling where each edge $(x,y)$ with $x,y \in B(R_n)$ is open in each of $\omega_n, \widetilde{\omega}$, and $\omega$ with probability the minimum of $\P_{\beta_n,G_n,J_n}(x \sim y), \P_{\beta_n,G,J}(x \sim y)$, and $\P_{\beta,G,J}(x \sim y)$.
    We let $\omega_n, \widetilde{\omega}$, and $\omega$ be configurations sampled from the respective marginals.
    It follows from Markov's inequality, the kernel bound in \eqref{eq:standard_kernel_bound}, and the definition of local convergence that for any $R\le R_n$ satisfying \eqref{eq:coupling-recall}
    \begin{equation}
    \begin{aligned}
    \label{eq:configuration_disagree_1}
        \P
        \left(
            \omega_n
            \left[
                B(R_n)
            \right]
            \neq 
            \omega
            \left[
                B(R_n)
            \right]
        \right)
        & \leq 
        \E
        \left[
            \#
            \{
                x,y \in B(R_n)
                :
                \omega_n((x,y))
                \neq
                \widetilde{\omega}_n((x,y))
            \}
        \right] 
        \\
        & 
        \qquad \qquad \qquad+ 
        \E
        \left[
        \#
        \{
            x,y \in B(R_n)
            :
            \widetilde{\omega}((x,y))
            \neq 
            \omega((x,y))
        \}
        \right]
        \\
        & \leq 
        \sum_{x,y \in B(R_n)}
        \beta| J_n(x,y) - J(x,y)| + 
        \sum_{x,y \in B(R_n)} |\beta_n-\beta|J(x,y)
    \end{aligned}
    \end{equation}
    By \eqref{eq:coupling-recall} the first term is at most $\beta \epsilon$, concluding the proof.
    The bound in \eqref{eq:configuration_disagree_7} follows from \eqref{eq:configuration_disagree_6} by noting that by the definition of local convergence we may choose $n$ sufficiently large such that both $R \leq R_n$ and $\beta \epsilon + |\beta_n-\beta| \sum_{x,y \in B(R)}J(x,y) < \xi$, and hence
    \begin{equation}
        \left| 
       		\P_{\beta,G_n,J_n}
        	\left(
            	\AA_{R}
    	    \right)
    	    -
    	    \P_{\beta,G,J}
        	\left(
            	\AA_{R}
    	    \right)
        \right|
        \leq 
        \P
        \left(
            \omega_n
            \left[
                B(R)
            \right]
            \neq 
            \omega
            \left[
                B(R)
            \right]
        \right)
        <
        \xi
    \end{equation}
    which concludes the proof.
\end{proof}
\subsection{Locality for \LRP{}}
For $G$ a locally finite transitive graph with a fixed origin $o$, $J : V \times V \to \R_+$ a transitive kernel, and $S \subset V$ a finite set containing the origin $o$, let 
\begin{equation}\label{eq:phi-S-def}
    \varphi_{\beta,G,J}(S)
    =
    \sum_{x \in S}
    \sum_{y \not\in S}
    \P_{\beta,G,J}
    \left(
        x 
        \sim 
        y
    \right)
    \P_{\beta,G,J}
    \big(
        o
        \xleftrightarrow{S}
        x
    \big).
\end{equation}
Duminil-Copin and Tassion prove in \cite{duminil-copin_new_2016} that
\begin{equation}
    \beta_c(G,J)
    =
    \sup
    \left\{ 
        \beta 
        \geq 
        0  
        : 
        \varphi_{\beta,G,J}(S) < 1
        \text{ for some finite }
        S \subset V
        \text{ containing }
        o
    \right\}.
\end{equation}
We begin by proving lower semi-continuity.

\begin{proposition}
\thlabel{prop:lower_semi_continuity}
    Let $G$ be a locally finite transitive graph and let $J : V \times V \to \R_+$ be transitive and integrable.
    Let $(G_n,J_n)_{n \in \N}$ be a sequence of transitive graphs and kernels converging locally to $(G,J)$ with $(J_n)_{n \in \N}$ uniformly integrable.
    Then
    \begin{equation}
        \label{eq:lower_semi_continuity}
        \liminf_{n \to \infty} 
        \left(
            \beta_c(G_n,J_n)
        \right) 
        \geq 
        \beta_c(G,J). 
    \end{equation}
\end{proposition}

\begin{proof}
    Let $\beta < \beta_c(G,J)$ and let $\epsilon > 0$.
    By the uniform integrability \eqref{eq:tight-decay-Jn} there exists $R^\star$ such that
    \begin{equation}
        \label{eq:unif_decay_bound}
        \sum_{y \in G_n\setminus B_n(R^\star)}
        J_n(o,y)
        \leq 
        \epsilon/(2 \beta)
    \end{equation}
    for all $n$ sufficiently large.
    By \cite{duminil-copin_new_2016} there exists $S \subset V$ a finite set containing the origin $o$ such that $\varphi_{\beta,G,J}(S) < 1$.
    Let $R = \diam(S)+R^\star$.
    By the definition of local convergence $B_{G_n}(2R) \cong B_G(2R)$ and in particular $S \in B(R-R^\star)$ for all $n$ sufficiently large.
    By definition
    \begin{equation}
        \label{eq:def_varphi_n}
        \varphi_{\beta,G_n,J_n}(S)
        =
        \sum_{x \in S}
        \P_{\beta,G_n,J_n}
        \left(
            o
            \xleftrightarrow{S}
            x
        \right)
        \sum_{y \in G_n \setminus S}
        \P_{\beta,G_n,J_n}
        \left(
            x 
            \sim 
            y
        \right).
    \end{equation}
    By the natural coupling between $\P_{\beta,G_n,J_n}$ and $\P_{\beta,G,J}$ in $B(R)$ in \thref{lem:coupling-on-radius} it holds that
    \begin{equation}
        \label{eq:S-set-eps}
        \P_{\beta,G_n,J_n}
        (
            o
            \xleftrightarrow{S}
            x
        )
        \leq 
        \P_{\beta,G,J}
        (
            o
            \xleftrightarrow{S}
            x
        )
        + 
        \epsilon
    \end{equation}
    and 
    \begin{equation}
        \label{eq:direct_connection_bound}
        \sum_{y \in B(R) \setminus S}
        \P_{\beta,G_n,J_n}
        \left(
            x \sim y 
        \right)
        \leq 
        \epsilon/2
        +
        \sum_{y \in B(R) \setminus S}
        \P_{\beta,G,J}
        \left(
            x \sim y 
        \right)
    \end{equation}
    for all $x \in S$ and for all $n$ sufficiently large.
    Since $B(x,R^\star) \subseteq B(R)$ for all $x \in S$, it follows that
    \begin{multline}
        \sum_{y \in G_n \setminus S}
        \P_{\beta,G_n,J_n}
        \left(
            x \sim y 
        \right)
        =
        \sum_{y \in G_n\setminus B(R)}
        \P_{\beta,G_n,J_n}
        \left(
            x \sim y 
        \right)
        +
        \sum_{z \in B(R) \setminus S}
        \P_{\beta,G_n,J_n}
        \left(
            x \sim z 
        \right)
        \\
        \leq
        \label{eq:direct_connection_1}
        \epsilon
        +
        \sum_{y \in B(R) \setminus S}
        \P_{\beta,G,J}
        \left(
            x \sim y 
        \right)
        \leq
        \epsilon
        +
        \sum_{y \in G \setminus S}
        \P_{\beta,G,J}
        \left(
            x \sim y 
        \right)
    \end{multline}
    where in the first inequality we apply the the kernel bound in \eqref{eq:standard_kernel_bound} and the bound in \eqref{eq:unif_decay_bound} to the first term and we apply the bound in \eqref{eq:direct_connection_bound} to the second term, and in the second inequality we use that $B(R) \subset G$.
    The bounds in \eqref{eq:S-set-eps} and \eqref{eq:direct_connection_1} in \eqref{eq:def_varphi_n} give, after distributing the products
    \begin{equation}
        \varphi_{\beta,G_n,J_n}(S)
        \leq
        \varphi_{\beta,G,J}(S)
        + 
        \epsilon
        \E_{\beta,J,G}[\deg(o)] 
        + 
        \epsilon 
        \P_{\beta, G,J}
        (S \subseteq \cluster) 
        + 
        \epsilon^2
        \#S
    \end{equation}
    since $\sum_{x \in S} \P_{\beta,G,J} (o \xleftrightarrow{S} x) = \P_{\beta, G,J}(S \subseteq \cluster)$ and $\sum_{y \in G \setminus S} \P_{\beta,G,J} \left(x \sim y \right) \leq \E_{\beta,G,J} \left[ \deg(o) \right]$.
    Clearly $\P_{\beta, G,J}(S \subseteq \cluster) < 1$, and $\E_{\beta,G,J} \left[ \deg(o) \right] < \infty$ since $J$ is assumed to be integrable.
    By construction $\varphi_{\beta,G,J}(S) < 1$ and we may choose $\epsilon$ sufficiently small such that $\varphi_{\beta,G_n,J_n}(S) < 1$ and in particular $\beta_c(G_n,J_n) > \beta$ for all $n$ for all $n$ sufficiently large, concluding the proof.
\end{proof}
  
The assumption of uniform integrability was used in \eqref{eq:direct_connection_1} to make $\P_{\beta,G_n,J_n}\left(S \sim G_n \setminus S \right)$ comparable to $\P_{\beta,G,J}\left(S \sim G \setminus S \right)$.
Due to the long-range nature of \LRP{} this probability is clearly not local, and the assumption of uniform integrability ensures that these non-local edges have a negligible effect on $\beta_c$.
Without uniform integrability the proof has no chance of success, and in fact obvious counterexamples arise.

\begin{example}[No locality without uniform integrability]
    \thlabel{ex:no_unif_integrability}
    Let $G$ be a transitive graph of polynomial growth with $d \geq 2$, and let $J : V \times V \to \R_+$ be the transitive kernel given by $J(x,y) = d_G(x,y)^{-d \alpha}$ for $\alpha > 1$. 
    Let $(G_n)_{n \in \N}$ be a sequence of transitive graphs with $\limsup_{n \to \infty} \dim(G_n) < d$, let $M_n$ be the maximal radius for which the balls in $G_n$ and $G$ are isomorphic as rooted graphs, and let $(J_n)_{n \in \N}$ be the sequence of transitive kernels given by
    \begin{equation}
        J_n(x,y) 
        =
        \begin{cases}
            J(x,y) 
            &
            d_G(x,y) 
            \leq 
            \lfloor 
                M_n/2 
            \rfloor
            \\
            d_{G_n}(x,y)^{-\dim(G_n)}
            & 
            d_G(x,y) > \lfloor M_n/2 \rfloor
        \end{cases}
        .
    \end{equation}
    Then $(G_n,J_n)_{n \in \N}$ converges locally to $(G,J)$, with $0 < \beta_c(G,J) < \infty$ yet  $\beta_c(G_n,J_n) = 0$ for all $n$.
\end{example}

We now assume \thref{prop:biskup} to prove upper semi-continuity.

\begin{proposition}
    \thlabel{prop:upper_semi_continuity}
    Let $G$ be a transitive graph of polynomial growth with $d \in \N$, and suppose that $J : V \times V \to \R_+$ is a transitive kernel satisfying $J(x,y) = \Omega(d_G(x,y)^{-d \alpha})$ for $\alpha \in (0, 2)$. 
    Let $(G_n,J_n)_{n \in \N}$ be a sequence of transitive graphs and kernels converging locally to $(G,J)$ with $\liminf_{n \to \infty} \dim(G_n) \geq 2$.
    Then
    \begin{equation}    
        \label{eq:locality_upper_semi_continuity}
        \limsup_{n \to \infty} 
        \left(
            \beta_c(G_n,J_n)
        \right) \leq \beta_c(G,J). 
    \end{equation}
\end{proposition}

\begin{proof}[Proof of \thref{prop:upper_semi_continuity}]
    Let $\beta > \beta_c(G,J)$ and consider the limit graph $G$. 
    Let $\delta$ be as in \thref{prop:finite_size_criterion}.
    By \thref{prop:biskup} there exists $\nu > 0$ and $R$ sufficiently large such that 
    \begin{equation}
    \label{eq:biskup_applied_2}
        \P_{\beta,G,J}
        \left(
            \giantone{R}
            \geq 
            \nu
            \#B_{G}(R)
        \right)
        \geq 
        1
        -
        \delta/2.
    \end{equation}
    By \thref{prop:unif_control_nets} there exists $A \geq 1$ such that $G_n$ has $A$-controlled nets for all $n$ sufficiently large.
    For $S_1,S_2 \in B(10 A R)$ with $d_G(S_1,S_2) \geq 2R$ and $\# S_1, \# S_2 \geq \nu \# B(R)$, we choose $R$ sufficiently large such that
    \begin{equation}
        \P_{\beta,G,J}
        \left(
            S_1 
            \sim 
            S_2 
        \right)
        \geq
        1
        -
        \exp
        \left(
            - 
            \frac{
                c_J 
                \beta
                (
                    \nu \# B(R)
                )^2
            }{
                (16AR)^{d \alpha} 
            }
        \right)
        \geq 
        1
        -
        \delta/2.
    \end{equation}
    By the definition of local convergence we have $B_{G_n}(20AR) \cong B_G(20AR)$ for all $n$ sufficiently large.
    By the natural coupling between $\P_{\beta,G_n,J_n}$ and $\P_{\beta,G,J}$ in $B(R)\subseteq B(20AR)$ in \thref{lem:coupling-on-radius}, we have
    \begin{equation}
        \label{eq:locality_proof_clusters_agree}
        \P_{\beta,G_n,J_n}
        \left(
            \giantone{R}
            \geq 
            \nu
            \#B(R)
        \right)
        \geq 
        \P_{\beta,G,J}
        \left(
            \giantone{R}
            \geq 
            \nu
            \#B(R)
        \right)
        -
        \delta/2
        \geq 
        1 
        - 
        \delta.
    \end{equation}
    Again by the natural coupling on $B(20AR)$, for all $n$ sufficiently large
    \begin{equation}
        \P_{\beta,G_n,J_n}
        \left(
            S_1
            \sim 
            S_2 \mid \#S_1, \#S_2
        \right)
        \geq
        \P_{\beta,G,J}
        \left(
            S_1
            \sim 
            S_2  \mid \#S_1, \#S_2
        \right)
        -
        \delta/2
        \geq
        1
        -
        \delta,
    \end{equation}
    which verifies the finite-size criterion in \thref{prop:finite_size_criterion} and gives $\beta > \beta_c(G_n,J_n)$.
\end{proof}

\begin{proof}[Proof of \thref{thm:lrp_locality}]
    Upper semi-continuity follows from \thref{prop:upper_semi_continuity}.
    When $\alpha \leq 1$ then \thref{lem:integrability} states that $\E_{\beta,G,J}[\deg(o)] = \infty$ and $\beta_c(G,J) = 0$ so that trivially $\liminf_{n \to \infty} \beta_c(G_n) \geq 0$.
    When $\alpha > 1$ and the sequence of kernels $(J_n)_{n \in \N}$ is uniformly decaying, lower semi-continuity follows from \thref{prop:lower_semi_continuity}.
\end{proof}

\subsection{Joint continuity of the percolation probability}
We begin by proving upper semi-continuity.

\begin{proposition}
    \thlabel{prop:theta_upper_semi_continuity}
    Let $G$ be a transitive graph of polynomial growth, and suppose that $J : V \times V \to \R_+$ is a transitive kernel.
    Let $\beta \in \R_+$ and let $\beta_n \to \beta$.
    Let $(G_n,J_n)_{n \in \N}$ be a sequence of transitive graphs and kernels converging locally to $(G,J)$ with $(J_n)_{n \in \N}$ uniformly integrable. 
    Then
    \begin{equation}
        \label{eq:theta_upper_semi_continuity}
        \limsup_{n \to \infty} 
        \theta(\beta_n,G_n,J_n) 
        \leq 
        \theta(\beta,G,J).
    \end{equation}
\end{proposition}

\begin{proof}
    The uniform integrability of $(J_n)_{n \in \N}$ implies that $J$ is integrable, which in turn implies that $\theta(\beta, G,J)<1$ for all $\beta\ge 0$. 
    Let $\epsilon >0$.
    Since the number of finite subsets of $V$ is countable, there exists $m \in \N$ and a finite collection of sets $(A_i)_{i \leq m}$ with $o \in A_i \subset V$ for every $i \leq m$ satisfying
    \begin{equation}
        \label{eq:theta_upper_continuity_finite}
        \sum_{i \leq m}
        \P_{\beta,G,J}
        \left(
            \cluster
            =
            A_i
        \right)
        \geq 
        1 
        - 
        \theta(\beta,G,J)
        - \epsilon/2.
    \end{equation}
    By the uniform integrability of $(J_n)_{n \in \N}$ there exists $R \in \N$ such that both $R > 2 \max_{i \leq m} \diam(A_i)$ and for all $n$ sufficiently large
    \begin{equation}
        \label{eq:Markov-R-choice}
        \sum_{x \in G_n \setminus B_{G_n}(R/2)} 
        J_n(o,x)
        \leq
        \frac{
            \epsilon
        }{
            4\beta m \max_{i\le m}\#A_i
        }
    \end{equation}
    By the local convergence $B_{G_n}(2R) \cong B_G(2R)$ for all $n$ sufficiently large.
    Since $A_i \subseteq B(R)$, the event $\cluster = A_i$ is equivalent to $\cluster_R = A_i$ and $A_i \not\sim B(R)^c$ and by a union bound
    \begin{equation}   
        \label{eq:n-graph-to-connect}
        \P_{\beta_n,G_n,J_n}
        \left(
            \cluster 
            =
            A_i
        \right)
        \geq 
        \P_{\beta_n,G_n,J_n}
        \left(
            \cluster_R
            =
            A_i
        \right)
        -
        \P_{\beta_n,G_n,J_n}
        \left(
            A_i
            \sim
            B(R)^c
        \right)
    \end{equation}  
    for every $i \leq m$.
    We bound both terms in turn.
    By the natural coupling between $\P_{\beta_n,G_n,J_n}$ and $\P_{\beta,G,J}$ in $B(R)$ in \thref{lem:coupling-on-radius} we bound from below as
    \begin{equation}
    \label{eq:limit-to-connect}
        \P_{\beta_n,G_n,J_n}
        \left(
            \cluster_R
            =
            A_i
        \right)
        \geq 
        \P_{\beta,G,J}
        \left(
            \cluster_R
            =
            A_i
        \right)
        -
        \epsilon/(4m)
        \geq 
        \P_{\beta,G,J}
        \left(
            \cluster
            =
            A_i
        \right)
        -
        \epsilon/(4m)
    \end{equation}
    for all $n$ sufficiently large, where in the second inequality we use that $K=A_i$ implies $K_R = A_i$.
    By Markov's inequality together with the fact that $A_i \subseteq B(R/2)$, the transitivity of the kernel, the kernel bound in \eqref{eq:standard_kernel_bound}, and the bound in \eqref{eq:Markov-R-choice}, we have
    \begin{multline}
        \P_{\beta_n,G_n,J_n}
        \left(
            A_i
            \sim
            B(R)^c
        \right)
        \leq 
        \E_{\beta_n,G_n,J_n}
        \left[
            \#
            \{
                x 
                \in 
                A_i
                :
                x 
                \sim B(R)^c
            \}
        \right]
        \\
        \leq 
        \# 
        A_i
        \P_{\beta_n,G_n,J_n}
        \left(
            o
            \sim 
            B(R/2)^c
        \right)
        \leq 
        \frac{
            \epsilon 
            \# 
            A_i
        }{
            8 
            m 
            \max_{i \leq m}
            \#A_i
        }\le \varepsilon/(4m)
    \end{multline}
    for all $i \leq m$.
    Together with \eqref{eq:n-graph-to-connect} and \eqref{eq:limit-to-connect}, after summing over $i \leq m$ we arrive at
    \begin{equation}
    \label{eq:theta_upper_continuity_coupling_bounds} 
        \sum_{i \leq m}
        \P_{\beta_n,G_n,J_n}
        \left(
            \cluster
            =
            A_i
        \right)
        \geq 
        \sum_{i \leq m}
        \P_{\beta,G,J}
        \left(
            \cluster
            =
            A_i
        \right)
        - 
        \epsilon/2
        \geq
        1 
        - 
        \theta(\beta,G,J)
        - 
        \epsilon.
    \end{equation}
    and in particular after rearranging 
    \begin{equation}
        \label{eq:theta_upper_semi_continuity_conclusion}
        \theta(\beta_n,G_n,J_n)
        \leq 
        1 
        -
        \sum_{i \leq m}
        \P_{\beta_n,G_n,J_n}
        \left(
            \cluster
            =
            A_i
        \right)
        \leq 
        \theta(\beta,G,J)
        + 
        \epsilon
    \end{equation}
    for all $n$ sufficiently large.
    This concludes the proof.
\end{proof}

We prove the lower semi-continuity of the percolation density. 
\begin{proposition}[Lower semi-continuity]
    \thlabel{prop:theta_lower_semi_cont}
    Let $G$ be a transitive graph of polynomial growth with $d \geq 2$, and suppose that $J : V \times V \to \R_+$ is a transitive kernel satisfying $J(x,y) = \Theta(d_G(x,y)^{-d \alpha})$ with $\alpha \in (1,2)$. 
    Let $\beta \in \R_+$ and let $\beta_n \to \beta$.
    Let $(G_n,J_n)_{n \in \N}$ be a sequence of transitive graphs and kernels converging locally to $(G,J)$ with $\liminf_{n \to \infty} \dim(G_n) \geq 2$.
    Then 
    \begin{equation}
        \liminf_{n \to \infty}
        \theta
        \left(
            \beta_n,
            G_n,
            J_n
        \right)
        \geq 
        \theta
        \left(
            \beta,
            G,
            J
        \right).
    \end{equation}      
\end{proposition}

\begin{proof}
    If $\beta < \beta_c(G,J)$ then by definition $\theta(\beta,G,J) = 0$, and if $\beta = \beta_c(G,J)$ then by \thref{thm:continuity_percolation_density} we have $\theta(\beta,G,J) = 0$.
    In both cases the result is immediate, and we suppose now that $\beta > \beta_c(G,J)$. 
    For $n \in \N$ and $R \in \N$, by a union bound we can write
    \begin{multline}\label{eq:theta-return}
        \theta(\beta_n,G_n,J_n)
        =
        \P_{\beta_n,G_n,J_n}
        \left(
            o
            \leftrightarrow
            \infty
        \right)
        \geq
        \P_{\beta_n,G_n,J_n}
        \left(
            o
            \in
            \giantone{R},
            \giantone{R}
            \leftrightarrow
            \infty
        \right)
        \\
        \geq 
        \P_{\beta_n,G_n,J_n}
        \left(
            o
            \in
            \giantone{R}
        \right)
        -
        \P_{\beta_n,G_n,J_n}
        \left(
            \giantone{R}
            \not\leftrightarrow
            \infty
        \right).
    \end{multline}
    Let $\epsilon > 0$.   
    By \thref{prop:o_in_kinfty} we can choose $R$ sufficiently large such that in the limit graph
    \begin{equation}
        \label{eq:o_in_limit}
        \P_{\beta,G,J}
        \left(
            o
            \in
            \giantone{R}
        \right)
        \geq 
        \theta(\beta,G,J)
        -
        \epsilon/4.
    \end{equation}
    Let $\delta > 0$.
    By \thref{prop:biskup} there exists $\nu > 0$ and $R$ sufficiently large such that 
    \begin{equation}
    \label{eq:biskup_applied_3}
        \P_{\beta,G,J}
        \left(
            \giantone{R}
            \geq 
            \nu
            \#B_{G}(R)
        \right)
        \geq 
        1
        -
        \delta
        /
        2.
    \end{equation}
    By \thref{prop:unif_control_nets} there exists $A \geq 1$ such that $G_n$ has $A$-controlled nets for all $n$ sufficiently large.
    For $S_1,S_2 \in B(10 A R)$ with $d_G(S_1,S_2) \geq 2R$ and $\# S_1, \# S_2 \geq \nu \# B(R)$, we choose $R$ sufficiently large such that
    \begin{equation}
        \P_{\beta,G,J}
        \left(
            S_1 
            \sim 
            S_2 
        \right)
        \geq
        1
        -
        \exp
        \left(
            - 
            \frac{
                c_J 
                \beta
                (
                    \nu \# B(R)
                )^2
            }{
                (16AR)^{d \alpha} 
            }
        \right)
        \geq 
        1
        -
        \delta/2.
    \end{equation}
    By the definition of local convergence we have $B_{G_n}(20AR) \cong B_G(20AR)$ for all $n$ sufficiently large, and by the natural coupling between $\P_{\beta,G_n,J_n}$ and $\P_{\beta,G,J}$ in $B(R)\subseteq B(20AR)$ we have
    \begin{equation}
        \P_{\beta_n,G_n,J_n}
        \left(
            o
            \in
            \giantone{R}
        \right)
        \geq 
        \P_{\beta,G,J}
        \left(
            o
            \in
            \giantone{R}
        \right)
        -
        \epsilon/4
        \geq 
        \theta(\beta,G,J)
        -
        \epsilon/2
    \end{equation}
    as well as
    \begin{equation}
        \P_{\beta,G_n,J_n}
        \left(
            \giantone{R}
            \geq 
            \nu
            \#B(R)
        \right)
        \geq 
        \P_{\beta,G,J}
        \left(
            \giantone{R}
            \geq 
            \nu
            \#B(R)
        \right)
        -
        \delta/2
        \geq 
        1 
        - 
        \delta
    \end{equation}
    and
    \begin{equation}
        \P_{\beta,G_n,J_n}
        \left(
            S_1
            \sim 
            S_2
        \right)
        \geq
        \P_{\beta,G,J}
        \left(
            S_1
            \sim 
            S_2
        \right)
        -
        \delta/2
        \geq
        1
        -
        \delta
    \end{equation}
    for all $n$ sufficiently large.
    It follows from \thref{prop:finite_size_criterion} that for $\delta$ sufficiently small
    \begin{equation}
        \P_{\beta_n,G_n,J_n}
        \left(
            \giantone{R}
            \not\leftrightarrow
            \infty
        \right)
        \leq 
        \epsilon/2.
    \end{equation}
    Using this and \eqref{eq:o_in_limit} in 
    \eqref{eq:theta-return} gives $\theta\left(\beta_n,G_n,J_n\right) \geq \theta\left(\beta,G,J\right) - \epsilon$ for all $n$ sufficiently large, concluding the proof.
\end{proof}

\begin{proof}[Proof of \thref{thm:theta_locality}]
    Upper semi-continuity follows from \thref{prop:theta_upper_semi_continuity}. 
    If $\alpha > 1$ then lower semi-continuity follows from \thref{prop:theta_lower_semi_cont}.
     If $\alpha \leq 1$ and $\liminf_{n \to \infty} \beta_n > 0$, then $\theta(\beta,G,J) = 0$ and lower semi-continuity is immediate.
\end{proof}

\begin{remark}
    The upper semi-continuity of $\beta_c$ also follows from \thref{thm:theta_locality}: if $\beta > \beta_c(G,J)$ then $\theta(\beta,G,J) > 0$ and since $\theta(\beta,G_n,J_n) \to \theta(\beta,G,J)$ then $\theta(\beta,G_n,J_n) > 0$ for all $n$ sufficiently large, so that in particular $\beta > \beta_c(G_n,J_n)$ for all $n$ sufficiently large.
\end{remark}
\section{Smoothness of percolation characters}
\label{sec:smoothness} 
In this section we assume \thref{thm:cluster_size_decay} to obtain \thref{cor:characters_smoothness} on the smoothness of percolation characters.
For $n,m \in \N$, we let $\HH_{n}^{m}$ be the set of subgraphs of $G$ containing the origin with $n$ vertices and $m$ edges.
Note that $m \leq n^2$.
Recall that $\norm{J} = \sum_{x \in V \setminus \{o\}} J(o,x)$ and that $\partial_E(H)$ denotes the edge boundary of $H$.
We also recall the multinomial formula for the differentiation of products: for a collection $f_1,\ldots,f_m$ of smooth functions of $\beta$, we have
\begin{equation}
    \label{eq:multinomial_formula}
    \frac{d^k}{d\beta^k} 
    \left[
        \prod_{i=1}^{m} 
        f_i(\beta) 
    \right]
    = 
    \sum_{\alpha_1 + \ldots + \alpha_m = k} 
    \binom{k}{\alpha_1,\ldots,\alpha_m}
    \prod_{i=1}^m 
    f_i^{(\alpha_i)}(\beta).
\end{equation}

\begin{proof}[Proof of \thref{cor:characters_smoothness}]
    We want to show that all derivatives of $\E_{\beta} \left[ F(\# \cluster) \charf \left(\# \cluster < \infty \right) \right]$ exist and are finite.
    Rewriting the expectation gives
    \begin{equation}
        \E_{\beta}
        \left[
            F(\# \cluster)
            \charf
            \left(
                \# 
                \cluster
                < \infty
            \right)
        \right]
        =
        \sum_{n = 1}^{\infty}
        \sum_{m = n - 1}^{n^2}
        \sum_{H \in \HH_{n}^{m}}
        F(H)
        \P_{\beta}
        \left(
            \cluster
            =
            H
        \right).
    \end{equation}
    Let $P$ be a polynomial such that $\vert F(H) \vert \leq P(\#H)$.
    It is sufficient to show that for all $\ell\in \N$
    \begin{equation}
        \label{eq:wts_absolute_convergence}
        \sum_{n = 1}^{\infty}
        \sum_{m = n - 1}^{n^2}
        \sum_{H \in \HH_{n}^{m}}
        \left\vert
        F(H)
        \frac{d^{\ell}}{d\beta^\ell}
        \P_{\beta}
        \left(
            \cluster
            =
            H
        \right)
        \right\vert
        \leq 
        \sum_{n = 1}^{\infty}
        P(n)
        \sum_{m = n - 1}^{n^2}
        \sum_{H \in \HH_{n}^{m}}
        \left\vert
        \frac{d^{\ell}}{d\beta^\ell}
        \P_{\beta}
        \left(
            \cluster
            =
            H
        \right)
        \right\vert
        < 
        \infty.
    \end{equation}
    %
    For $\ell,m,n \in \N$ with $n > \ell$ and $H \in \HH_n^m$, we begin by showing that 
    \begin{equation}
        \label{eq:intermed_wts}
        \left\vert
        \frac{d^{\ell}}{d\beta^\ell}
        \P_{\beta}
        \left(
            \cluster
            =
            H
        \right)
        \right\vert
        \leq 
        \left(
            2
            n^3
            \left(
                \norm{J}
                + 
                \max(1,\beta^{-1})
            \right)
        \right)^{\ell}
        \P_{\beta}
        \left(
            \cluster
            =
            H
        \right).
    \end{equation}
    Let $e_1,\ldots,e_m$ be the set of edges present in $H$.
    By definition 
    \begin{equation}
        \label{eq:def_cluster}
        \P_{\beta}
        \left(
            \cluster 
            = 
            H
        \right)
        =
        \exp
        \left(
            -
            \beta
            J(\partial_E(H))
        \right)
        \prod_{i = 1}^m
        \left(
            1
            -
            \exp
            \left(
                - 
                \beta
                J(e_i)
            \right)
        \right),
    \end{equation}  
    and writing $\phi_H(\beta) = \prod_{i = 1}^m \left(1 - \exp \left(-  \beta J(e_i) \right) \right)$ the formula in  \eqref{eq:multinomial_formula} for $m=2$ gives
    \begin{align}
        \label{eq:cluster_equals_1}
        \frac{d^{\ell}}{d\beta^\ell}
        \P_\beta(\cluster=H)
        & =
        \sum_{k = 0}^{\ell}
        \binom{\ell}{k}
        \left(
            - 
            J(\partial_E(H))
        \right)^{\ell - k}
        \exp
        \left(
            -
            \beta
            J(\partial_E(H))
        \right)
        \frac{d^{k}}{d\beta^k}
        \phi_H(\beta).
    \end{align}
    The multinomial formula \eqref{eq:multinomial_formula} with $f_i(\beta) = 1 - \exp(-\beta J(e_i))$ gives, for all $k\le \ell$
    \begin{align}
        \frac{d^{k}}{d\beta^k}
        \phi_H(\beta)
        & =
        \sum_{\alpha_1 + \ldots + \alpha_m = k} 
        \binom{k}{\alpha_1,\ldots,\alpha_m}
        \prod_{i=1}^m 
        \left(
            -
            J(e_i)
        \right)^{\alpha_i}
        \left(
            \charf
            \left(
                \alpha_i
                =
                0
            \right)
            -
            \exp
            \left(
                -
                \beta
                J(e_i) 
            \right)
        \right)
        \\
        \label{eq:factor_out_phi}
        & =
        \phi_H(\beta)\cdot 
        k!
        \sum_{\alpha_1 + \ldots + \alpha_m = k} 
        \prod_{i=1}^m 
        \frac{
            \left(
                -
                J(e_i)
            \right)^{\alpha_i}
            \left(
                \charf
                \left(
                    \alpha_i
                    =
                    0
                \right)
                -
                \exp
                \left(
                    -
                    \beta
                    J(e_i) 
                \right)
            \right)
        }{
            \alpha_i !
            \left(
                1 
                -
                \exp
                \left(
                    -
                    \beta
                    J(e_i) 
                \right)
            \right)
        },
    \end{align}
    where we factor out a term $\phi_H(\beta)$ and expand the multinomial coefficient.
    Note that
    \begin{equation}
        \left\vert
        \frac{
            \left(
                -
                J(e_i)
            \right)^{\alpha_i}
            \left(
                \charf
                \left(
                    \alpha_i
                    =
                    0
                \right)
                -
                \exp
                \left(
                    -
                    \beta
                    J(e_i) 
                \right)
            \right)
        }{
            1 
            -
            \exp
            \left(
                -
                \beta
                J(e_i) 
            \right)
        }
        \right\vert
        \leq 
        \alpha_i !
        \max(1,\beta^{-\alpha_i})
    \end{equation}
    for all $\alpha_i \in \N$, and hence
    \begin{align}
        \label{eq:yet_another_derivative}
        \left\vert
            \frac{d^{k}}{d\beta^k}
            \phi_H(\beta)
        \right\vert
        \leq 
        \phi_H(\beta)
        k!
        \sum_{\alpha_1 + \ldots + \alpha_m = k} 
        \max(1,\beta^{-k})
        =
        \phi_H(\beta)
        \max(1,\beta^{-1})^k
        \prod_{j=1}^k
        \left(
            m
            +
            j
            -
            1
        \right)
    \end{align}
    where in the equality we use the balls-and-bars formula and we expand the binomial coefficient to cancel with the $k!$ factor.
    Combining \eqref{eq:yet_another_derivative} with \eqref{eq:cluster_equals_1}, the definition of $\P_{\beta}(\cluster = H)$ in \eqref{eq:def_cluster}, and using that $\vert J(\partial_E(H)) \vert \leq n \norm{J}$ yields that 
    \begin{equation}    
        \label{eq:intermed_k_eq_h}
        \left\vert
            \frac{d^{\ell}}{d\beta^\ell}
            \P_{\beta}
            \left(
                \cluster
                =
                H
            \right)
        \right\vert
        \leq 
        n^{\ell}
        \P_\beta(\cluster=H)
        \sum_{k = 0}^{\ell}
        \binom{\ell}{k}
        \norm{J}^{\ell - k}
        \max(1,\beta^{-1})^k
        \prod_{j=1}^k
        \left(
            m
            +
            j
            -
            1
        \right).
    \end{equation}
    For all $n > \ell$ it holds that $m \geq \ell$, and since $k \leq \ell$ and $m \leq n^2$ we have $\prod_{j = 1}^{k} \left(m + j - 1 \right) \leq (2m)^{\ell} \leq (2n^2)^{\ell}$.
    Applying this bound to \eqref{eq:intermed_k_eq_h} and recognising the binomial formula proves the inequality in \eqref{eq:intermed_wts}.
    Decomposing the sum in \eqref{eq:wts_absolute_convergence} over $n$ and applying the inequality in \eqref{eq:intermed_wts} gives 
    \begin{multline}
        \label{eq:binom_exp}
        \sum_{n = 1}^{\infty}
        P(n)
        \sum_{m = n - 1}^{n^2}
        \sum_{H \in \HH_{n}^{m}}
        \left\vert
        \frac{d^{\ell}}{d\beta^\ell}
        \P_{\beta}
        \left(
            \cluster
            =
            H
        \right)
        \right\vert
        \\
        \leq
        c_1 
        +
        \left(
            2
            \left(
                \norm{J}
                +
                \max(1,\beta^{-1})
            \right)
        \right)^{\ell}
        \sum_{n = \ell + 1}^{\infty}
        P(n)
        n^{3 \ell}
        \sum_{m = n - 1}^{n^2}
        \sum_{H \in \HH_{n}^{m}}
        \P_{\beta}
        \left(
            \cluster
            =
            H
        \right)
    \end{multline}
    for some $c_1 = c_1(\ell) > 0$ which corresponds to the finite sum up to $\ell$.
    The triple sum in the right-hand side of \eqref{eq:binom_exp} corresponds to $\E_{\beta} \left[P(\# \cluster) (\# \cluster)^{3 \ell} \charf(\# \cluster > \ell) \right]$, and it follows from Fubini's theorem and \thref{prop:biskup} that
    \begin{align}
        \label{eq:sum_with_f}
        \E_{\beta} 
        \left[
            P(\# \cluster) 
            (\# \cluster)^{3 \ell} 
            \charf(\# \cluster > \ell) 
        \right]
        \leq 
        \sum_{k = \ell + 1}^{\infty}
        P(k)
        k^{3 \ell}
        \exp
        \left(
            - 
            \beta
            c_2 
            k^{2-\alpha}
        \right)
    \end{align}
    for some $c_2 > 0$. 
    Since $P$ is assumed to be a polynomial, the sum in \eqref{eq:sum_with_f} converges.
    This shows that the differentiated series in \eqref{eq:wts_absolute_convergence} is absolutely convergent, concluding the proof.
\end{proof}

\section{Coarse-graining: nesting Voronoi tiles}
\label{subsec:coarse_graining}
Recall the definition of nets from Section \ref{sec:renormalisation}.
Given a graph $G$ and a net $V_0 = \{x_j\}_{j \geq 0}$ on $G$, we define the \textbf{Voronoi tile} $\vor{x_j}$ centred at $x_j$ as
\begin{equation}
    \label{eq:voronoi_tile_definition}
    \vor{x_j}
    =
    \{
        v \in V 
        :
        d(v,x_j) < d(v,x_k) 
        \ \forall k < j,
        d(v,x_j) \leq d(v,x_k) 
        \ \forall k \geq j
    \}.
\end{equation}
The \textbf{Voronoi tiling} associated to the net $V_0$ is the set of Voronoi tiles $(\vor{x_j})_{j \geq 0}$.
In words, the Voronoi tile $\vor{x_j}$ consists of those vertices $v \in V$ which are closer to $x_j$ than to any other $x_k$, and we break ties when $v$ is equidistant from $x_j$ and $x_k$ by using an arbitrary well-ordering on the indices of tiles. 
Note that a Voronoi tiling partitions the graph. 
Further, we can bound the volume of Voronoi tiles of an $(a,b)$-net $V_0$ by observing that if $x_j,x_k$ are net vertices at distance exactly $\lceil a\rceil$, then the middle vertex on the path may only belong to one of the tiles. Any vertex at distance $b+1$ from $x_j$ is allocated to another Voronoi tile. For all $x_j \in V_0$ and all $a\ge 5$ it holds that
\begin{equation}
\label{eq:tile_volume}
    B(x_j,\lceil a/3\rceil) 
    \subseteq 
    B(x_j,(a-1)/2) 
    \subseteq \mathrm{Vor}(x_j) \subseteq B(x_j,b)
\end{equation}
If  $V_0$ is an $(a,b)$-net in a transitive graph $G$ of polynomial growth, the volume bounds for balls \eqref{eq:polynomial_volume_growth} implies volume bounds of Voronoi tiles.
We will refer to these bounds as the \textbf{volume bounds for Voronoi tiles}.
Whereas in Section \ref{sec:locality} we worked with a specific class of nets (the $A$-controlled nets of \thref{prop:fixed_unif_control_nets}), in  what follows we work with maximal $a$-separated nets.
Recall from Section \ref{sec:renormalisation} that such nets always exist by Zorn's lemma.

\begin{definition}
\thlabel{def:recursive_tiles}
    Given $\radius_0 \ge r_0 > 0$ and a sequence of positive reals $(m_i)_{i \geq 1}$, we define the \textbf{renormalisation scales} to be the sequence of positive reals $(r_i)_{i \geq 0}$ recursively defined by 
    \begin{equation}
        \label{eq:renormalisation_scales}
        r_{i + 1} = m_{i + 1} r_i.
    \end{equation}
    We let $(V_i)_{i\geq 0}$ be a sequence of maximal $r_i$-separated nets, where a given net $V_i$ consists of the elements $\{\netvertex{i}{j}\}_{j \geq 0}$, and we let $((\vor{\netvertex{i}{j}})_{j \geq 0})_{i \geq 0}$ be the sequence of Voronoi tilings associated to the nets $(V_i)_{i \geq 0}$. 
    We define a \textbf{level $0$ tile} $\Tile{0}{j} = \tile{0}{\netvertex{0}{j}} = \vor{\netvertex{0}{j}}$ and its associated \textbf{cell} $\Cell{0}{j} = \cell{0}{\netvertex{0}{j}} = B(\netvertex{0}{j}, \radius_0)$. 
    For $i \geq 1$, we define a \textbf{\level{i} tile} as the union of \level{i-1} tiles whose centre is contained in $\vor{\netvertex{i}{j}}$,
    \begin{equation}
        \label{eq:tile_definition}
        \Tile{i}{j}
        =
        \tile{i}{\netvertex{i}{j}}
        =
        \bigcup_{\netvertex{i-1}{k} \in \vor{\netvertex{i}{j}}}
        \tile{i-1}{\netvertex{i-1}{k}},
    \end{equation}
    and its associated \textbf{cell} is similarly defined
    \begin{equation}
        \label{eq:cell_definition}
        \Cell{i}{j}
        =
        \cell{i}{\netvertex{i}{j}}
        = 
        \bigcup_{\netvertex{i-1}{k} \in \vor{\netvertex{i}{j}}}
        \cell{i-1}{\netvertex{i-1}{k}}.
    \end{equation}
    We say that $\netvertex{i}{j}$ is the \textbf{centre} of $\Tile{i}{j}$ and we call the level $i-1$ tiles in $\Tile{i}{j}$ the \textbf{subtiles} of $T_{i,j}$. 
    Similarly, we say that $\netvertex{i}{j}$ is the \textbf{centre} of $\Cell{i}{j}$ and we call the level $i-1$ cells in $\Cell{i}{j}$ the \textbf{subcells} of $\Cell{i}{j}$.
    We write $\numsub{i}$ for the number of \level{i-1} net vertices in $\vor{\netvertex{i}{j}}$.
    We define the \textbf{renormalization scheme} associated to the renormalisation scales $(r_i)_{i \geq 0}$ to be the sequence of tilings $((\Tile{i}{j})_{j \geq 0})_{i \geq 0}$ associated to the sequence of nets $(V_i)_{i \geq 0}$.
\end{definition}

We make a number of remarks about this construction.
By the volume bounds for Voronoi tiles in \eqref{eq:tile_volume} and for $\radius_0 > r_0$, a cell $\Cell{i}{j}$ strictly contains its associated tile $\Tile{i}{j}$.
Whereas two neighbouring tiles are disjoint, two neighbouring cells are not. 
We will only use \textit{cells} in the first step of the renormalisation argument to prove \thref{prop:almost_linear}: we use the extra room provided by $\radius_0>r_0$ to apply ergodicity of $\P_{\beta}$ and initialise the renormalisation.
After that, we will work with tiles.
By construction, $\numsub{i}$ is the number of subtiles in $\Tile{i}{j}$ as well as the number of subcells in $\Cell{i}{j}$.
If $m_i$ is large, by the volume bounds for Voronoi tiles we may think of $r_i$ as roughly the radius of the tile $\Tile{i}{j}$, so that $\Tile{i}{j}$ has volume $\Theta(r_i^d)$ and contains $\Theta(m_i^d)$ tiles of level $i-1$.
Similarly, if $m_i$ is large and $\radius_0$ is small compared to $r_1$, the cell $\Cell{i}{j}$ has volume $\Theta(r_i^d)$ and contains $\Theta(m_i^d)$ cells of level $i-1$. The following proposition makes this precise.

\begin{proposition}
\thlabel{prop:tile_and_cell_volume_bounds}
    Let $G$ be a transitive graph of polynomial growth with $d \geq 1$ and let $\epsilon > 0$.
    Let $\radius_0 > r_0 > 0$, let $(m_i)_{i \in \N}$ satisfy $m_i > (1+\epsilon)/\epsilon$ for all $i \in \N$, and suppose that $\radius_0 \leq \epsilon r_1$. 
    Consider the renormalisation scheme $((\Tile{i}{j})_{j \geq 0})_{i \geq 0}$ associated to the renormalisation scales $(r_i)_{i \geq 0}$.
    Then there exist $\cvor{} = \cvor{}(G,\epsilon) > 0$ and $\Cvor{} = \Cvor{} (G,\epsilon) \ge 1$ such that 
    \begin{equation}
        \cvor{}
        r_{i}^d
        \leq 
        \# \Cell{i}{j} 
        \leq
        \Cvor{}
        r_{i}^d,
        \quad
        \cvor{}
        r_{i}^d
        \leq 
        \# \Tile{i}{j} 
        \leq
        \Cvor{}
        r_{i}^d,
        \quad 
        \text{and}
        \quad
        \cvor{} 
        m_i^d
        \leq
        \numsub{i}
        \leq 
        \Cvor{}
        m_i^d
    \end{equation}
    for any tile $\Tile{i}{j}$ and any cell $\Cell{i}{j}$.
    Further, $\Tile{i}{j} \subseteq B(x_{i,j},\Cvor{} r_i)$ for any tile $\Tile{i}{j}$.
\end{proposition}

The requirement that $\Cvor \ge 1$ is purely technical and can be achieved without loss of generality. An important consequence of these volume bounds and the transitivity and local finiteness of $G$ is that at each scale there are finitely many isomorphism classes of tiles and cells, viewed as induced subgraphs of $G$.

\begin{proof}
    A \level{0} cell is defined as $\Cell{0}{j} = B(x_{0,j},\radius_0)$, and for $i \geq 1$ a \level{i} cell is defined as the union of \level{i-1} cells whose centres are contained in $\vor{\netvertex{i}{j}}$.
    Hence the maximal distance between $\Cell{i}{j}$ and $\vor{\netvertex{i}{j}}$ is at most the radius of a \level{i{-}1} cell.
    By the volume bounds for Voronoi tiles in \eqref{eq:tile_volume}, and writing $s_{i-1} = \sum_{\ell = 1}^{i-1} r_{\ell}$, we obtain recursively for all $i \in \N$ that
    \begin{equation}
        \label{eq:tile_and_cell_proof_2}
        B(\netvertex{i}{j},(r_{i}-1)/2 - (\radius_0 + s_{i-1})) 
        \subseteq 
        \Cell{i}{j} 
        \subseteq B(\netvertex{i}{j},r_{i} + \radius_0 + s_{i-1}).
    \end{equation}
    Similarly, a \level{0} tile is defined as $\Tile{0}{j} = \vor{\netvertex{0}{j}}$ and hence for all $i \in \N$
    \begin{equation}
        \label{eq:tile_and_cell_proof_5}
        B(\netvertex{i}{j},(r_{i}-1)/2 - (r_0 + s_{i-1})) 
        \subseteq 
        \Tile{i}{j} 
        \subseteq B(\netvertex{i}{j},r_{i} + r_0 + s_{i-1}).
    \end{equation}
    The sequence $(m_i)_{i \in \N}$ is assumed to satisfy $m_i > (1+\epsilon)/\epsilon$ for all $i \in \N$.
    This implies for all $j < i$ that $r_j < (\epsilon/(1+\epsilon))^{i-j} r_i$, and hence $s_{i-1} < r_i \sum_{\ell=1}^{i-1} (\epsilon/(1+\epsilon))^{i-\ell}
    < r_i \sum_{\ell = 1}^{\infty}
    (\epsilon/(1+\epsilon))^{\ell} < \epsilon r_i$.
    By assumption, also $\radius_0 \leq \epsilon r_1 \leq \epsilon r_i$ and similarly $r_0 < (1+\epsilon)r_0 < \epsilon r_1 \leq \epsilon r_i$ for all $i \in \N$. 
    For all $\varepsilon<1/16$ and $r_0$ large enough $(r_i-1)/2-2\varepsilon r_i > \lceil r_i/3\rceil$, and we obtain
    \begin{equation}
        \label{eq:cell_bounds}
        B(\netvertex{i}{j},\lceil r_{i}/3\rceil ) 
        \subseteq 
        \Tile{i}{j} 
        \subset
        \Cell{i}{j} 
        \subseteq B(\netvertex{i}{j},(1+\epsilon)r_{i}),
    \end{equation}
    By construction and by the volume bounds for Voronoi tiles in \eqref{eq:tile_volume}, the number of \level{i-1} net vertices in a \level{i} Voronoi tile $\vor{\netvertex{i}{j}}$ satisfies
    \begin{multline}
        \label{eq:numtiles}
        \left\lfloor
            \frac{
                \# 
                B(\netvertex{i}{j},\lceil r_i/3
            \rceil )
            }{
                \#
                B(\netvertex{i-1}{j},\lceil r_{i-1}\rceil)
            }
        \right\rfloor\le 
        \left\lfloor
            \frac{
                \min_j 
                \# 
                \vor{\netvertex{i}{j}}
            }{
                \max_j 
                \#
                \vor{\netvertex{i-1}{j}}
            }
        \right\rfloor
        \leq
        \numsub{i}
        \le \left\lfloor
            \frac{
                \#
                B(\netvertex{i}{j},2 r_i )
            }{
                \# 
                B(\netvertex{i-1}{j},\lceil r_{i-1}/3\rceil )
            }
        \right\rfloor
        ,
    \end{multline}
    where we got the upper bound by observing that the ball $B(x_{i-1,k},\lfloor r_{i-1}/3\rfloor)$ around any $x_{i-1,k}\in \mathrm{Vor}(x_{i,j})$ must be contained in $B(x_{i,j}, \lceil r_i\rceil + \lfloor r_{i-1}/3\rfloor)$, and the latter is at most $2r_i$ for $r_0$ sufficiently large.
    Applying the volume bounds for balls in transitive graphs of polynomial growth \eqref{eq:polynomial_volume_growth} to \eqref{eq:cell_bounds} and \eqref{eq:numtiles}, applying the definition of $m_i = r_i/r_{i-1}$ to \eqref{eq:numtiles}, and setting $\cvor{}(G,\epsilon) > 0$ and $\Cvor{} (G,\epsilon) > 0$ concludes the proof.
\end{proof}

\section{Local existence of the giant}
\label{sec:cluster_decay_below}
In this section we prove \thref{prop:biskup} on the local existence of a giant.
The proof follows a multi-step renormalisation scheme.
The aim is to, for a clever choice of scales, construct a large cluster in each tile and ensure that nearby clusters connect as the scales increase.
The plentiful bulk-to-bulk connections in \eqref{eq:bulk-to-bulk-P} for $\alpha < 2$ are essential to this argument.
In the first step of the scheme we use the uniqueness of the infinite cluster $\cluster_{\infty}$ and the ergodicity of $\P_{\beta}$ to prove the following result on the existence of an almost-linear cluster with high probability.
Recall that $\giantone{r}$ denotes the largest cluster in $B(r)$ and $\theta = \theta(\beta) = \P_{\beta} \left(o \leftrightarrow \infty \right)$ denotes the percolation probability.

\begin{proposition}
\thlabel{prop:almost_linear}
    Let $G$ be a transitive graph of polynomial growth with $d \geq 1$ and suppose that $J : V \times V \to \R_+$ is a transitive kernel satisfying $J(x,y) = \Omega (d_G(x,y)^{-d \alpha})$ with $\alpha \in (0,2)$.
    Let $\beta > \beta_c$.
    For all $\epsilon > 0$ and $\delta > 0$, there exists $r$ sufficiently large such that
    \begin{equation}
        \label{eq:almost_linear_giant}
        \P_{\beta}
        \left( 
            \# 
            \giantone{r}
            \geq
            \theta
            \#
            B(r)^{1-\epsilon}
        \right) 
        \geq 
        1 - \delta.
    \end{equation}
\end{proposition}

In the second step, we set up another renormalisation scheme, with \thref{prop:almost_linear} as the base case, where we use the concentration of binomials and the connectivity properties of the \ER{} random graph to obtain the existence of a linear giant with stretched-exponential error.

\begin{proposition}
\thlabel{prop:biskup_with_error}
   Let $G$ be a transitive graph of polynomial growth with $d \geq 1$ and suppose that $J : V \times V \to \R_+$ is a transitive kernel satisfying $J(x,y) = \Omega (d_G(x,y)^{-d \alpha})$ with $\alpha \in (0,2)$.
    Let $\beta > \beta_c$.
    For $\delta \in (\max(0, 1-\alpha),2-\alpha)$ there exist $\nu = \nu(G) > 0$ and $c > 0$ such that  for all $r \in \N$
    \begin{equation}
        \label{eq:linear_giant_1}
        \P_{\beta}(
            \# 
            \giantone{r}
            \geq
            \nu
            \# 
            B(r)
        ) 
        \geq 
        1
        -
        \exp(
            -c 
            (\# B(r))^{2 - \alpha - \delta}
        ).
    \end{equation}
\end{proposition}

In the third and final step we use a result of O'Connell \cite{oconnell_large_1998} on the large deviations of the giant in the \ER{} random graph to obtain \thref{prop:biskup} from \thref{prop:biskup_with_error}.
The coarse-graining argument from the previous section is at the core of the proofs of \thref{prop:almost_linear} and \thref{prop:biskup_with_error}, and the main challenge is to maintain uniform control over the fluctuations in density of the local giant built along the renormalisation argument.

\subsection{First renormalisation scheme: almost-linear giant}
In this section we prove \thref{prop:almost_linear}.
We begin by making a choice of renormalisation scales for the renormalisation scheme.

\begin{definition}
\thlabel{def:renormalisation_scheme}
    Let $\xi > 0$ and let $\eta = \eta(\xi) = \lceil 2 \left( 1/\xi - 1 \right) \rceil$.
    Let $\radius_0 \geq r_0 \geq 1$ be two large constants. For $\alpha \in (0,2)$, we define $s = s(\alpha,G,\xi)$ to be the least integer satisfying $\cvor (i+s)^{2d} \ge 2$ and
    \begin{align}
        \label{eq:s_property_1}
        s^{\xi + \alpha}
        & \geq 
        (1 + s)^{\alpha},
        \\
        \label{eq:s_property_2}
        (1 + s)^{2 + \eta} 
        & > 
        (1 + \xi) / \xi,
        \\
        \label{eq:s_property_3}
        (1 + s)^{2 + \eta}
        & \geq 
        \radius_0/(\xi r_0),
        \\
        \label{eq:s_property_4}
        \frac{
            C_G (1 + \xi)^d 
        }{
            c_G (1 - 2 \xi)^d
        }
        (1+s)^{-2d}
        & < 
        1/2
    \end{align}
    simultaneously. For $i \geq 1$ we iteratively define 
    \begin{equation}
        \label{eq:inductive_definitions}
        m_i 
        = 
        (i + s)^{2 + \eta}
        \quad 
        \text{and}
        \quad
        r_i 
        = 
        m_i r_{i-1}.
    \end{equation}
    Let $((\Tile{i}{j})_{j \geq 0})_{i \geq 0}$ be the renormalisation scheme associated to the renormalisation scales $(r_i)_{i \geq 0}$.
\end{definition}

Note that this choice of renormalisation scales satisfies the hypotheses of \thref{prop:tile_and_cell_volume_bounds} with $\epsilon = \xi$.
Indeed, it follows from \eqref{eq:s_property_2} and \eqref{eq:inductive_definitions} that $m_i > (1 + \xi)/\xi$ for all $i \geq 1$, and it follows from \eqref{eq:s_property_3} and \eqref{eq:inductive_definitions} that $\radius_0 \leq \xi r_1 \leq \xi r_i$ for all $i \geq 1$. Recall $c_G$ from \eqref{eq:polynomial_volume_growth}.

\begin{definition}
\thlabel{def:berger_good}
    For $i \geq 1$, for $s$ as in \thref{def:renormalisation_scheme}, and for $\cvor{}$ as in \thref{prop:tile_and_cell_volume_bounds}, let
    \begin{equation}
    \label{eq:rho_definition}
        \rho_0 
        = 
        \theta 
        \cvol
        /
        2^{2d + 1},
        \quad
        \rho_i 
        =
        1
        /
        \left(
            \cvor{} 
            (i+s)^{2d}
        \right).
    \end{equation}
    Let $((\Tile{i}{j})_{j \geq 0})_{i \geq 0}$ be the renormalisation scheme in \thref{def:renormalisation_scheme}, and consider the associated sequence of cells $((\Cell{i}{j})_{j \geq 0})_{i \geq 0}$. 
    We say that a \level{0} cell $\Cell{0}{j}$ is \textbf{good} if it contains a cluster $\goodcluster{0}{j}$ satisfying
    \begin{equation}
        \label{eq:good_level_zero_cell}
        \# (\goodcluster{0}{j} \cap B(\netvertex{0}{j},r_0/4))
        \geq 
        \rho_0
        r_0^d. 
    \end{equation}
    We call the cluster $\goodcluster{0}{j}$ the \textbf{good cluster} and we call the set $\goodsize{0}{j} = \goodcluster{0}{j} \cap B(\netvertex{0}{j},r_0/4)$ the \textbf{semi-cluster}. 
    For $i \geq 1$, we say that a \level{i} cell $\Cell{i}{j}$ is \textbf{good} if at least a $\rho_i$ proportion of its \level{i{-}1} subcells are good, and  their good clusters are connected to each other inside $\mathrm{Cell}_{i,j}$ to form the cluster $\goodcluster{i}{j}$. 
    We call the cluster $\goodcluster{i}{j}$ the \textbf{good cluster} and we call the set $\goodsize{i}{j} = \goodcluster{i}{j} \cap \Tile{i}{j}$ the \textbf{semi-cluster}. 
    When the good cluster is not unique, we choose one arbitrarily.
\end{definition}

Note that a semi-cluster $\goodsize{i}{j}$ is not necessarily connected inside the tile $\Tile{i}{j}$, but the associated good cluster $\goodcluster{i}{j}$ forms a connected subgraph in the spanned graph inside the cell $\Cell{i}{j}$.
Note too that a good cluster may not be maximal. 
Further, the semi-cluster $\goodsize{0}{j}$ is strictly contained in the tile $\Tile{0}{j}$ since
\begin{equation}   
    \label{eq:good_cluster_containment}
    \goodsize{0}{j} 
    =  
    \goodcluster{0}{j} \cap B(\netvertex{0}{j},r_0/4)
    \subseteq 
    B(\netvertex{0}{j},r_0/3)
    \subseteq
    \Tile{0}{j},
\end{equation}
where in the last inclusion we use the volume bounds for Voronoi tiles.
The semi-cluster is defined in this way so that two distinct semi-clusters $\goodsize{i}{j}$ and $\goodsize{i}{k}$ are disjoint with $d_G(\goodsize{i}{j},\goodsize{i}{k}) \geq r_0/2$.
For sufficiently large $r_0$, this will allow us to use the asymptotics of the kernel in \thref{prop:almost_linear} for connections between semi-clusters.
We record this and some size estimates about the semi-cluster in the following lemma.

\begin{lemma}
\thlabel{lem:local_giant_density}
    Let $G$ be a transitive graph of polynomial growth with $d \geq 1$ and let $((\Tile{i}{j})_{j \geq 0})_{i \geq 0}$ be the renormalisation scheme in \thref{def:renormalisation_scheme}.
    Let $\Cell{i}{j}$ be a good \level{i} cell and let $\Cell{i-1}{k}$ and $\Cell{i-1}{\ell}$ be two distinct good subcells of $\Cell{i}{j}$.
    The semi-cluster $\goodsize{i}{j}$ satisfies $\# \goodsize{i}{j} \geq \rho_0 r_i^{d(1 - \xi)}$. 
    The semi-clusters $\goodsize{i-1}{k}$ and $\goodsize{i-1}{\ell}$ satisfy $d_G(\goodsize{i-1}{k},\goodsize{i-1}{\ell}) \geq r_0/2$ and $\goodsize{i-1}{k},\goodsize{i-1}{\ell} \subset \Tile{i}{j} \subseteq B(\netvertex{i}{j},\Cvor{}r_i)$.
\end{lemma}

\begin{proof}
    It follows from \thref{prop:tile_and_cell_volume_bounds}, the definition of $\rho_i$ in \eqref{eq:rho_definition}, and the definition of $r_i$ in \eqref{eq:inductive_definitions}, that $\numsub{i} \geq \cvor{} m_i^d$, $\cvor{} \rho_i = 1/(i + s)^{2d}$, and $r_0\prod_{k=1}^i m_i= r_i$ for all $i \geq 1$.
    By the inductive definition of the semi-cluster $\goodsize{i}{j}$, we have that 
    \begin{equation}
        \# 
        \goodsize{i}{j} 
        \geq 
        \rho_0 
        \# B(r_0/4)
        \prod_{k = 1}^i
        \rho_k
        \min_j 
        \numtiles{k}
        \geq
        \rho_0 
        r_i^d 
        \prod_{k = 1}^i 1/(k+s)^{2d} 
        \geq  
        \rho_0 
        r_i^{d(1-\xi)}.
    \end{equation}
    By the definitions of $r_i$ and $\eta$ in \thref{def:renormalisation_scheme} we have $\prod_{k=1}^{i} 1/(k+s)^{2d} \geq (r_i/r_0)^{-2d/(2+\eta)} \geq r_i^{-d \xi}$, and hence $\# \goodsize{i}{j} \geq \rho_0 r_i^{d(1 - \xi)}$ as required.
    Suppose now that $\goodsize{0}{j}$ and $\goodsize{0}{k}$ are the semi-clusters in two distinct \level{0} good cells $\Cell{0}{j}$ and $\Cell{0}{k}$. 
    By the definition of the net $V_0$ we have $d_G(\netvertex{0}{j},\netvertex{0}{k}) \geq r_0$.
    Since $\goodsize{0}{j} \subseteq B(\netvertex{0}{j},r_0/4)$ and $\goodsize{0}{k} \subseteq B(\netvertex{0}{k},r_0/4)$, it follows that $d_G(\goodsize{0}{j},\goodsize{0}{k}) \geq r_0/2$.
    It follows from the iterative definition of tiles and of semi-clusters that the distance between a semi-cluster $\goodsize{i}{j}$ and the boundary of the tile $\Tile{i}{j}$ containing it is at least $r_0/4$.
    Since tiles are disjoint we have that $d_G(\goodsize{i}{j},\goodsize{i}{k}) \geq r_0/2$ for any two semi-clusters $\goodsize{i}{j}$ and $\goodsize{i}{k}$ in two distinct \level{i} good cells $\Cell{i}{j}$ and $\Cell{i}{k}$.
    If $\Cell{i}{j}$ is a \level{i} cell and $\Cell{i-1}{k}$ and $\Cell{i-1}{l}$ are two distinct good subcells of $\Cell{i}{j}$ it follows that the semi-clusters $\goodsize{i-1}{k}$ and $\goodsize{i-1}{\ell}$ are contained in the tiles $\Tile{i-1}{k}$ and $\Tile{i-1}{\ell}$ which are contained in $\Tile{i}{j}$.
    Lastly, by \thref{prop:tile_and_cell_volume_bounds} we have that $\Tile{i}{j} \subseteq B(\netvertex{i}{j},\Cvor{}r_i)$.
\end{proof}

\thref{lem:local_giant_density} gives that a good cell contains a component whose size is almost linear in the volume. 
Our goal is to show that $r_0$ can be taken sufficiently large such that $\P_{\beta}(\Cell{i}{j} \text{ is good})$ is arbitrarily close to $1$ for all cells and at all scales. 
The lemma below gives that we may choose $r_0$ and $\radius_0$ in \thref{def:renormalisation_scheme} sufficiently large such that a \level{0} cell is good with high probability. 

\begin{lemma}
\thlabel{lem:base_tile}
    Let $G$ be a transitive graph of polynomial growth with $d \geq 1$ and suppose that $J : V \times V \to \R_+$ is a transitive kernel satisfying $J(x,y) = \Omega (d_G(x,y)^{-d \alpha})$ with $\alpha \in (0,2)$.
    Let $\beta > \beta_c$, and let $((\Tile{i}{j})_{j \geq 0})_{i \geq 0}$ be the renormalisation scheme in \thref{def:renormalisation_scheme}. 
    For $\delta > 0$, there exist $r_0$ and $\radius_0$ sufficiently large such that for any cell $\Cell{0}{j}$
    \begin{equation}
        \label{eq:prob_level_0_cell}
        \P_{\beta}
        \left(
            \Cell{0}{j}
            \text{ is good}
        \right)
        \geq
        1 - \delta.
    \end{equation}
\end{lemma}

\begin{proof}
    By ergodicity, the proportion of vertices in  $B(r_0/4)$ which are also in the unique infinite cluster $\cluster_{\infty}$ converges in probability to $\theta$ as $r_0$ tends to infinity.
    Together with the transitivity of $G$ it follows that we can choose $r_0$ sufficiently large such that
    \begin{equation}
        \P_{\beta}
        (
            \# (B(\netvertex{0}{j},r_0/4) \cap \cluster_\infty)
            \geq 
            \rho_0 r_0^d
        )
        \geq
        \P_{\beta}
        \left(
            \# (B(r_0/4) \cap \cluster_\infty)
            \geq 
            \theta \# B(r_0/4)
            /
            2
        \right)
        \geq 
        1
        -
        \delta / 2
    \end{equation}
    for any cell $\Cell{0}{j}$, where the first inequality uses that in \eqref{eq:rho_definition} we chose $\rho_0$ such that $\rho_0 r_0^d \leq \theta \# B(r_0/4) / 2$, and the second inequality uses ergodicity.
    Since the infinite component is unique, and since $G$ is transitive and locally finite, we can choose $\radius_0$ sufficiently large such that 
    \begin{equation}
        \label{eq:prob_level_0_cell_proof_3}
    	\P_{\beta}
        (
            B(\netvertex{0}{j},r_0/4)
            \cap 
            \cluster_{\infty} 
            \text{ is connected in } 
            \Cell{0}{j} 
        )
        \geq 
        1 - \delta / 2
    \end{equation}
    for any cell $\Cell{0}{j}$.
    For this choice of $r_0$ and $\radius_0$ a union bound yields that
    \begin{multline}
        \label{eq:prob_level_0_cell_proof_4}
        \P_{\beta}
        \left(
            \Cell{0}{j} 
            \text{ is good}
        \right)
        \geq 
        \P_{\beta}
        \left(
            \# (B(\netvertex{0}{j},r_0/4) \cap \cluster_{\infty})
            \geq 
            \rho_0 r_0^d
        \right)
        \\
        +
        \P_{\beta}
        \left(
            B(\netvertex{0}{j},r_0/4)
            \cap 
            \cluster_{\infty} 
            \text{ is connected in } 
            \Cell{0}{j}
        \right)
        \geq 
        1 - \delta
    \end{multline}
    for any cell $\Cell{0}{j}$, as desired.
\end{proof}  

The following lemma shows that we may choose $r_0$ and $\radius_0$ sufficiently large in the renormalisation scheme such that an arbitrary tile at any scale is good with uniformly high probability. 
Note that the bound is uniform over all {\color{blue} cells} in a given renormalisation scheme.

\begin{lemma}[Iterative renormalisation]
    \thlabel{lem:first_iterative_renormalisation}
    Let $G$ be a transitive graph of polynomial growth with $d \geq 1$ and suppose that $J : V \times V \to \R_+$ is a transitive kernel satisfying $J(x,y) = \Omega (d_G(x,y)^{-d \alpha})$ with $\alpha \in (0,2)$.
    Let $\beta > \beta_c$, let $((\Tile{i}{j})_{j \geq 0})_{i \geq 0}$ be the renormalisation scheme in \thref{def:renormalisation_scheme}, and let $\xi=\xi(d,\alpha) > 0$ be sufficiently small. 
    For $\delta > 0$, there exist $r_0$ and $\radius_0$ sufficiently large such that  for any cell $\Cell{i}{j}$
    \begin{equation}
        \label{eq:prob_level_i_cell}
        \P_{\beta}
        \left(
            \Cell{i}{j}
            \text{ is good}
        \right)
        \geq
        1 - \delta.
    \end{equation}
\end{lemma}
\begin{proof}
    We argue by induction on $i$, and we consider the probability that a cell is \textit{bad}. 
    For $i \geq 0$ we write $\lambda_i = \max_j \P_{\beta,J}\left(\Cell{i}{j} \text{ is bad} \right)$ and $q_i = \prod_{\ell = 0}^i (1 + 3 \rho_{\ell})$.
    By the choice of $s$ in \eqref{eq:s_property_4} and by the definition of $\rho_i$ in \eqref{eq:rho_definition}, the product $q = \prod_{\ell = 0}^{\infty} (1 + 3 \rho_{\ell})$ converges.
    Note that $q_i/q < 1$ for all $i \geq 0$.
    For $i = 0$, it follows from \thref{lem:base_tile} that we can choose $r$ and $\radius_0$ sufficiently large such that $\lambda_0 \leq \delta q_0/ q$.
    Now let $\Cell{i}{j}$ be a \level{i} cell, 
    let $\lambda_{i,j} = \P_{\beta}\left(\Cell{i}{j} \text{ is bad}\right)$, 
    let $\phi_{i,j}$ be the probability that less than a $\rho_i$ proportion of the subcells of $\Cell{i}{j}$ are good,
    and let $\varphi_{i,j}$ be the probability that the good clusters are not connected provided that there are enough good subcells.
    This gives $\lambda_{i,j} \leq \phi_{i,j} + \varphi_{i,j}$, and we bound each probability in turn.
    Recall from \thref{def:recursive_tiles} that $\numsub{i}$ denotes the number of subcells in $\Cell{i}{j}$.
    By Markov's inequality, together with translation invariance and the induction hypothesis,
    \begin{equation}    
        \label{eq:good_cells_bound}
        \phi_{i,j}
        \leq 
        \frac{
            \E
            \left[
                \# 
                \left\{
                    \Cell{i-1}{k} 
                    \subseteq \Cell{i}{j} : 
                    \Cell{i-1}{j} 
                    \text{ is bad}
                \right\}
            \right]
        }{
            (1-\rho_i) \numsub{i}
        }
        \leq 
        \frac{
            \P
            \left(
                \Cell{i-1}{k} 
                \text{ is bad}
            \right)
        }{
            1-\rho_i
        }
        \leq
        \frac{
            \lambda_{i-1}
        }{
            1-\rho_i
        }.
    \end{equation}
    We now bound $\varphi_{i,j}$ from above by the probability that there is at least one pair of good clusters that are not connected by a direct edge. 
    Let $\Cell{i-1}{k}$ and $\Cell{i-1}{\ell}$ be two distinct good subcells of $\Cell{i}{j}$.
    Since the good clusters contain their associated semi-clusters, it is immediate that 
    \begin{multline}
        \label{eq:cluster_no_connection_proof_1}
        \P_{\beta}
        \left(
            \goodcluster{i-1}{k}
            \not\sim
            \goodcluster{i-1}{\ell}
            \mid
            \Cell{i-1}{k},
            \Cell{i-1}{\ell}
            \text{ are good}
        \right)
        \\
        \leq 
        \P_{\beta}
        \left(
            \goodsize{i-1}{k}
            \not\sim
            \goodsize{i-1}{\ell}
            \mid
            \Cell{i-1}{k},
            \Cell{i-1}{\ell}
            \text{ are good}
        \right).
    \end{multline}
    \thref{lem:local_giant_density} implies that $\goodsize{i-1}{k}$ and $\goodsize{i-1}{\ell}$ are disjoint with $d_G(\goodsize{i-1}{k},\goodsize{i-1}{\ell}) \geq r_0 /2$, as well as $\goodsize{i-1}{k},\goodsize{i-1}{\ell} \subseteq B(\netvertex{i}{j},\Cvor{}r_i)$ and $\# \goodsize{i-1}{k},\# \goodsize{i-1}{\ell} \geq \rho_0 r_{i-1}^{d(1-\xi)}$.
    We choose $r_0$ sufficiently large such that $r_0/2 \geq R_J$.
    We argue by induction on $i$ to show that we can choose $r_0$ sufficiently large such that $m_i^{\alpha} \leq r_{i-1}^{\xi}$ for all $i \geq 1$.
    By the definition of $m_1$ we can choose $r_0$ sufficiently large such that $m_1^{\alpha} = (1+s)^{(2+\eta)\alpha} \leq r_0^{\xi}$.
    By the definition of $r_i$ and $m_i$, by the induction hypothesis, and by the choice of $s$ in \eqref{eq:s_property_1}, we have
    \begin{equation}   
        \label{eq:bound_mi}
        r_{i-1}^{\xi}
        =
        m_{i-1}^{\xi}
        r_{i-2}^{\xi}
        \geq
        m_{i-1}^{\xi + \alpha}
        =
        m_{i}^{\alpha}
        \frac{
            m_{i-1}^{\xi + \alpha} 
        }{
            m_i^{\alpha}
        }   
        =
        m_{i}^{\alpha}
        \frac{
            (i - 1 + s)^{
            (2 + \eta)
            (\xi + \alpha)
            }
        }{
            (i + s)^{(2 + \eta) \alpha}
        }
        \geq 
        m_i^{\alpha}
    \end{equation}
    as required.
    For such a choice of $r_0$, for some $c_1 = c_1(\beta,J) > 0$ and $c_2 = c_2(\beta,G,J) >0$,
    \begin{align}
        \P_{\beta}
        \left(
            \goodsize{i-1}{k}
            \not\sim
            \goodsize{i-1}{l}
            \mid
            \Cell{i-1}{k},
            \Cell{i-1}{l}
            \text{ are good}
        \right)
        & \leq
        \exp
        \Bigg(
            -
            \frac{
                c_1
                (
                    \rho_0 
                    r_{i-1}^{d(1 - \xi)}
                )^2
            }{
                \left(
                    2
                    \left(
                        1+\xi
                    \right)
                    r_i
                \right)^{d \alpha}
            }
        \Bigg)
        \\
        & =
        \exp
        \left(
            -
            c_2
            r_{i - 1}^{d(2 - \alpha - 3 \xi)}
        \right)
    \end{align}
    where we use that $r_i = m_i r_{i-1}$.
    It follows that 
    \begin{align}
    \label{eq:decay_stretched_exponentially_1}
        \varphi_{i,j}
        & \leq
        \numsub{i}^2
        \exp
        \left(
            - 
            c
            r_{i-1}^{d(2 - \alpha - 3\xi)}
        \right).
    \end{align}
    Since $\alpha \in (0,2)$, we may choose $\xi$ sufficiently small such that $d(2 - \alpha - 3\xi)>0$.
    By \thref{prop:tile_and_cell_volume_bounds} and by the definition of $m_i$ in \eqref{eq:inductive_definitions}, the term $\numsub{i}^2$ is polynomial in $i$.
    By the definition of $r_{i-1}$ in \eqref{eq:inductive_definitions}, and since $s \geq 1$, it is trivially the case that $r_{i-1} \geq i$. 
    It follows that $\varphi_{i,j}$ decays at least stretched-exponentially in $i$.
    On the other hand, $\rho_i$ decays polynomially in $i$, and we may choose $r_0$ sufficiently large such that $\max_j \varphi_{i,j} \leq \delta \rho_i / q$ for all $i \geq 1$. 
    Together with the induction hypothesis, and using that $\rho_i < 1/2$ by the choice of $s$ in \eqref{eq:s_property_4}, using $q_i=\prod_{\ell=0}^i(1+3\rho_\ell)$, we bound
    \begin{equation}
        \label{eq:max_good_prob}
        \max_{j \geq 0}
        \lambda_{i,j}
        \leq 
        \frac{\lambda_{i-1}}{1 - \rho_i}
        + 
        \frac{\delta \rho_i}{q}
        \leq 
        \lambda_{i-1}(1 + 2 \rho_i)
        + 
        \frac{\delta \rho_i}{q}
        \leq
        \frac{\delta q_{i-1}(1 + 2 \rho_i)}{q}
        +
        \frac{\delta \rho_i}{q}
        \leq
        \frac{\delta q_i}{q}
        <
        \delta,
    \end{equation}
    concluding the induction.
\end{proof}

\begin{proof}[Proof of \thref{prop:almost_linear}]
    Let $\epsilon > 0$ and $\delta > 0$, and let $((\Tile{i}{j})_{j \geq 0})_{i \geq 0}$ be the renormalisation scheme in \thref{def:renorm_scheme_2}.
    By \thref{lem:first_iterative_renormalisation} we can choose $\xi < \epsilon$ sufficiently small and $r_0$ and $\radius_0$ sufficiently large such that
    \begin{equation}
        \label{eq:arbitrary_radius_base}
        \P
        \left(
            \Cell{i}{j}
            \text{ is good}
        \right)
        \geq
        1 - \delta/4
    \end{equation}
    for any cell $\Cell{i}{j}$.
    Fix such a choice of $r_0,\radius_0$, and $\xi$, and for $r > r_1(1 + \xi) + \radius_0$, let
    \begin{equation}\label{eq:rtilde-def}
        i_* = \max\{i \geq 1 : r_i(1+ \xi) + \radius_0 < r\}, 
        \quad
        \tilde{r} 
        = 
        r 
        - 
        r_{i_*}(1+ \xi) 
        - 
        \radius_0.
    \end{equation}
    The choice of $\tilde{r}$ is such that the \level{i_*} cells whose centre is in $B(\tilde{r})$ are contained in $B(r)$.
    By \thref{lem:local_giant_density}, a good cell $\Cell{i_*}{j}$ contains a good cluster $\goodcluster{i_*}{j}$ with $\# (\goodcluster{i_*}{j} \cap \Tile{i_*}{j}) \geq \rho_0 r_{i_*}^{d(1 - \xi)}$.
    Let $\kappa$ be the number of \level{i_*} cells whose centre is in $B(\tilde{r})$.
    We distinguish two cases according to the value of $\kappa$. 
    Let $M$ be the least integer such that $c_G M^d / ( 3^{d} C_G) \geq 2$.
    Suppose first  that $\tilde{r} < M r_{i_*}$. 
    By the choice of $s$ in \eqref{eq:s_property_3} and by the definition of $r_i$ in \eqref{eq:inductive_definitions} we have $\radius_0 \leq \xi r_1$, so that in particular $\radius_0 \leq \xi r$ and hence $r_{i_*} \geq r(1 - \xi)/(1 + M + \xi)$.
    Without loss of generality we can translate the net $V_{i_*}$ so that $0 \in V_{i_*}$. 
    Then the ball $B(\tilde{r})$ and contains the centre of at least one \level{i_*} cell, namely $\Cell{i_*}{0}$, which is fully contained in $B(r)$. 
    Suppose that $\Cell{i_*}{0}$ is good.
    By the definition of $\rho_0$ in \eqref{eq:rho_definition},
    \begin{equation}
        \label{eq:arbitrary_radius_giant_size}
        \# 
        \giantone{r}
        \geq 
        \# \goodcluster{i_*}{0}
        \geq 
        \rho_0 
        r_{i_*}^{d(1-\xi)} 
        \geq
        \frac{
            \theta
            c_G
            ((1 - \xi) r)^{d(1-\xi)}
        }{
            2^{2d + 1}
            (1 + M + \xi)^{d(1-\xi)}
        }
        \geq 
        \theta
        \#
        B(r)^{1-\epsilon}.
    \end{equation}
    where the last inequality holds by $\epsilon > \xi > 0$, for all $r$ sufficiently large. Together with \eqref{eq:arbitrary_radius_base} this gives
    \begin{equation}
        \label{eq:arbitrary_radius_good_prob}
        \P
        \left(
            \# 
            \giantone{r}
            \geq 
            \theta
            \# B(r)^{1-\epsilon}
        \right)
        \geq 
        \P
        \left(
            \Cell{i_\star}{0}
            \text{ is good}
        \right)
        \geq
        1
        - 
        \delta/4
        \geq
        1
        - 
        \delta.
    \end{equation}
    Suppose now that $\tilde{r} \geq M r_{i_*}$.
    Again by the choice of $s$ and by the definition of $r_i$ we have $\radius_0 \leq \xi r$ and hence $\tilde{r} \geq r M (1 - \xi)/(M + 1 + \xi)$. 
    By the volume bounds for balls and Voronoi tiles, we have
    \begin{equation}
        \label{eq:kappa_lower_bound}
        \kappa 
        \geq 
        \left\lfloor
            \frac{
                \min \# B(\tilde{r})
            }{
                \max_j \# \vor{\netvertex{i_*}{j}}
            }
        \right\rfloor
        \geq
        \frac{
                c_G 
                \tilde{r}^d
            }{
                3^{d}
                C_G
                r_{i_*}^d
            }
        \geq 
        \frac{
            c_G 
            \left(
                M
                \left(
                    1 - \xi
                \right)
            \right)^d
            r^d
            }{
                3^{d}
                C_G
                \left(
                    1
                    + 
                    M
                    +
                    \xi
                \right)^d
                r_{i_*}^d
        }.
    \end{equation}
    As $\tilde r > r_{i_*}$, the third term from the right gives that $\kappa \geq 2$ by the choice of $M$.
    Let $\lambda$ be the probability that at least half of the cells are good and that their good clusters are connected to form the good cluster $\CC$. 
    If this is the case, and since by construction $\CC$ is contained in $B(r)$, it follows from the definition of $\rho_0$ in \eqref{eq:rho_definition} and the fact that $\epsilon > \xi > 0$, for $r$ sufficiently large,
    \begin{align}
        \label{eq:berger_giant_size}
        \# \cluster_r^{(1)} 
        \geq 
        \# \CC
        \geq 
        \frac{
            \rho_0 
            \kappa 
            r_{i_*}^{d(1-\xi)}
        }{
            2
        }
        \geq
        \frac{
            \theta
            c_G^{2 - \xi}
            \left(
                M
                \left(
                    1 - \xi
                \right)
            \right)^{d(1-\xi)}
            r^{d(1-\xi)}
        }{
            2^{2d + 2} 
            \left(
                3^{d}
                C_G
                \left(
                    1
                    +
                    M
                    +
                    \xi
                \right)
            \right)^{1 - \xi}
        }
        \geq 
        \theta
        \#B(r)^{1-\epsilon}.
    \end{align}
    We bound $\lambda$ by an argument similar to that in \thref{lem:first_iterative_renormalisation}. 
    Let $\phi$ be the probability that less than half of the associated cells are good, and let $\varphi$ be the probability that the good clusters are not connected given that there are enough good \level{i_*} cells. 
    This gives $\lambda \geq 1 - \phi - \varphi$.
    By Markov's inequality and by translation invariance,
    \begin{equation}
        \label{eq:arbitrary_radius_phi_bound}
        \phi
        \leq 
        \frac{
            \E
            \left[
                \# 
                \left\{
                    \netvertex{i_*}{j}
                    \in
                    B(0,\tilde{r}) : 
                    \Cell{i_*}{j}
                    \text{ is bad}
                \right\}
            \right]
        }{
            \kappa/2
        }
        \leq
        2 \max_j\P(\Cell{i_*}{j} \text{ is bad})
        \leq 
        \delta/2.
    \end{equation}
    We now bound $\varphi$ from above by the probability that there is at least one pair of good clusters that are not connected by a direct edge. 
    By the same reasoning as in \thref{lem:first_iterative_renormalisation} (specifically from equation \eqref{eq:cluster_no_connection_proof_1} to \eqref{eq:decay_stretched_exponentially_1}), with the only difference that the good clusters are contained in $B(r)$ and not $B(\Cvor{} r_{i_{\star}})$, we can choose $r_0$ sufficiently large such that
    \begin{align}
    \label{eq:decay_stretched_exponentially_2}
        \varphi
        \leq
        \kappa^2
        \exp
        \left(
            - 
            c
            r_{i_*}^{d(2 - \alpha - 3\xi)}
        \right)
    \end{align}
    for some $c = c(\beta,G,J)$.
    By the same reasoning as in \thref{lem:first_iterative_renormalisation} (specifically the paragraph below \eqref{eq:decay_stretched_exponentially_1}), $\varphi$ decays at least stretched-exponentially in $i$, and we may choose $r_0$ sufficiently large such that $\varphi \leq \delta/2$. 
    For this choice of $r$ 
    \begin{equation}
        \label{eq:berger_conclusion}
        \P
        \left(
            \# 
            \giantone{r}
            \geq 
            \theta
            \# B(r)^{1-\epsilon}
        \right)
        \geq 
        \lambda 
        \geq
        1
        -
        \phi
        -
        \varphi
        \geq 
        1 
        - 
        \delta,
    \end{equation}
    as desired.
\end{proof}

\subsection{Second renormalisation: linear giant with stretched exponential error}
In this section we prove \thref{prop:biskup_with_error} on the existence of a linear giant with stretched-exponential error.
We use the same coarse-graining argument as outlined in \thref{def:recursive_tiles}, but with a different set of renormalisation scales.
In particular, we use \textit{tiles} instead of \textit{cells}, so that there is no overlap.

\begin{definition}
\thlabel{def:renorm_scheme_2}
    For $\alpha \in(0,2)$ and $\delta \in (\max(0, 1-\alpha),2-\alpha)$, let
    \begin{equation}   
        \label{eq:biskup_xi}
        \xi
        =
        \xi(\delta, \alpha)
        = 
        \frac{
            \delta/2
        }{
            1 - \delta /2
        }.
    \end{equation}
    Let $r_0 > 1$ be a large constant, $\bar r_0:=r_0$, and for $i \geq 1$ we iteratively define
    \begin{equation}
        \label{eq:biskup_mi}
        m_i = r_{i-1}^{\xi}
        \quad
        \text{and}
        \quad
        r_i = m_i r_{i-1}.
    \end{equation}
    Let $((\Tile{i}{j})_{j \geq 0})_{i \geq 0}$ be the renormalisation scheme associated to the renormalisation scales $(r_i)_{i \geq 0}$ in \thref{def:recursive_tiles}.
\end{definition}

Just as before in \thref{def:renormalisation_scheme}, this choice of renormalisation scales satisfies the hypotheses of \thref{prop:tile_and_cell_volume_bounds}. 
Indeed, by the definition of $m_i$ in \eqref{eq:biskup_mi} we have $m_i \geq r_0^{\xi}$ for all $i \geq 1$ and we can choose $r_0$ sufficiently large such that $m_i > (1 + \epsilon)/\epsilon$ for all $i \geq 1$.

\begin{definition}
    \thlabel{def:biskup_good}
    For $\Cvor{}\ge 1$ from \thref{prop:tile_and_cell_volume_bounds}, let
    \begin{equation}
    \label{eq:epsilon_def}
        \epsilon_0
        = 
        1 - r_0^{- \delta / 4},
        \quad
        \epsilon_i 
        = 
        \frac{
            1
        }{
            C_{\mathrm{vor}}^2
            (i+1)^2
        }.
    \end{equation}  
    For $\delta \in (\max(0, 1-\alpha),2-\alpha)$, $\cvor{}$ as in \thref{prop:tile_and_cell_volume_bounds}, and for $i \geq 1$, let
    \begin{equation}
        \label{eq:density_definition}
        \rho_0 
        =
        \theta
        r_0^{- d \delta/4}/\cvor{},
        \quad
        \rho_i 
        =
        \theta
        r_0^{- \delta/4} 
        \prod_{j=2}^{i+1} 
        (1 - j^{-2}),
        \quad 
        \rho_\infty
        = 
        \lim_{j \to \infty} \rho_j.
    \end{equation}
    Let $((\Tile{i}{j})_{j \geq 0})_{i \geq 0}$ be the renormalisation scheme in \thref{def:renorm_scheme_2}. 
    We say that a tile $\Tile{0}{j}$ is \textbf{good} if $B(x_{0,j}, r_0/4)$ contains a cluster $\goodcluster{0}{j}$ satisfying
    \begin{equation}
        \label{eq:biskup_good_cluster}
        \# 
        \goodcluster{0}{j} 
        \geq 
        \rho_0 
        \# \Tile{0}{j}.
    \end{equation}
    We call the cluster $\goodcluster{0}{j}$ the \textbf{good cluster}.
    For $i \geq 1$, we say that a \level{i} tile $\Tile{i}{j}$ is \textbf{good} if at least a $(1-\epsilon_i)$ proportion of its subtiles are good and their good clusters are connected to each other inside $\Tile{i}{j}$ to form the \textbf{good cluster} $\goodcluster{i}{j}$.
    When the good cluster is not unique, we choose one arbitrarily.
\end{definition}

Note that the product in \eqref{eq:density_definition} is telescopic and can be written as $(i+2)/(2(i+1))$, so that the limit $\rho_{\infty}$ exists and in particular $\rho_i > \rho_{\infty}$ for all $i \geq 1$.
Just as in \thref{def:berger_good}, a good cluster $\goodcluster{i}{j}$ forms a connected subgraph in the spanned graph inside the tile $\Tile{i}{j}$, but it is not necessarily maximal. 
Just as in \thref{lem:local_giant_density}, we establish some deterministic facts about the good cluster.

\begin{lemma}
    \thlabel{lem:biskup_size}
    Let $G$ be a transitive graph of polynomial growth with $d \geq 1$ and let $((\Tile{i}{j})_{j \geq 0})_{i \geq 0}$ be the renormalisation scheme in \thref{def:renorm_scheme_2}.
    The good cluster $\goodcluster{i}{j}$ in a good tile $\Tile{i}{j}$ satisfies $\# \goodcluster{i}{j} \geq \rho_{\infty} \# \Tile{i}{j}$.
    Let $\Tile{i}{j}$ be a \level{i} tile, and let $\Tile{i-1}{k}$ and $\Tile{i-1}{l}$ be two distinct good subtiles of $\Tile{i}{j}$.
    The good clusters $\goodcluster{i-1}{k}$ and $\goodcluster{i-1}{\ell}$ satisfy $d_G(\goodcluster{i-1}{k},\goodcluster{i-1}{\ell}) \geq r_0/2$ and $\goodcluster{i-1}{k},\goodcluster{i-1}{\ell} \subset \Tile{i}{j} \subseteq B(\netvertex{i}{j},\Cvor{} r_i)$.
\end{lemma}

\begin{proof}
    We argue by induction on $i$. For $k = 0$, we have by definition that a good cluster $\goodcluster{0}{j}$ in a good tile $\Tile{0}{j}$ satisfies $\# \goodcluster{0}{j} \geq \rho_0 \# \Tile{0}{j}$, and we may assume inductively that for all $k \leq i - 1$ a good cluster $\goodcluster{k}{j}$ in a good tile $\Tile{k}{j}$ satisfies $\# \goodcluster{k}{j} \geq \rho_k \# \Tile{k}{j}$. 
    Let $\Tile{i}{j}$ be a good tile. By definition the good cluster $\goodcluster{i}{j}$ contains at least $(1-\epsilon_i) \numsub{i}$ \level{i{-}1} good clusters, and we re-index so that $1 \leq j \leq (1 - \epsilon_i) \numsub{i}$ are the indices for the good  tiles. 
    By \thref{prop:tile_and_cell_volume_bounds} we have $\numsub{i}\max_j \# \Tile{i-1}{j} \leq C_{\mathrm{vor}}^2 \# \Tile{i}{j}$, and together with the definitions of $\epsilon_i$ and $\rho_i$ in \eqref{eq:epsilon_def} and \eqref{eq:density_definition},
    \begin{align}
        \label{eq:biskup_good_cluster_size_1}
        \#
        \goodcluster{i}{j}
        & \geq
        \rho_{i-1}
        \sum_{j \leq (1 - \epsilon_i) 
        \numsub{i}}
        \# \Tile{i-1}{j}        
        \geq 
        \rho_{i-1}
        \left(
            \# \Tile{i}{j}
            -
            \epsilon_i \numsub{i} \max_j \# \Tile{i-1}{j}
        \right)
        \\
        \label{eq:biskup_good_cluster_size_2}
        & \geq 
        \rho_{i-1} 
        \# \Tile{i}{j}
        \left(
            1
            -
            \frac{1}{(i+1)^2}
        \right)
        \geq
        \rho_i 
        \# \Tile{i}{j}
        \geq 
        \rho_{\infty} 
        \# \Tile{i}{j},
    \end{align}
    where in the last inequality we use that $\rho_i > \rho_{\infty}$ for all $i \geq 1$.
    Suppose now that $\goodcluster{i}{j}$ and $\goodcluster{i}{k}$ are two distinct good clusters in good tiles. 
    The proof that $d_G(\goodcluster{i}{j},\goodcluster{i}{k}) \geq r_0/2$ is just as in \thref{lem:local_giant_density}, now noting that \eqref{eq:biskup_good_cluster} does restrict the good cluster to be contained in $B(x_{0,j}, r_0/2)$, and the later clusters are obtained by merging these components.
\end{proof}

We also establish some analytic consequences of the choice of scales in \thref{def:renorm_scheme_2} which we will require in the induction.

\begin{lemma}
\thlabel{lem:useful_bounds}
    Let $\alpha \in(0,2)$ and $\delta \in (\max(0, 1-\alpha),2-\alpha)$.
    For all $c > 0$ we can choose $r_0$ sufficiently large such that  for all $i \geq 1$
    \begin{align}
        \label{eq:useful_bound_1}
        \epsilon_i 
        \numsub{i} 
        r_{i-1}^{d(2 - \alpha - \delta)} 
        & \geq 
        c 
        r_{i}^{d(2 - \alpha - \delta)},
        \\
        \label{eq:useful_bound_2}
        (r_i / r_0)^{d(2 -\alpha - \delta)} 
        & \geq 
        c
        \log(e/\epsilon_{i+1}),
        \\
        \label{eq:useful_bound_3}
        \log
        (
            6 
            (
                e 
                \numsub{i}
            )^2
        )
        -
        \numsub{i}
        r_{i-1}^{2d}/r_i^{d \alpha}
        & \geq 
        c
        (r_i/r_0)^{d(2-\alpha-\delta)}.
    \end{align}
\end{lemma}

\begin{proof}
    Recursively applying $m_i = r_i / r_{i-1}$ and the definition of $m_i$ in \eqref{eq:biskup_mi} yields
    \begin{equation}
        \label{eq:unpacking_mi_1}
        m_i 
        =
        r_0^{
            \left(
                \frac{1}{1 - \delta/2}
            \right)^{i-1}
            \left(
                \frac{\delta /2}{1 - \delta /2}
            \right)
        }
    \end{equation}
    for $i \geq 1$. 
    The term $m_i$ increases doubly-exponentially in $i$, whereas $\epsilon_i = 1 / ( \Cvor^2(i+1)^2)$ decays polynomially in $i$.
    By \thref{prop:tile_and_cell_volume_bounds} and using that $r_i = m_i r_{i-1}$, we have
    \begin{equation}
        \label{eq:unpacking_mi_2}
        \frac{
            \epsilon_i 
            \numsub{i} 
            r_{i-1}^{d(2 - \alpha - \delta)} 
        }{
            r_{i}^{d(2 - \alpha - \delta)}
        }
        =
        \frac{
            \cvor{} 
            \epsilon_i 
            m_i^d
            r_{i-1}^{d(2 - \alpha - \delta)} 
        }{
            (m_i r_{i-1})^{d(2 - \alpha - \delta)}
        }
        \geq 
        \cvor{} 
        \epsilon_i 
        m_i^{d(-1 + \alpha + \delta)}.
    \end{equation}
    Since $\delta \in (\max(0,1-\alpha), 2-\alpha)$, we have $-1 + \alpha + \delta > 0$ and $r_0$ may be chosen sufficiently large such that $\cvor{} \epsilon_i m_i^{d(-1 + \alpha + \delta)}$ is arbitrarily large for all $i \geq 1$.
    This proves \eqref{eq:useful_bound_1}. 
    Similarly, 
    \begin{equation}
        \label{eq:unpacking_mi_3}
        \frac{
            r_{i}^{d(2 - \alpha - \delta)}
        }{
            \log(e/\epsilon_{i+1})
            r_0^{d(2 - \alpha - \delta)}
        }
        \geq
        \frac{
            m_{i}^{d(2 - \alpha - \delta)}
        }{
            \log(e/\epsilon_{i+1})
        }
        \geq 
        \epsilon_{i+1} 
        m_{i}^{d(2 - \alpha - \delta)}
    \end{equation}
    where $2 - \alpha - \delta > 0$ and $r_0$ may be chosen sufficiently large such that $\epsilon_{i+1} m_{i}^{d(2 - \alpha - \delta)}$ is arbitrarily large for all $i \geq 1$.
    This proves \eqref{eq:useful_bound_2}.
    We now prove \eqref{eq:useful_bound_3}.
    By \thref{prop:tile_and_cell_volume_bounds} we have that $\cvor{} m_i^d \leq \numsub{i} \leq \Cvor{} m_i^d$.
    By construction we have that $m_i^d = (r_i/r_{i-1})^d \leq (r_i/r_0)^d$, and we let $x = m_i^d$.
    The function $f(x) = \log(6 (e \Cvor{} x)^2)$ is slowly varying, while the function $g(x) = c_1 x^{2 - \alpha - \delta}$ is a polynomial with positive exponent, so that by Potter's theorem (see for instance \cite[Chapter 1.4]{bingham_regular_1989}) it is the case that $f(x) \leq g(x)$ for all $x$ sufficiently large.
    As we have set $x = (r_i/r_{i-1})^d \geq (r_1/r_0)^d = r_0^{d \xi}$, we can choose $r_0$ sufficiently large such that 
    \begin{equation}   
        \label{eq:log_kappa_bound}
        \log
        (
            6
            ( 
                e 
                \numsub{i}
            )^2
        )
        \leq 
        c_1 (r_i/r_0)^{d(2-\alpha-\delta)}
    \end{equation}
    for all $i \geq 1$.
    Since $(d \delta / 2)(1/(1-\delta/2))^i > 0$ and $d(2 - \alpha - \delta) > 0$, we can choose $r_0$ sufficiently large such that 
    \begin{equation}
        r_0^{
            (
                d 
                \delta/2 
            )
            (
                1/(1-\delta/2)
            )^i 
            + 
            d(2-\alpha - \delta)
        }
        \geq
        2
        c_1 
        /
        \cvor{}
    \end{equation}
    for all $i \geq 1$.
    By the definitions of $m_i$ and $r_i$ we have $r_i = r_0^{(1/(1-\delta/2))^i}$ and also $m_i \leq r_{i-1}^{(\delta/2)/(1-\delta/2)}$, so that $r_i \leq r_{i-1}^{1/(1-\delta/2)}$ and hence $r_i^{d(1-\delta/2)} \leq r_{i-1}^d$.
    It follows that for this choice of $r_0$ 
    \begin{equation}
        r_i^{d(1-\delta/2)}
        \geq 
        \frac{
            2
            c_1
            r_i^{d(1-\delta)}    
        }{
            \cvor{} 
            r_0^{d(2 - \alpha - \delta)}
        }
    \end{equation}
    for all $i \geq 1$. 
    Multiplying both sides by $r_i^{d(1-\alpha)}$, rearranging and using using the definition of $m_i = r_i/r_{i-1}$ and the bounds for $\numsub{i}$ yields for $\delta> \max(0, 1-\alpha)$ by \eqref{eq:useful_bound_3}
    \begin{equation}
        \label{eq:kappa_ratio_bound}
        \numsub{i}
        r_{i-1}^{2d}/r_i^{d \alpha}
        \geq 
        2
        c_1
        (r_i/r_0)^{d(2-\alpha-\delta)}.
    \end{equation}
    The conclusion follows from \eqref{eq:log_kappa_bound} and \eqref{eq:kappa_ratio_bound}. 
\end{proof}

We use \thref{prop:almost_linear} to initialise the induction.

\begin{lemma}[Induction base]
\thlabel{lem:biskup_induction_base}
    Let $G$ be a transitive graph of polynomial growth with $d \geq 1$ and suppose that $J : V \times V \to \R_+$ is a transitive kernel satisfying $J(x,y) = \Omega (d_G(x,y)^{-d \alpha})$ with $\alpha \in (0,2)$.
    Let $\beta > \beta_c$.
    Let $((\Tile{i}{j})_{j \geq 0})_{i \geq 0}$ be the renormalisation scheme in \thref{def:renorm_scheme_2}. 
    For all $c > 0$ we can choose $r_0$ sufficiently large such that for any tile $\Tile{0}{j}$
    \begin{equation}
        \label{eq:biskup_level_0_bad}
        \P(\Tile{0}{j} \text{ is bad})
        \leq 
        \exp(-c).
    \end{equation}
\end{lemma}

\begin{proof}
    By construction the tile $\Tile{0}{j}$ is contained in the ball $B(x_{0,j},r_0)$. 
    Let $\epsilon > 0$.
    If $\Tile{0}{j}$ is good, then by \thref{lem:biskup_size}, by the definition of $\rho_0$ and choosing $\epsilon > \delta/4$, and by \thref{prop:tile_and_cell_volume_bounds},
    \begin{equation}
        \label{eq:biskup_level_0_giant}
        \# \cluster_{r_0}^{(1)}
        \geq 
        \# \goodcluster{0}{j}
        \geq 
        \rho_0 
        \# \Tile{0}{j}
        = 
        \theta 
        r_0^{-d \delta/4} \# \Tile{0}{j} / \cvor{}
        \geq
        \theta
        r_0^{d(1-\epsilon)}.
    \end{equation}
    Let $c > 0$.
    By \thref{prop:almost_linear} there exists $r_0$ sufficiently large such that
    \begin{equation}
        \label{eq:biskup_level_0_conclusion}
        \P_{\beta}
        \left(
            \Tile{0}{j} \text{ is good}
        \right)
        \geq 
        \P
        \left(
            \# \cluster_{r_0}^{(1)}
            \geq 
            \theta
            r_0^{d(1-\epsilon)}
        \right)
        \geq 
        1 - \exp(-c),
    \end{equation}
    and taking the complement concludes the proof.    
\end{proof}

In the inductive step of the renormalisation, we use the following result on the connectivity properties of the Erd{\H o}s-R\'enyi random graph. 
A short proof of this statement may be found in \cite[Claim 4.7]{jorritsma_large_2025}. 
We write $f(n) = \omega(g(n))$ to mean that for all $c > 0$ there exists $n_0$ such that for all $n > n_0$, $f(n) > c g(n)$. 

\begin{proposition}
    \thlabel{prop:erdos_renyi_connectivity}
    Consider the \ER{} random graph $G(n,p_n)$ with $n p_n = \omega(\log n)$. Then 
    \begin{equation}
        \label{eq:erdos_renyi_connectivity}
        \P 
        \left(
            G(n,p_n)
            \text{ is not connected}
        \right)
        \leq 
        3 (en)^2 (1-p_n)^{n/2}
    \end{equation}  
    for all $n \geq 1$.
\end{proposition}

\begin{lemma}[Iterative renormalisation]
    \thlabel{lem:second_iterative_renormalisation}
    Let $G$ be a transitive graph of polynomial growth with $d \geq 1$ and suppose that $J : V \times V \to \R_+$ is a transitive kernel satisfying $J(x,y) = \Omega (d_G(x,y)^{-d \alpha})$ with $\alpha \in (0,2)$.
    Let $\beta > \beta_c$, and let $((\Tile{i}{j})_{j \geq 0})_{i \geq 0}$ be the renormalisation scheme in \thref{def:renorm_scheme_2}.
    For $\delta \in (\max(0, 1-\alpha),2-\alpha)$ and $c > 0$, for  $r_0$ sufficiently large it holds for any tile $\Tile{i}{j}$ that 
    \begin{equation}
        \label{eq:biskup_iterative_renormalisation}
        \P_{\beta}
        (
            \Tile{i}{j} \text{ is bad}
        )
        \leq 
        \exp
        (
            -
            c
            (r_i/r_0)^{d(2 - \alpha - \delta)}
        ).
    \end{equation}
\end{lemma}

\begin{proof}
    We argue by induction on $i$. 
    We let $c_1 > 0$, and for $i \geq 0$ we write $\lambda_i = \max_j \P_{\beta} \left(\Tile{k}{j} \text{ is bad}\right)$ and $\error{k} = \exp(-c_1(r_k/r_0)^{d(2-\alpha - \delta)})$.
    For $i = 0$, it follows from \thref{lem:biskup_induction_base} that we can choose $r_0$ sufficiently large such that $\lambda_0 \leq \error{0}$.
    Now let $\Tile{i}{j}$ be a \level{i} tile and let $\lambda_{i,j} = \P_{\beta}(\Tile{i}{j} \text{ is bad})$.
    Similarly to as in \thref{lem:first_iterative_renormalisation}, let $\phi_{i,j}$ be the probability that less than a $(1 - \epsilon_i)$ proportion of the sub-tiles of $\Tile{i}{j}$ are good, and let $\varphi_{i,j}$ be the probability that the good clusters are not connected provided that there are enough good subtiles. 
    This gives $\lambda_{i,j} \leq \phi_{i,j} + \varphi_{i,j}$, and we bound each probability in turn.
    Since disjoint tiles are independent, the induction hypothesis gives that the number of bad subtiles is stochastically dominated by a $\Bin(\numsub{i},\error{i-1})$ random variable, so that
    \begin{equation}
        \label{eq:binomial_domination}
        \phi_{i,j} 
        \leq 
        \P
        (
            \Bin(\numsub{i},\error{i-1}) \geq \epsilon_i \numsub{i}
        )
        \leq 
        \sum_{k = \lceil \epsilon_i \numsub{i} \rceil}^{\numsub{i}} 
        \binom{\numsub{i}}{k} 
        \error{i-1}^k
        \leq
        \sum_{k = \lceil \epsilon_i \numsub{i} \rceil}^{\infty} 
        \left(
            \frac{e}{\epsilon_i}
            \error{i-1}
        \right)^k
    \end{equation}
    where we use that $\binom{\numsub{i}}{k} \leq (e \numsub{i} / k)^k \leq (e/ \epsilon_i)^k$ for all $k \geq \lceil \epsilon_i \numsub{i} \rceil$.
    We show that this geometric sum is convergent.
    By \thref{lem:useful_bounds} we may choose $r_0$ sufficiently large such that
    \begin{align}
        \label{eq:useful_bound_4}
        \log
        \left(
            e/\epsilon_k
        \right)
        -
        c_1
        \left(
            r_{k-1}/r_0
        \right)^{d(2-\alpha - \delta)}
        & < 
        - 
        c_1
        \left(
            r_{k-1}/r_0
        \right)^{d(2-\alpha - \delta)}
        /2,
        \\
        \log 4 
        - 
        c_1
        \epsilon_k \numsub{k} r_{k-1}^{d(2-\alpha - \delta)}
        /
        2
        & < 
        \label{eq:useful_bound_5}
        -
        c_1
        r_{k}^{d(2-\alpha - \delta)}
    \end{align}
    for all $k \leq i$.
    For this choice of $r_0$, the inequality in \eqref{eq:useful_bound_4} gives
    \begin{align}
        \label{eq:biskup_geometric_error}
        \frac{e}{\epsilon_i} \error{i-1} 
        & \leq 
        \exp
        \left(
            -
            c_1
            \left(
                r_{i-1}/r_0
            \right)^{d(2 - \alpha - \delta)}
            / 
            2
        \right)
        \leq
        \frac{1}{2},
    \end{align}
    and evaluating the geometric series in \eqref{eq:binomial_domination} together with the bound in \eqref{eq:useful_bound_5} gives
    \begin{equation}
        \label{eq:biskup_phi_bound}
        \phi_{i,j} 
        \leq 
        2
        \exp
        \left(
            - 
            c_1
            \epsilon_i \numsub{i}
            \left(
                r_{i-1}/r_0
            \right)^{d(2 - \alpha - \delta)}
            / 
            2
        \right)
        \leq 
        \frac{1}{2}
        \exp
        \left(
            - 
            c_1
            \left(
                r_{i}/r_0
            \right)^{d(2 - \alpha - \delta)}
        \right),
    \end{equation}
where we used \eqref{eq:useful_bound_1} (equvalently, \eqref{eq:useful_bound_5}) to obtain the last inequality.  
    We now bound $\varphi_{i,j}$ by considering the random graph $\GG$ in which each vertex corresponds to the good cluster of a good subtile, and two vertices are connected by an edge if and only if the two associated good clusters are connected.
    Let $\Tile{i-1}{k}$ and $\Tile{i-1}{\ell}$ be two distinct good subtiles of $\Tile{i}{j}$.
    By \thref{lem:biskup_size} we have that $\goodcluster{i-1}{k}$ and $\goodcluster{i-1}{\ell}$ are disjoint with $d_G(\goodcluster{i-1}{k},\goodcluster{i-1}{\ell}) \geq r_0/2$, as well as $\goodcluster{i-1}{k},\goodcluster{i-1}{\ell} \subset B(\Cvor{} r_i)$ and $\# \goodcluster{i-1}{k}, \# \goodcluster{i-1}{\ell} \geq \rho_{\infty} \cvor{} r_{i-1}^d$.
    Choosing $r_0 \geq R_J$ we have
    \begin{align}
        \label{eq:tile_connection}
        \P_{\beta}
        \left(
            \goodcluster{i-1}{k}
            \not\sim
            \goodcluster{i-1}{\ell}
            \mid
            \Tile{i-1}{k},
            \Tile{i-1}{\ell}
            \text{ are good}
        \right) 
        & \leq
        \exp
        \left(  
            -
            \frac{
                c_2 
                \left(
                    \rho_{\infty} 
                    \cvor{} 
                    r_{i-1}^d
                \right)^2
            }{
                \left(
                    2
                    \Cvor{}
                    r_i
                \right)^{d \alpha}
            }
        \right)
        \\
        & \leq 
        \exp
        \left(  
            -
            c_3
            r_{i-1}^{2d}
            /
            r_i^{d \alpha}
        \right)
    \end{align}
    for some $c_2 = c_2(\beta,J) > 0$ and $c_3 = c_3(\beta,G,J) > 0$.
    For convenience we write $1 - q_i = \exp\left(- c_3 r_{i-1}^{2d}/r_i^{d \alpha}\right)$, and we write $\NN_{i,j}$ for the number of good subtiles in $\Tile{i}{j}$, so that in particular $\lceil(1-\epsilon_i) \numsub{i} \rceil \leq \NN_{i,j} \leq \kappa_{i,j}$.
    Given $\NN_{i,j}$, the random graph $\GG$ stochastically dominates an \ER{} random graph $G(\NN_{i,j},q_i)$. 
    \thref{prop:erdos_renyi_connectivity} bounds the probability that such a graph is not connected, and we obtain 
    \begin{align}
        \label{eq:biskup_erdos_renyi_lemma_1}
        \varphi_{i,j}
        \leq 
        \P_{\beta}
        \left(
            G(\NN_{i,j},q_i) 
            \text{ is not connected} 
            \mid 
            \NN_{i,j}
        \right)
        & \leq
        3
        (e \numsub{i})^2 
        \exp 
        \left( 
            -
            c_3
            r_{i-1}^{2d}/r_i^{d \alpha}
        \right)^{\lceil (1-\epsilon_i) \numsub{i} \rceil}
        \\
        & \leq
        \label{eq:biskup_erdos_renyi_lemma_3}
        \frac{1}{2}
        \exp
        (
            \log
            (
                6 
                \left(
                    e 
                    \numsub{i}
                \right)^2
            )
            -
            c_4
            \numsub{i}
            r_{i-1}^{2d}/r_i^{d \alpha}
        )
    \end{align}
    for some $c_4 > 0$, where in \eqref{eq:biskup_erdos_renyi_lemma_3} we use that $\epsilon_{i} \leq 1/(4 \Cvor{}^2)$ and hence $1 - \epsilon_i \geq 3/(4\Cvor{}^2)$.
    By the bounds in \thref{lem:useful_bounds} we can choose $r_0$ sufficiently large such that
    \begin{equation}
        \label{eq:induction_varphi_conclusion}
        \varphi_{i,j}
        \leq 
        \frac{1}{2}
        \exp
        \left(
            - 
            c_1
            (r_i/r_0)^{d(2-\alpha-\delta)}
        \right).
    \end{equation}
    Together with \eqref{eq:biskup_phi_bound} we can finish the induction  for any tile $\Tile{i}{j}$ by noting that
    \begin{equation}
        \lambda_{i,j}
        \leq
        \phi_{i,j}
        + 
        \varphi_{i,j}
        \leq
        \exp
        \left(
            - 
            c_1
            (r_i/r_0)^{d(2-\alpha-\delta)}
        \right)
    \end{equation}
    \vskip-1em
\end{proof}

\begin{proof}[Proof of \thref{prop:biskup_with_error}]
    By \thref{lem:second_iterative_renormalisation}, for any $c_1 > 0$ we can choose $r_0$ sufficiently large such that for any tile $\Tile{i}{j}$,
    \begin{equation}
        \label{eq:biskup_arbitrary_radii_bad}
        \P_{\beta}
        (
            \Tile{i}{j} \text{ is bad}
        )
        \leq 
        \exp
        (
            -
            c_1
            (r_{i}/r_0)^{d(2 - \alpha - \delta)}
        ).
    \end{equation}
    Fix such a choice of $r_0$ and for $r > \Cvor{} r_1$ let
    \begin{equation}
        i_{*} 
        = 
        \max
        \{
            i \geq 1
            : 
            \Cvor
            r_i
            < r 
        \},
        \quad 
        \tilde{r} 
        = 
        r 
        - 
        \Cvor{}
        r_{i_*}
        .
    \end{equation}
    The choice of $\tilde{r}$ is such that the \level{i_*} tiles whose centre is in $B(0,\tilde{r})$ are contained in $B(0,r)$. 
    By \thref{lem:biskup_size}, a good tile $\Tile{i_*}{j}$ contains a good cluster $\goodcluster{i_*}{j}$ with $\# \goodcluster{i_*}{j} \geq \rho_{\infty} \# \Tile{i_*}{j}$.
    Let $\kappa$ be the number of \level{i_*} tiles whose centre is in $B(\tilde{r})$.
    We distinguish two cases according to the value of $\kappa$. 
    Let $M > 0$ be the least integer such that $c_G M^d / (3^d C_G) \geq 2$.
    Suppose in the first instance that $\tilde{r} < M r_{i_*}$ so that in particular $r_{i_*} > r/(M + \Cvor{})$. 
    Without loss of generality, we can translate the net $V_{i_*}$ so that $0 \in V_{i_*}$. It follows that the ball $B(\tilde{r})$ contains the centre of at least one \level{i_*} tile, say $\Tile{i_*}{j}$. If $\Tile{i_*}{j}$ is good, the good cluster $\goodcluster{i_*}{0}$ satisfies 
    \begin{equation}
        \# \goodcluster{i_*}{0}
        \geq 
        \rho_{\infty}
        \# \Tile{i_*}{0}
        \geq 
        \rho_{\infty} 
        \cvor{}
        r_{i^*}^d
        \geq 
        \rho_{\infty}
        \cvor{} 
        r^d 
        / 
        (M + \Cvor{})^d
        \geq 
        \nu
        \# B(r)
    \end{equation}
    for some $\nu = \nu(G) > 0$.
    Since by construction $\goodcluster{i_*}{0}$ is contained in $B(r)$, the size of $\goodcluster{i_*}{0}$ is a lower bound for $\# \giantone{r}$.
    Together with \eqref{eq:biskup_arbitrary_radii_bad} this gives for all $r$ sufficiently large and some $c_2 > 0$ that
    \begin{align}
        \P_{\beta}
        \left(
            \# 
            \giantone{r}
            \leq 
            \nu
            \# B(r)
        \right) 
        \leq
        \P_{\beta}
        \left(
            \Tile{i_*}{j} \text{ is bad}
        \right)
        & \leq 
        \exp
        \left(
            -
            c_1
            (r_{i_*}/r_0)^{d(2 - \alpha - \delta)}
        \right)
        \\
        & \leq 
        \exp
        \left(
            - 
            c_2
            \# B(r)^{2-\alpha-\delta}
        \right).
    \end{align}
    Suppose now that $\tilde{r} \geq M r_{i_*}$, so that in particular $\tilde{r} \geq r M / (1 + \Cvor{})$.
    In this case, $B(\tilde{r})$ contains the centres of 
    \begin{equation}
        \label{eq:silly_label_2}
        \kappa 
        \geq 
        \left\lfloor
            \frac{
                \min \# B(\tilde{r})
            }{
                \max_j \# \vor{\netvertex{i_*}{j}}
            }
        \right\rfloor
        \geq
        \frac{
                c_G 
                \tilde{r}^d
            }{
                3^d
                C_G
                r_{i_*}^d
            }   
        \geq 
        \frac{
                c_G
                M^d
                r^d
            }{
                \cdot 3^d
                C_G
                \left(
                    M
                    +
                    \Cvor{}
                \right)^d   
                r_{i_*}^d
            }
    \end{equation}
    many \level{i_*} cells. 
    By the choice of $M$, the third expression from the right gives $\kappa \geq 2$.
    We follow an argument similar to that in \thref{lem:second_iterative_renormalisation}.
    Let $\lambda$ be the probability that at least half of the cells are good and that their good clusters are connected to form the cluster $\CC$. If this is the case then
    \begin{equation}
        \# 
        \giantone{r}
        \geq 
        \# \CC
        \geq 
        \frac{
            \kappa 
            \rho_{\infty} 
            \Tile{i_*}{j}
        }{2}
        \geq
        \frac{
            \kappa 
            \rho_{\infty} 
            \cvor{} 
            r_{i_*}^d
        }{2}
        \geq
        \frac{
            \cvor{} 
            c_G 
            \rho_{\infty}
            M^d 
            r^d
        }{
            2 \cdot 3^d
            C_G
            \left(
                M 
                +
                \Cvor
            \right)^d
        }
        \geq
        \nu \# B(r),
    \end{equation}
    for some $\nu = \nu(G) > 0$.
    Let $\phi$ be the probability that less than half of the associated cells are good, and let $\varphi$ be the probability that the good clusters are not connected given that there are enough good \level{i_*} tiles. 
    This gives $\lambda \geq 1 - \phi - \varphi$, and we bound each probability in turn. 
    Writing $\error{i_*} = \exp(- c_3(r_{i_*}/r_0)^{d(2 - \alpha - \delta)})$, the same argument as in \thref{lem:second_iterative_renormalisation} (specifically that from \eqref{eq:binomial_domination} to \eqref{eq:biskup_phi_bound}) gives that we can choose $r_0$ sufficiently large such that
    \begin{equation}
        \label{eq:silly_label}
        \phi
        \leq 
        \P(\Bin(\kappa,\error{i_*}) \geq \kappa/2)
        \leq 
        2 
        \exp
        \left(
            \frac{
                - 
                c_1
                \kappa 
                r_{i_*}^{d(2-\alpha-\delta)}    
            }
                {2 r_{0}^{d(2-\alpha-\delta)}
            }
        \right)
        \leq 
        \frac{1}{2}
        \exp
        \left(
            - 
            c_3
            \#B(r)^{2 - \alpha - \delta}
        \right)
    \end{equation}
    for some $c_3 > 0$, where in the last inequality we use the bound on $\kappa r_{i_*}^d$ from \eqref{eq:silly_label_2} followed by a similar calculation as done in \eqref{eq:unpacking_mi_2} but without the $\varepsilon_i$, using that $\kappa r_{i_*}^d =\Theta(r^d)$ and that $\delta>\max(0, 1-\alpha)$.
    
    We now bound $\varphi$.
    The same argument as in \thref{lem:second_iterative_renormalisation} (specifically that from \eqref{eq:tile_connection} to \eqref{eq:induction_varphi_conclusion}) and using that $r_{i_* + 1} \geq r/(1+\xi)$ gives that we can choose $r_0$ sufficiently large such that
    \begin{equation}
        \varphi
        \leq 
        \frac{1}{2}
        \exp
        \left(
            \log 
            (
                6 
                (
                    e 
                    \kappa
                )^2
            )
            -
            c_5 
            \kappa
            r_{i_* + 1}^{2d}
            /
            r^{d \alpha}
        \right)
        \leq 
        \frac{1}{2}
        \exp
        \left(
            - 
            c_3
            \#B(r)^{2 - \alpha - \delta}
        \right),
    \end{equation}
    by choosing $c_3$ sufficiently small. As the total error is $\phi+\psi$, we have shown that
    \begin{equation}
        \P(
            \# 
            \giantone{r}
            \leq 
            \nu
            \# 
            B(r)
        ) 
        \leq 
        \exp
        \left(
            -
            c_3
            B(r)^{2 - \alpha - \delta}
        \right)
    \end{equation}
    for all $r$ sufficiently large, which concludes the proof.
\end{proof}

\subsection{Sharp bounds for the giant}
We are now ready to prove \thref{prop:biskup}. 
Comparing the statements of \thref{prop:biskup} and \thref{prop:biskup_with_error}, we need to work off the error term to get sharp bounds on the existence of the giant.
The proof uses a large deviations result of the \ER{} random graph due to O'Connell \cite{oconnell_large_1998}.

\begin{proof}[Proof of \thref{prop:biskup}]
    Fix $\alpha\in(0,2)$. Let $r \in \N$ and let $M > 0$ be a large constant. 
    For $R = M^{1/d}r^{\max(\alpha-1,0)}$ let $V_0$ be a maximal $R$-separated net, and for $\epsilon > 0$ let $r$ be sufficiently large such that $r - R \geq r(1-\epsilon)$.
    We write $(\Tile{R}{j})_{j \in \N}$ for the Voronoi tiles associated to the net $V_0$.
    As before $\diam (\Tile{R}{j}) \leq 2R$ and $\# \Tile{R}{j} \asymp R^d$, and in particular tiles with centre in $B(r - R)$ are contained in $B(r)$.
    For $\nu_1 = \nu_1(G) > 0$ as in \thref{prop:biskup_with_error}, we say that a tile $\Tile{R}{j}$ is good if it contains a good cluster $\CC_j$ with $\CC_j \subseteq B(R/4) \subseteq \Tile{R}{j}$ and $\# \CC_j \geq \nu_1 \# B(R/4)$. 
    For $\xi > 0$, it follows from \thref{prop:biskup_with_error} that we may choose $r$ (or $M$, in case $\alpha\le 1$) sufficiently large such that
    \begin{equation}
        \P_{\beta}
        \left(
            \Tile{R}{j}
            \text{ is good}
        \right)
        \geq 
        1 - \xi/2
    \end{equation}
    for any tile $\Tile{R}{j}$.
    Let $\kappa$ denote the number of tiles with centre in $B(r - R)$. 
    We have $\kappa \geq \lfloor \min \# B(r-R) / \max \# \Tile{R}{j} \rfloor$, and hence $\kappa \geq c_1 \# B(r)^{\min(2 - \alpha,1)} / M$ and $\kappa \# B(R) \geq c_2 \# B(r)$ for some $c_1 = c_1(G,\epsilon)>0 $ and $c_2 = c_2(G,\epsilon) > 0$.
    Let $\lambda$ be the probability that more than a $(1 - \xi)$ proportion of the tiles with centres in $B(r - R)$ are good, and that at least a $(1 - \xi)$ proportion of their good clusters are connected to form the cluster $\CC$.
    If this is the case then
    \begin{equation}
        \# 
        \giantone{r}
        \geq 
        \# 
        \CC
        \geq 
        (1 - \xi)^2
        \kappa
        \nu_1
        \# B(R/4)
        \geq 
        \nu_2
        \# B(r)
    \end{equation}
    for some $\nu_2 = \nu_2(G,M,\epsilon,\xi) > 0$.
    Let $\phi$ be the probability that more than a $\xi$ proportion of tiles are bad, and let $\varphi$ be the probability that less than a $(1-\xi)$ proportion of the good clusters are connected given that there are enough good tiles.
    This gives $\lambda \leq \phi + \varphi$, and we bound each probability in turn.
    Since tiles are disjoint, the number of bad tiles is stochastically dominated by a $\Bin(\kappa,\xi/2)$ random variable. 
    By Chernoff's bound and by the lower bound for $\kappa$ we have
    \begin{equation}\label{eq:phi-bound-in-O}
        \phi
        \leq 
        \P
        \left(
            \Bin
            \left(
                \kappa,\xi/2
            \right) 
            \geq 
            \xi \kappa
        \right)
        \leq 
        \exp
        \left(
            - 
            \xi 
            \kappa 
            / 
            8
        \right)
        \leq 
        \exp
        \left(
            - 
            c_3
            \right(
                \#B(r)
            \left)^{\min(2-\alpha,1)}
        \right)
    \end{equation}
    for some $c_3 = c_3(G) > 0$.
    We now bound $\varphi$.
    We first consider the probability that the good clusters $\CC_j$ and $\CC_k$ in two distinct good tiles $\Tile{R}{j}$ and $\Tile{R}{k}$ do not connect.
    The distance between two tiles is at most $r$, and choosing $R$ sufficiently large so as to satisfy the asymptotics of the kernel $J$ for $\alpha\in(1,2)$, and for some $c_4 = c_4(G,J,\beta) > 0$ and $c_5 =  c_4(G,J,\beta)> 0$  we have that
    \begin{equation}
        \P_{\beta}
        \left(
            \CC_j
            \not\sim
            \CC_k
            \mid
            \Tile{R}{j},
            \Tile{R}{k}
            \text{ are good}
        \right)
        \leq 
        \exp
        \left(
            - 
            c_4
            M^2
            \left(
                \# B(r)
            \right)^{\alpha - 2}
        \right)
        \leq 
        \exp
        \left(
            \frac{
                    -
                    c_5
                    M 
                }{
                    \kappa
                }
        \right) 
        \leq
        \frac{
                c_5
                M
            }{
                \kappa
            }
    \end{equation}
    where we choose $M$ sufficiently large such that $c_5 M / \kappa > 1$. 
    For $\alpha\in(0,1]$, we arrive to the same bound by a similar calculation using that $\alpha \leq 1$ and $\kappa>c_1 \#B(r)$:
    \begin{equation}
        \P_{\beta}
        \left(
            \CC_j
            \not\sim
            \CC_k
            \mid
            \Tile{R}{j},
            \Tile{R}{k}
            \text{ are good}
        \right)
        \leq 
        \exp
        \left(
            - 
            c_4
            M^2
            \left(
                \# B(r)
            \right)^{-\alpha}
        \right)
        \leq 
        \exp
        \left(
            \frac{
                    -
                    c_5
                    M 
                }{
                    \kappa
                }
        \right) 
        \leq
        \frac{
                c_5
                M
            }{
                \kappa
            }.
    \end{equation}
    Consider the random graph $\GG$ in which each vertex corresponds to the good cluster of a good tile, and two vertices are connected by an edge if the two associated good clusters are connected. 
    By our assumptions there are at least $N = (1-\xi) \kappa$ good tiles, and the random graph $\GG$ stochastically dominates an \ER{} random graph $G(N,c_6 M / N)$ for $c_6 = (1 - \xi)c_5$.
    We write $\CC^{\scriptscriptstyle (1)}$ for the largest connected component in $G(N,c_6 M / N)$, and we choose $M$ sufficiently large such that $c_6 M > 1$.
    By \cite{oconnell_large_1998}, for some $c_7 = c_7(\xi) > 0$ and $c_8 = c_8(G,\xi) > 0$,
    \begin{equation}
        \psi
        \leq
        \P
        \left(
            \# 
            \CC^{\scriptscriptstyle (1)}
            \leq 
            (1-\xi)
            N
        \right)
        \leq
        \exp
        \left(
            -c_7
            N
        \right)
        \leq 
        \exp
        \left(
            -
            c_8
            \left(
                \#
                B(r)
            \right)^{\min(2 - \alpha,1)}
        \right).   
    \end{equation}
    Combining this with the error bounds on  $\phi$ in \eqref{eq:phi-bound-in-O}, for some $c_9=c_9(G, M, \varepsilon, \xi)>0$,
    \begin{equation}
        \P_{\beta}
        \left(
            \giantone{r}
            \leq
            \nu_2
            \# 
            B(r)
        \right)
        \leq 
        \exp
        \left(
            -
            c_9
            \left(
                \# 
                B(r)
            \right)^{\min(2-\alpha,1)}
        \right)
    \end{equation}    
    for all $r$ sufficiently large.
    The result follows by setting $c_9$ sufficiently small.
\end{proof}
\section{Local uniqueness of the giant}
\label{sec:local_uniqueness}
In this section we prove \thref{prop:second_largest} on the local uniqueness of the giant. 
The proof follows a four-step revealment scheme where we progressively reveal the percolation configuration in the ball $B(r)$. 
The name of the game is to repeatedly apply \thref{prop:biskup}, and prove that it is unlikely that there is a large non-giant component.
We introduce the revealment scheme and its notation.

For $r \in \N$ and $k \in \N$, let $V$ be a maximal $k^{1/d}$-separated net, and let $(T_i)_{i \in \N}$ be the associated family of Voronoi tiles.
Each tile has diameter at most $k$ and volume at least $\cvol{} k$.
For the ball $B(r)$, consider the set $(T_i)_{i \in I}$ of Voronoi tiles whose centres are in $B(r - k^{1/d})$. 
By construction, the union of these tiles is contained in $B(r)$.
To obtain a partition of $B(r)$, for each $x \in B(r)$ not already in a tile, we assign $x$ to the nearest tile (chosen arbitrarily when there are several equidistant tiles). 
After this assignment procedure, the largest tile has volume at most $c_G 2^dk$ and diameter at most $4 k^{1/d}$. 
Importantly, the tiles fully partition $B(r)$.
We now reveal the percolation configuration in $B(r)$ in the following sequential manner.

In \textbf{Step 1} we reveal edges with both endpoints in a single tile $T_i$.
This reveals the largest cluster in each tile $T_i$, which which we denote by $\giantone{i}$.
We consider the graph where tiles are identified as vertices, with edges between vertices whose tiles neighbour each other, and we let $\mathrm{Tree}$ be a spanning tree of this graph.
In \textbf{Step 2} we reveal edges between vertices in $\giantone{i}$ and vertices in $\giantone{j}$ for each pair of tiles $T_i$ and $T_j$ whose associated vertices are connected in $\mathrm{Tree}$.

\begin{definition}
	For $\nu > 0$ as in \thref{prop:biskup}, we let $\rho = \nu c_G/3^d$, we say that a tile $T_i$ is \textbf{good} if $\# \giantone{i}\geq \rho k$, and we say that \textbf{Step 1} is good if all tiles are good.
	Provided that \textbf{Step 1} is good, we say that \textbf{Step 2} is \textbf{good} if all the largest clusters become connected to form the cluster $\CC$.
\end{definition}

By a union bound we have that
\begin{align}
\label{eq:beautiful_bound_1}
    \begin{split}
        \P_{\beta}
        \left( 
            \# \gianttwo{r} > k
        \right)
        & \leq
        \P_{\beta}
        \left(
            \left(
                \text{\textbf{Steps 1 and 2} are not good}
            \right)
        \right)
        \\
        & \qquad
        + 
        \P_{\beta}
        \left(
            \left(
                \# \gianttwo{r} > k
            \right)
            \cap
            \left(
                \text{\textbf{Steps 1 and 2} are good}
            \right)
        \right)
    \end{split}
\end{align}
In \textbf{Step 3} we reveal edges between vertices \textit{not} in $\giantone{i}$ and vertices \textit{not} in $\giantone{j}$ for each pair of tiles $T_i$ and $T_j$.
In \textbf{Step 4} we reveal all remaining edges.

\begin{definition}
	Provided that \textbf{Steps 1 and 2} are good, we say that \textbf{Step 4} is good if all clusters with at least $k$ vertices merge with $\CC$.	
\end{definition}

Supposing that \textbf{Steps 1 and 2} are good, the event $\# \gianttwo{r} > k$ implies that there exists a cluster $\CC^\star$ in $B(r)$ with $\# \CC^\star > k$ and $\CC^\star \neq \CC$, and hence
\begin{multline}
    \label{eq:beautiful_bound_2}
    \P_{\beta}
    \left( 
        \# \gianttwo{r}
        > 
        k
    \right)
    \leq
    \P_{\beta}
    \left(  
        \text{\textbf{Steps 1 and 2} are not good}
    \right)
    \\
    +
    \P_{\beta}
    \left(  
        \left(
            \text{\textbf{Step 4} is not good}
        \right)
        \cap 
        \left(
            \text{\textbf{Steps 1 and 2} are good}
        \right)
    \right).
\end{multline} 
It remains now to bound the probability that each of the steps is good.

\begin{proof}[Proof of \thref{prop:second_largest}]
    By the definition of conditional probability we have
    \begin{multline}
        \P_{\beta}
        \left( 
            \text{\textbf{Steps 1 and 2} are not good}
        \right)
        \
        \leq 
        \P_{\beta}
        \left(
            \text{\textbf{Step 1} is not good}
        \right)
        \\
        +
        \P_{\beta}
        \left(
            \text{\textbf{Step 2} is not good}
            \mid
            \text{\textbf{Step 1} is  good}
        \right).
    \end{multline}
    As $T_i$ is a tile of a scale-$k^{1/d}$-net, a ball of radius at least $k^{1/d}/3$ is contained in the tile and its giant is $\nu\#B(k^{1/d}/3)\ge \nu c_G /3^d k $ with error probability given in \thref{prop:biskup}. So we have that there exists $c_1 > 0$ such that
    \begin{equation}
        \P_{\beta}
        \left(
            T_i
            \text{ is not good}
        \right)
        \leq
        \exp
        \left(
            -
            c_1\beta
            k^{\min(2 - \alpha,1)}
        \right)
    \end{equation}
    for any tile.
    Let $\kappa$ be the number of tiles in $B(r)$.
    Since tiles are disjoint, it follows that
    \begin{equation}
        \P_{\beta}
        \left(
            \text{\textbf{Step 1} is not good}
        \right)
        \leq 
        \kappa
        \exp
        \left(
            -
            c_1 
            k^{\min(2 - \alpha,1)}
        \right)
        \leq 
        \frac{\# B(r)}{k}
        \exp
        \left(
            -
            c_2 \beta
            k^{\min(2 - \alpha,1)}
        \right)
    \end{equation}
    for some $c_2 = c_2(G) > 0$, where $c_2$ absorbs also constants in the bound of $\kappa$.
    For two largest clusters $\giantone{i}$ and $\giantone{j}$ in two good neighbouring tiles $T_i$ and $T_j$ we have by `bulk-to-bulk' computation
    \begin{equation}
        \P_{\beta}
        \left(
            \giantone{i}
            \not\sim
            \giantone{j}
            \mid
            T_i
            \text{ and }
            T_j
            \text{ are good}
        \right)
        \leq 
        \exp
        \left(
            -
            c_3
            k^{2 - \alpha}
        \right)
    \end{equation}
    for some $c_3 = c_3(G) > 0$.
    The spanning tree $\mathrm{Tree}$ has $\kappa - 1$ many edges and we calculate
    \begin{align}
        \P_{\beta}
        \left(
            \text{\textbf{Step 2} is not good}
            \mid
            \text{\textbf{Step 1} is good}
        \right)
        & \leq 
        \frac{\# B(r)}{k}
        \exp
        \left(
            -
            c_4\beta
            k^{2 - \alpha}
        \right)
    \end{align}
    for some $c_4 = c_4(G) > 0$.
    Suppose now that \textbf{Step 1} and \textbf{Step 2} are good, and that we have revealed edges in \textbf{Step 3}. 
    If \textbf{Step 4} is not good, there exists a cluster $\CC^*$ in $B(r)$ which does not connect to $\CC$.
    Note that there are trivially at most $\# B(r) / k$ clusters in $B(r)$ with size greater than $k$.
    Each vertex in $\CC^*$ has an allocated neighbouring tile with respect to the labelling induced by the spanning tree $\mathrm{Tree}$, and we upper-bound the probability that $\CC^*$ does not connect to $\CC$ by the probability that each vertex in $\CC^*$ does not connect to its neighbouring largest cluster. 
    Since $V$ is a maximal $k^{1/d}$-separated net the distance between a vertex in $\CC^*$ and its neighbouring largest cluster is at most $4 k^{1/d}$, so that single vertex does not connect to its neighboring largest with probability at most $\exp(-c \beta \rho k^{1-\alpha})$.  The cluster $\CC^*$ has at least $k$ vertices, and so
    \begin{align}
        \P_{\beta}
        \left(\text{\textbf{Step 4} is not good
            }
            \mid
            \text{
                \textbf{Steps 1} and \textbf{2} are good
            }
        \right)
        & \leq 
        \frac{\# B(r)}{k}
        \left(
           \exp
            \left(
                -
                c_5\beta
                k^{1-\alpha}
            \right)
        \right)^{k}
        \\
        & =
        \frac{\# B(r)}{k}
        \exp
        \left(
            -
            c_5\beta 
            k^{2-\alpha}
        \right)
    \end{align}
    for some $c_5 = c_5(G,J) > 0$, where we use that $\# \giantone{i} / \# T_{i} \geq \rho k$.
    Together with the union bound in \eqref{eq:beautiful_bound_1}, for some $c_6 = c_6(G,J) > 0$, we obtain that
    \begin{equation}
        \P_{\beta}
        \left( 
            \# \gianttwo{R}
            > 
            k
        \right)
        \leq 
        \frac{\# B(r)}{k}
        \exp(-c_6 \beta k^{\min(2-\alpha,1)}).
    \end{equation}  
     concluding the proof. For this proof to work, we need at least $2$ different tiles fully contained inside $B(r)$. This can be achieved whenever $\#B(r)\le k$, as we may shrink the separation of the net to be a small constant times $k^{1/d}$ at the cost of changing the constant prefactors. For $\#B(r)\le k$, the second largest component can never be above size $k$, so the bound holds for all $r$ and $k$ values. 
\end{proof}
\section{\texorpdfstring{Transience with $\alpha \leq 2$}{Transience with α ≤ 2}}
\label{sec:transience}
In this section we prove \thref{thm:random_walks}.
For $\alpha \in (1,1 + 1/d)$, \thref{cor:anchored_isop_dim} gives a new proof of this result.
Indeed, for these values of $\alpha$, \thref{cor:anchored_isop_dim} gives that the infinite cluster $\kinfty$ has anchored isoperimetric dimension at least $2$, and it follows from a theorem of Thomassen \cite{thomassen_isoperimetric_1992} that graphs of anchored isoperimetric dimension strictly larger than 2 are transient.
To prove transience for the full range $\alpha \in (1,2)$, we follow \cite{berger_transience_2002}.
The strategy is to use a similar renormalisation argument to the one we used in \thref{prop:almost_linear} to show that almost surely $\kinfty$ contains a tree-like branching subgraph rooted at the origin and growing towards infinity, which we call a renormalised subgraph.
We then build a finite-energy flow on this renormalised subgraph.
It is a standard fact that a graph admitting a finite-energy flow is transient \cite{lyons_simple_1983}, see also \cite[Theorem 2.11]{lyons_probability_2016}.
We begin by choosing the renormalisation scales and defining the renormalisation scheme.

\begin{definition}
    \thlabel{def:transience_renorm_scheme}
    Let $\varepsilon > 0$, and let $r_0 \in \N$ be a large constant.
    We define $s=s(\varepsilon)$ to be the least integer satisfying simultaneously $2(1+s)^2\ge (1+\varepsilon)/\varepsilon$.
    We set $
    \bar r_0=r_0$, set $m_0=0$, and for $i \geq 1$ we iteratively define 
    \begin{equation}
        \label{eq:transience_r_def}
        m_i 
        =
        2 
        (i+s)^2
        \quad
        \text{and}
        \quad
        r_i
        =
        r_{i-1}
        m_i.
    \end{equation}
    Let $((\Tile{i}{j})_{j \geq 0})_{i \geq 0}$ be the renormalisation scheme associated to the renormalisation scales $(r_i)_{i \geq 0}$ in \thref{def:recursive_tiles}.
\end{definition}

Similarly to as in \thref{def:renormalisation_scheme,def:renorm_scheme_2}, this choice of renormalisation scales satisfies the hypotheses of \thref{prop:tile_and_cell_volume_bounds} with $\epsilon = \xi$.

\begin{definition}
    For $i \geq 0$ and for $\cvor{}$ as in \thref{prop:tile_and_cell_volume_bounds}, we let 
    \begin{equation}
        \label{eq:third_renormalisation}
        \rho_i 
        = 
            1 - (i+s)^{-3/2}/2
        ,
        \quad
        \xi_i
        = 
        (i +1+ s)^{-3/2}/4,
        \quad
        \text{and}
        \quad
        C_i
        =
        \cvor{}
        (m_i/2)^d.
    \end{equation}
    We say that a \level{0} tile $\Tile{0}{j}$ is \textbf{good} if it contains a cluster $\CC_{0,j}$ satisfying
    \begin{equation}\label{eq:good-part-transience}
        \# 
        \left(
            \CC_{0,j}
            \cap
            B(x_{0,j},r_0/4)
        \right)
        \geq 
        \nu 
        \# B(r_0/2),
    \end{equation}
    where $\nu$ is as in \thref{prop:biskup}.
    For a good \level{0} tile $\Tile{0}{j}$, the \textbf{good part} is $\CC_{0,j}
            \cap
            B(x_{0,j},r_0/4)$.
    We say that a \level{1} tile $\Tile{1}{j}$ is \textbf{good} if at least a $\rho_1$ proportion of its \level{0} subtiles are good, and their good parts are connected to each other inside $\Tile{1}{j}$.
    For a good \level{1} tile $\Tile{1}{j}$, we choose $C_1$ of its good \level{0} subtiles and the \textbf{good part} of $\Tile{1}{j}$ is the union of the good parts of these subtiles, denoted by $\CC_{1,j}$.
    For $i \geq 2$, we say that \level{i} tile $\Tile{i}{j}$ is \textbf{good} if at least a $\rho_i$ proportion of its \level{i-1} subtiles are good, and additionally: 
    if $\Tile{i-2}{j}$ is a good \level{i-2} subtile in a good \level{i-1} subtile of $\Tile{i}{j}$, 
    and $\Tile{i-2}{k}$ is a good \level{i-2} subtile in some other good \level{i-1} subtile of $\Tile{i}{j}$, 
    then the good parts of $\Tile{i-2}{j}$ and $\Tile{i-2}{k}$ are connected in $\Tile{i}{j}$.
    For a good \level{i} tile $\Tile{i}{j}$, we choose $C_i$ of its good \level{i-1} subtiles and the \textbf{good part} of $\Tile{i}{j}$ is the union of the good parts of these subtiles, denoted by $\CC_{i,j}$.
    When choosing the good part of $\Tile{i}{j}$, if the subtile containing the origin $o$ is good, we choose it to be part of the good part.
\end{definition}

By \thref{prop:tile_and_cell_volume_bounds} and the definitions of $\rho_i$ and $C_i$ in \eqref{eq:third_renormalisation}, we have $C_i \leq \rho_i \min \kappa_{i,j}$ for all $i \in \N$.
This guarantees that there are enough good \level{i-1} subtiles in each \level{i} tile so that we can choose $C_i$ of them to be in the good part.
Notice that if the origin $o$ is in the infinite component $\kinfty$ and if all tiles $T_{i,o}, i\in \N$ containing the origin are good, then the definition of the renormalisation scheme gives the the infinite cluster $\kinfty$ contains a tree-like branching subgraph rooted at the origin and growing towards infinity.
Accordingly, we say that the infinite cluster $\kinfty$ admits a \textbf{renormalised subgraph} if there is a vertex of the infinite cluster $\kinfty$ such that all tiles containing that vertex are good.
We now work towards showing that the infinite cluster $\kinfty$ contains a renormalised subgraph almost surely.
We begin by collecting some geometric facts.

\begin{lemma}
    \thlabel{lem:transience_geometry}
    Let $G$ be a transitive graph of polynomial growth with $d \geq 1$, and let $((\Tile{i}{j})_{j \geq 0})_{i \geq 0}$ be the renormalisation scheme in \thref{def:transience_renorm_scheme}.
    Let $\Tile{i}{j}$ be a good \level{i} tile and let $\Tile{i-1}{k}$ and $\Tile{i-1}{\ell}$ be two distinct good subtiles of $\Tile{i}{j}$.
    The good part $\CC_{i,j}$ satisfies $\# \CC_{i,j} \geq c_G \nu r_i^d (\cvor{}/2)^{id}$.
    The good parts $\CC_{i-1,k}$ and $\CC_{i-1,\ell}$ satisfy $d_G(\CC_{i-1,k},\CC_{i-1,\ell}) \geq r_0/2$ and $\CC_{i-1,k},\CC_{i-1,\ell} \subset \Tile{i}{j} \subseteq B(\netvertex{i}{j},\Cvor{}r_i)$.
\end{lemma}

\begin{proof}
    It follows from the definition of $C_i$ in \eqref{eq:third_renormalisation}, the definition of $r_i$ in \eqref{eq:transience_r_def}, and the inductive definition of the good part that
    \begin{equation}
        \# \CC_{i,j}
        \geq 
        \nu 
        \# 
        B(r_0/2)
        \prod_{k = 1}^i
        C_k
        \geq
        c_G 
        \nu
        r_i^d
        \prod_{k = 1}^i
        \left(
            \cvor{}
            /
            2
        \right)^d
    \end{equation}
    The proof that $d_G(\CC_{i-1,k},\CC_{i-1,\ell}) \geq r_0/2$ and $\CC_{i-1,k},\CC_{i-1,\ell} \subset \Tile{i}{j} \subseteq B(\netvertex{i}{j},\Cvor{}r_i)$ is just as in \thref{lem:biskup_size}.
\end{proof}

\begin{lemma}
    \thlabel{prop:transience_renorm_base}
    Let $G$ be a transitive graph of polynomial growth with $d \geq 1$, and suppose that $J : V \times V \to \R_+$ is a transitive kernel satisfying $J(x,y) = \Omega (d_G(x,y)^{-d \alpha})$ with $\alpha \in (1,2)$. 
    Let $\beta > \beta_{c}$. 
    We can choose $r_0$ sufficiently large such that for all tiles $\Tile{0}{j}$.
    \begin{equation}
        \P_{\beta}
        \left(
            \Tile{0}{j} 
            \text{ is good}
        \right)
        \geq
        1
        -
        \xi_0.
    \end{equation}
\end{lemma}

\begin{proof}
    As being good is defined in \eqref{eq:good-part-transience}, we can apply  \thref{prop:biskup} for $r_0$ sufficiently large.
\end{proof}

\begin{lemma}
    \thlabel{prop:renorm_subgraph_wts}
    Let $G$ be a transitive graph of polynomial growth with $d \geq 1$, and suppose that $J : V \times V \to \R_+$ is a transitive kernel satisfying $J(x,y) = \Omega (d_G(x,y)^{-d \alpha})$ with $\alpha \in (1,2)$. 
    Let $\beta > \beta_{c}$. 
    For $r_0$ as above,  for all tiles $\Tile{i}{j}$,
    \begin{equation}
        \P_{\beta}
        \left(
            \Tile{i}{j} 
            \text{ is good}
        \right)
        \geq
        1
        -
        \xi_i.
    \end{equation}
\end{lemma}

\begin{proof}
    We argue by induction on $i$, and we consider the probability that a tile is \textit{bad}. 
    For $i = 0$, the statement follows from \thref{prop:transience_renorm_base}.
    We may assume inductively that the statement holds for $k \leq i-1$.
    Let $\Tile{i}{j}$ be a \level{i} tile, let $\lambda_{i,j} = \P_{\beta}(\Tile{i}{j} \text{ is bad})$, let $\phi_{i,j}$ be the probability that less than a $\rho_i$ proportion of the \level{i-1} subtiles are good, and let $\varphi_{i,j}$ be the probability that the connection fails provided that there are enough good subtiles.
    This gives $\lambda_{i,j} \leq \phi_{i,j} + \varphi_{i,j}$, and we bound each probability in turn.
    Since disjoint tiles are independent, the induction hypothesis gives that the number of bad subtiles is stochastically dominated by a $\mathrm{Bin}(\kappa_{i,j},\xi_{i-1})$ random variable.
    Note that $1-\xi_{i-1} < (1 - \rho_i)/2$, and by Chernoff's bound 
    \begin{align}
        \phi_{i,j}
        \leq
        \P
        \left(
            \Bin
            (  
                \kappa_{i,j},
                \xi_{i-1}
            )
            \geq 
            (1 - \rho_i)
            \kappa_{i,j}
        \right)
        \leq
        \exp
        \left(
            - 
            c_1
            \kappa_{i,j}
            \frac{1-\rho_{i}}{2}
        \right)
        \leq
        \exp
        \left(
            - 
            c_2
            (i+s)^{2d - 3/2}
        \right)
    \end{align}
    for some $c_1 > 0$ and $c_2 > 0$.
    It follows that $\phi_{i,j}$ decays at least stretched-exponentially in $i$, whereas the term $\xi_i$ in \eqref{eq:third_renormalisation} decays polynomially in $i$, and we can choose $r_0$ sufficiently large such that $\phi_i \leq \xi_i/2$ for all $i \in \N$.
    
    We now bound $\varphi_{i,j}$ from above by the probability that the good parts in any \textit{any} two good \level{i-2} subtiles are not connected.
    Let $\Tile{i-2}{k}$ be a good \level{i-2} subtile in a good \level{i-1} subtile of $\Tile{i}{j}$, 
    and let $\Tile{i-2}{\ell}$ be a good \level{i-2} subtile in some other good \level{i-1} subtile of $\Tile{i}{j}$. 
    By \thref{lem:transience_geometry} we have that $\CC_{i-2,k}$ and $\CC_{i-2,\ell}$ are disjoint with $d_G(\CC_{i-2,k},\CC_{i-2,\ell}) \geq r_0 / 2$, as well as $\CC_{i-2,k}, \CC_{i-2,\ell} \subseteq B(\Cvor{} r_i)$.
    By this same lemma, for $\epsilon > 0$ we can choose $r_0$ sufficiently large such that $\# \CC_{i-2,k}, \#\CC_{i-2,\ell} \geq c_G \nu r_{i-2}^{d} (\cvor/2)^{(i-2)d}$.
    Further, we can choose $r_0$ sufficiently large such that $m_i^{2} \leq r_i^{\xi}$ for all $i \geq 1$. 
    This gives for some $c_1 = c_1(\beta,J) > 0$ and $c_2 = c_2(\beta,G,J) > 0$ that
    \begin{align}
        \P_{\beta}
        \left(
            \CC_{i-2,j}
            \not\sim
            \CC_{i-2,k}
            \mid 
            \Tile{i-2}{j},
            \Tile{i-2}{k} 
            \text{ are good}
        \right)
        & \leq 
        \exp
        \left(
            \frac{
                -
                c_1 
                \left(
                    c_G \nu r_{i-2}^{d} 
                    (\cvor/2)^{(i-2)d}
                \right)^2
            }{
                \left(
                    c
                    r_i
                \right)^{d \alpha}
            }
        \right)
        \\
        & \leq 
        \exp
        \left(
                -
                c_2 
                r_{i}^{d(2-\alpha - 2\xi)}
                (\cvor/2)^{2(i-2)d}
        \right).
    \end{align}      
    The tile $\Tile{i}{j}$ has at most $\max_j (\kappa_{i,j}) \max_j (\kappa_{i-1,j}) \leq \Cvor^2 m_i^{2d}$ \level{i-2} subtiles, and hence there are at most $\Cvor^4 m_i^{4d}$ pairs of \level{i-2} subtiles in $\Tile{i}{j}$.
    It follows that 
    \begin{equation}
        \varphi_i
        \leq
        \Cvor^4 m_i^{4d}
        \exp
        \left(
            -
            c_2
            r_{i}^{d(2 - \alpha - 2 \epsilon)}
            (\cvor/2)^{2(i-2)d}
        \right).
    \end{equation}    
    Since $\alpha \in (1,2)$, we may choose $\xi$ sufficiently small such that $d (2 - \alpha - \xi) > 0$, and hence by the definition of $r_i$ the term $r_{i}^{d(2 - \alpha - 2 \epsilon)} (\cvor/2)^{2(i-2)d}$ is at least linear in $i$.
    Similarly, by definition the term $m_i$ is polynomial in $i$.
    It follows that $\varphi_{i,j}$ decays at least stretched-exponentially in $i$.
    On the other hand, the term $\xi_i$ decays polynomially in $i$, and we can choose $r_0$ sufficiently large such that $\varphi_i \leq \xi_i/2$ for all $i \in \N$.
    We have shown that $\lambda_{i,j} \leq \phi_{i,j} + \varphi_{i,j} \leq \xi_i$, concluding the induction.
\end{proof}

\begin{lemma}
    \thlabel{prop:renormalised_graph}
    Let $G$ be a transitive graph of polynomial growth with $d \geq 1$, and suppose that $J : V \times V \to \R_+$ is a transitive kernel satisfying $J(x,y) = \Omega (d_G(x,y)^{-d \alpha})$ with $\alpha \in (1,2)$. 
    Let $\beta > \beta_{c}$. 
    Then almost surely the infinite cluster $\kinfty$ admits a renormalised subgraph.
\end{lemma}

\begin{proof}
    Since the events that increasingly large tiles are good are positively correlated, by the FKG inequality, by \thref{prop:renorm_subgraph_wts}, and by the definition of $\xi_i$, we have that
    \begin{equation}
        \P_{\beta}
        \left(
            \bigcap_{i = 1}^{\infty}
            \{
                \Tile{i}{0}
                \text{ is good}
            \}
        \right)
        \geq 
        \prod_{i = 1}^{\infty}
        \P_{\beta}
        \left(
            \Tile{i}{0}
            \text{ is good}
        \right)
        \geq 
        \prod_{i = 1}^{\infty}
        (1
        -
        \xi_i)
        > 0.
    \end{equation}
    It follows that with positive probability the origin $o$ is in the infinite cluster and all tiles containing at the origin are good.
    Since the event that there exists a renormalised subgraph of the infinite cluster is translation-invariant, it follows from the ergodicity of $\P_{\beta}$ that the infinite cluster $\kinfty$ admits a renormalised subgraph almost surely.
\end{proof}

Given that the infinite cluster $\kinfty$ admits a renormalised subgraph almost surely, we now use the branching structure of the renormalised subgraph to construct a finite-energy flow.
Recall that a \textbf{flow} from $o$ to $\infty$ is a function $\theta : \Vec{E} \to \R$ which is antisymmetric, namely $\theta(x,y) = - \theta(y,x)$, and satisfies the Kirchhoff current law for all $v \in V \setminus \{o\}$. 
The \textbf{energy} $\mathcal{E}(\theta)$ of $\theta$ is defined to be $\mathcal{E}(\theta) = \sum_{e \in E} \theta(e)^2$.
See for instance \cite[Chapter 2]{grimmett_probability_2018} for further details.

\begin{lemma}
    \thlabel{lem:finite_energy_flow}
    Let $G$ be a transitive graph of polynomial growth with $d \geq 1$, and suppose that $J : V \times V \to \R_+$ is a transitive kernel satisfying $J(x,y) = \Omega (d_G(x,y)^{-d \alpha})$ with $\alpha \in (1,2)$. 
    Let $\beta > \beta_{c}$. 
    Suppose that the infinite cluster admits a renormalised subgraph almost surely. 
    Then there exists a finite-energy flow to $\infty$ on $\kinfty$. 
\end{lemma}

\begin{proof} For this part of the proof we closely follow \cite{berger_transience_2002} with some more details spelled out.
    Suppose that the origin $o$ is in the infinite cluster $\kinfty$ and that the tiles containing the origin are good for all $i \in \N$.
    We construct a finite energy flow from $o$ to $\infty$ as follows.
    For $i \geq 1$ we let $B_i = C_i - 1$ and $N_i = \prod_{r = 1}^i B_r$. 
    For $i \geq 2$, the good part $\CC_{i-1,0}$ contains $\CC_{i-2,0}$ and $B_{i-1}$ more \level{i-2} good parts $\CC_{i-2,j}$ for $1 \leq j \leq B_{i-1}$ that we call the $i-1$ core.
    Similarly, the bigger good part $\CC_{i,0}$ contains $\CC_{i-1,0}$ and $B_i$ more \level{i-1} good parts, $\CC_{i-1,k}$ for $1 \leq k \leq B_i$. 
    By the goodness of $\CC_{i,0}$, for each pair $(j,k)$ with $1 \leq j \leq B_{i-1}$ and $1 \leq k \leq B_{i}$ we can choose vertices $u_{i-1,j,k} \in \CC_{i-2,j}$ in the core and $v_{i-1,j,k} \in \CC_{i-1,j}$ in another tile $T_{i-1,j}$ inside $T_{i,0}$ not centred at $0$ such that $u_{i-1,j,k} \sim v_{i-1,j,k}$.
    We now define the flow.
    For $i \geq 2$, we define the upward part of the flow at stage $i$ by
    \begin{equation}
        \phi_{\mathrm{up}}^{(i)}(x,y)
        =
        \begin{cases}
            \dfrac{1}{B_{i-1}B_i}
            &
            \text{if }
            x = u_{i-1,j,k}
            \text{ and }
            y = v_{i-1,j,k}
            \text{ for some } j,k,
            \\
            -\dfrac{1}{B_{i-1}B_i}
            &
            \text{if }
            x = v_{i-1,j,k}
            \text{ and }
            y = u_{i-1,j,k}
            \text{ for some } j,k,
            \\
            0
            &
            \text{otherwise.}
        \end{cases}
    \end{equation}
    We still need to define the redistributing part of the flow. 
    At stage $1$, there exists a redistributing flow $\phi_{\mathrm{dist}}^{(1,0)}$ supported on $\CC_{1,0}$ sending current $1/B_1$ current to each $u_{1,1,k}$ with $1 \leq k \leq B_1$.
    At stage $i \geq 2$, for each good part $\CC_{i,j}$ with $1 \leq j \leq B_i$ there exists a redistributing flow $\phi_{\mathrm{dist}}^{(i,j)}$ with sources $v_{i,j,k}$ and sinks $u_{i+1,j,k}$, with $1 \leq k \leq B_i$, such that each sink $u_{i+1,j,k}$ receives $1/B_i$ units of current.
    For $i \geq 1$ we let 
    \begin{equation}
        \theta^{(i)}
        =
        \sum_{k = 1}^i 
        \left(
            \phi_{\mathrm{up}}^{(k)}
            + 
            \sum_{j = 1}^{B_k} 
            \phi_{\mathrm{dist}}^{(k,j)}
        \right)
        \quad 
        \text{ and }
        \quad
        \theta 
        = 
        \sum_{i \in \N} 
        \theta^{(i)}.
    \end{equation}
    It follows readily that $\theta$ is a flow from $o$ to $\infty$, and we now verify that $\theta$ has finite energy.
    For $i \geq 1$ let $E_i$ denote the energy of $\theta^{(i)}$, and we bound the total energy iteratively. 
    The energy at \level{i} is at most the energy at \level{i-1} together with the energy of the upwards parts $\phi_{\mathrm{up}}^{(i)}$ together with the energy of the redistributing parts $\phi_{\mathrm{dist}}^{(i,j)}$ with $1 \leq j \leq B_i$.
    The upwards part $\phi_{\mathrm{up}}^{(i)}$ has $B_{i-1}B_i$ edges each carrying current
    $1/(B_{i-1}B_i)$, so that its energy is $1/(B_{i-1}B_i)$.
    
    Since energy scales quadratically, the energy of the redistributing flow $\phi_{\mathrm{dist}}^{(i,j)}$ for some $1 \leq j \leq B_i$ is at most $E_{i-1}/B_i^2$, and since there are $B_i$ redistributing flows the energy of the redistributing flows is at most $E_{i-1}/B_i$.
    It follows that $E_i \leq E_{i-1} \left(1 + 1/B_i \right) + 1/(B_{i-1}B_i)$, so that after iterating    
    \begin{equation}
        E_i
        \leq 
        E_1 
        \prod_{k = 2}^{i}
        \left(
            1 
            +
            \frac{1}{B_k}
        \right)
        + 
        \sum_{\ell = 2}^i
        \frac{1}{B_{\ell-1}B_{\ell}}
        \prod_{k = \ell+1}^{i}
        \left(
            1 
            +
            \frac{1}{B_k}
        \right).
    \end{equation}
    By the definition of $C_i$ in \eqref{eq:third_renormalisation}, both the infinite sum $\sum_{\ell = 2}^\infty 1 / (B_{\ell-1} B_{\ell})$ and the infinite product $\prod_{k = 2}^{\infty} \left(1 + 1/(B_k) \right)$ converge, so that $\sup_i E_i < \infty$.
    This concludes the proof.
\end{proof}

\begin{proof}[Proof of \thref{thm:random_walks}]
    By \thref{prop:renormalised_graph} the infinite cluster $\kinfty$ admits a renormalised subgraph almost surely. 
    By \thref{lem:finite_energy_flow} this implies that $\kinfty$ admits a finite-energy flow almost surely, and by the finite-energy flow criterion for recurrence \cite{lyons_simple_1983} (see also \cite[Theorem 2.11]{lyons_probability_2016}) the infinite cluster $\kinfty$ is transient almost surely.
\end{proof}
\section{Consequences}
\label{sec:consequences}
In this section we prove the many consequences which result from combining the local existence-and-uniqueness of the giant in \thref{prop:biskup,prop:second_largest} and the sphere calculus in \thref{lem:averaged_ball_connection}.
We highlight the following result on the probability that the local giant \textit{fails} to keep being the local giant, which captures how the bounds on the first and second largest local clusters work together.
The key idea is that $\giantone{r}$ is not contained in $\giantone{r+1}$ when either $\giantone{r}$ is unusually small or $\gianttwo{r + 1}$ is unusually large.
Recall that we write $\giantone{r}$ and $\gianttwo{r}$ for the first and second largest components in $B(r)$, respectively, and $\giantcomp{x}{r}$ for the largest cluster in $B(x,r)$.

\begin{proposition}
    \thlabel{leaving_giant}
    Let $G$ be a transitive graph of polynomial growth with $d \geq 1$ and suppose that $J : V \times V \to \R_+$ is a transitive kernel satisfying $J(x,y) = \Omega(d_G(x,y)^{-d \alpha})$ with $\alpha \in (0,2)$.
    Let $\beta > \beta_c$.
    There exists $c > 0$ such that for all $1\le r< R<\infty$
    \begin{equation}
        \label{eq:finite_failure}
        \P_{\beta}
        \left(  
            \giantcomp{x}{r}
            \not\subseteq
            \giantone{R}
        \right)   
        \leq 
        \exp
        \left(
            -
            c
            \# 
            B(r)^{\min(2-\alpha,1)}
        \right)
    \end{equation}
    for all $x \in V$ with $x \in B(R-r)$.
    In particular, there exists $c > 0$ such that for all $r \in \N$
    \begin{equation}
        \label{eq:infinite_failure}
        \P_{\beta}
        \left(  
            \giantone{r}
            \not\subseteq
            \kinfty
        \right)   
        \leq 
        \exp
        \left(
            -
            c
            \# 
            B(r)^{\min(2-\alpha,1)}
        \right).
    \end{equation}
\end{proposition}

\subsection{Concentration of the local giant}
\label{sec:volume_upper_bound}
In this section we prove \thref{thm:cluster_size_decay} on the stretched-exponential decay of the distribution of finite clusters.
A similar argument yields \thref{prop:o_in_kinfty}.
We begin by proving \thref{leaving_giant}.

\begin{proof}[Proof of \thref{leaving_giant}]
    Let $m = d_G(o,x)$, let $\gamma = (\gamma_0,\ldots,\gamma_m)$ be a geodesic path from $x$ to $o$, and let $(z_r,\ldots,z_R)$ be the sequence of vertices with $z_i = x$ for $r \leq i \leq R - m$ and $z_i = \gamma_{i - (R - m)}$ for $R - m <  i \leq R$.
    Then
    \begin{equation}
        B(x,r)
        =
        B(z_r,r) 
        \subseteq B(z_{r+1}, r+1)\subseteq
        \ldots
        \subseteq 
        B(z_R,R)
        =
        B(R)
    \end{equation}
    and hence
    \begin{equation}
        \label{eq:witness_failure}
        \P_{\beta}
        \left(  
            \giantcomp{x}{r}
            \not\subseteq
            \giantone{R}
        \right)   
        \leq 
        \sum_{i = r}^{R - 1}
        \P_{\beta}
        \left(
            \giantcomp{z_i}{i}
            \not\subseteq 
            \giantcomp{z_{i+1}}{i+1}
        \right),
    \end{equation}
    namely there is a scale $i$ at which the largest cluster in $B(z_{i},i)$ is not contained in the largest cluster in $B(z_{i+1},i+1)$.
    Let $\nu > 0$ be as in \thref{prop:biskup}, and let $\rho > 0$ be such that $\nu \#B(r) > \rho \# B(r+1)$ for all $r$ sufficiently large. Such a choice is possible by taking $\rho \leq \nu c_G/(2^dC_G)$.
    For $i \geq r$, the events $\# \giantcomp{z_i}{i} \geq \nu \# B(i)$ and $\# \cluster_{z_{i+1},i}^{\scriptscriptstyle (2)} \leq \rho \# B(i+1)$ imply that $\giantcomp{z_i}{i} \subseteq \giantcomp{z_{i+1}}{i+1}$.
    It follows from \thref{prop:biskup,prop:second_largest} that
    \begin{multline}
        \P_{\beta}
        \left(
            \giantcomp{z_i}{i}
            \not\subseteq 
            \giantcomp{z_{i+1}}{i+1}
        \right)
        \leq 
        \P_{\beta}
        \left(
            \giantcomp{z_i}{i}
            \leq 
            \nu
            \#
            B(i)
        \right)
        +
        \P_{\beta}
        \left(
            \cluster_{z_{i+1},i + 1}^{\scriptscriptstyle (2)}
            \geq 
            \rho \# B(i+1)
        \right)
        \\
        \leq
        \exp
        \left(
            -
            c_1 
            \# 
            B(i)^{2-\alpha}
        \right)
        +
        \frac{1}{\rho}
        \exp
        \left(
            -
            c_2
            \left(
                \rho
                \# 
                B(i+1)
            \right)^{\min(2-\alpha,1)}
        \right)
        \leq
        \exp
        \left(
            -
            c_3
            \# 
            B(i)^{\min(2-\alpha,1)}
        \right)
    \end{multline}
    for some $c_1,c_2,c_3 > 0$, where in in the last inequality we use the choice of $\rho$.
   We sum over $i$ to get
    \begin{align}
        \P_{\beta}
        \left(  
            \giantcomp{x}{r}
            \not\subseteq
            \giantone{R}
        \right)  
        \leq
        \sum_{i = r}^{\infty}
        \exp
        \left(
            -
            c_3
            \# 
            B(i)^{\min(2-\alpha,1)}
        \right)
        \leq 
        \exp
        \left(
            -
            c_4
            \# 
            B(r)^{\min(2-\alpha,1)}
        \right)
    \end{align}
    for some $c_4 > 0$ and for all $r,R \in \N$ with $R > r$ and $x \in V$ with $x \in B(R-r)$, concluding the proof of \eqref{eq:finite_failure}.
    For $r \in \N$ we have by a union bound, by \thref{prop:biskup}, and by \eqref{eq:finite_failure} that
    \begin{equation}
        \sum_{r= 1}^{\infty}
        \P_{\beta}
        \left(
             \{\giantone{r} \subseteq \giantone{r+1},
                \giantone{r} \geq \nu \# B(r)
            \}^c
        \right)
        < 
        \infty.
    \end{equation}
    By the Borel-Cantelli lemma there exists $R_0 \in \N$ such that $\cluster^\star = \cup_{r \geq R_0} \giantone{r}$ is an infinite connected component almost surely, and by the uniqueness of the infinite cluster we have that $\cluster^\star \subseteq \kinfty$ almost surely.
    In particular, the indicator $\charf \left( \giantone{r} \not\subseteq \giantone{R} \right)$ converges to $\charf \left( \giantone{r} \not\subseteq \kinfty \right)$ almost surely and by the dominated convergence theorem 
    \begin{equation}
        \lim_{R \to \infty}
        \P_{\beta}
        \left(  
            \giantone{r}
            \not\subseteq
            \giantone{R}
        \right)
        = 
        \P_{\beta}
        \left(  
            \giantone{r}
            \not\subseteq
            \kinfty
        \right).
    \end{equation}
    Together with \eqref{eq:finite_failure} this proves \eqref{eq:infinite_failure}.
\end{proof}

The following proposition shows that the giant is largely determined in a ball. 

\begin{proposition}
    \thlabel{prop:giant_determines_c_infty}
    Let $G$ be a transitive graph of polynomial growth with $d \geq 1$ and suppose that $J : V \times V \to \R_+$ is a transitive kernel satisfying $J(x,y) = \Omega(d_G(x,y)^{-d \alpha})$ with $\alpha \in (0,2)$.
    Let $\beta > \beta_c$.
    Then there exist $c_1,c_2 > 0$ such that for every $k,r \in \N$ and $u \in B(r - \exp(c_1 k^{\min(2-\alpha,1)}))$ we have
    \begin{equation}
        \P_{\beta}
        \left(
            \# 
            \cluster_r(u)
            > 
            k,
            u 
            \not\in 
            \giantone{r}
        \right)
        \leq
        \exp
        \left(
            -c_2 
            k^{\min(2-\alpha,1)}
        \right).
    \end{equation}
\end{proposition}

\begin{proof}
    For $k \in \N$, let $R_k = \exp(c_1 k^{2-\alpha} / (2d))$
    and let $r \in \N$ be such that $d_G(u,x) \geq R_k + r$ so that in particular $B(u,R_k) \subseteq B(x,r)$.
    By a union bound we have
    \begin{multline}
        \label{eq:localised_giant_union_bound}
        \P_{\beta}
        \left(
            \#
            \cluster_r(u)
            > 
            k,
            u 
            \not\in 
            \giantone{r}
        \right)
        \leq 
        \P_{\beta}
        \left(
            u 
            \not\in
            \giantone{r},
            u 
            \in
            \giantcomp{u}{R_k}
        \right)
        \\
        +
        \P_{\beta}
        \left(
            u 
            \not\in
            \giantcomp{u}{R_k},
            \# 
            \loccomp{u}{R_k}
            >
            k
        \right)
        +
        \P_{\beta}
        \left(
            \# 
            \cluster_r(u) 
            > 
            k,
            \# 
            \loccomp{u}{R_k}
            \leq
            k
        \right),
    \end{multline}
    and we bound each term in order.
    If $u \not\in \giantcomp{u}{R_k}$ and $\# \loccomp{u}{R_k} > k$ then $\# \loccomp{u}{R_k}^{\scriptscriptstyle (2)} > k$ and hence by \thref{prop:second_largest}, by translation invariance, and by the choice of $R_k$, we have
    \begin{equation}
        \label{localised_giant_second_largest}
        \P_{\beta}
        \left(
            u 
            \not\in
            \giantcomp{u}{R_k},
            \# 
            \loccomp{u}{R_k}
            >
            k
        \right)
        \leq
        \P_{\beta}
        \left(
            \# 
            \loccomp{u}{R_k}^{\scriptscriptstyle (2)}
            >
            k
        \right)
        \leq
        \exp
        \left(
            -
            c_2
            k^{\min(2-\alpha,1)}
        \right)
    \end{equation}
    for some $c_2 = c_2(G) > 0$.
    By \thref{leaving_giant} we have
    \begin{equation}
        \label{localised_giant_leaving_giant}
        \P_{\beta}
        \left(
            u 
            \in
            \giantcomp{u}{R_k},
            u 
            \not\in
            \giantone{r}
        \right)
        \leq
        \exp
        \left(
            - 
            c_2
            \# 
            B(r)^{\min(2-\alpha,1)}
        \right)
        \leq 
        \exp
        \left(
            - 
            c_3
            k^{\min(2-\alpha,1)}
        \right)
    \end{equation}
    for some $c_2 = c_2(G) > 0$ and $c_3 = c_3(G) > 0$, where in the second inequality we use that $r > R$ and the definition  of $R$.
    Finally, if $\# \cluster_r(u) > k$ and $\# \loccomp{u}{R} \leq k$ then $u$ has a path to $B(u,R_k)^c$, and by the pigeonhole principle at least one of the edges in $\loccomp{u}{R_k}$ has length at least $R_k/k$. 
    It follows that
    \begin{align}
        \label{localised_giant_arm_path}
        \P_{\beta}
        \left(
            \# 
            \cluster_r(u)
            > 
            k,
            \# 
            \loccomp{u}{R_k}
            \leq
            k
        \right)
        \leq 
        k
        \sum_{L = R_k/k}^{\infty}
        \E_{\beta}
        \left[
            \deg(o,L)
        \right]
        \leq 
        \exp
        \left(
            -
            c_4
            k^{\min(2-\alpha,1)}
        \right)
    \end{align}
    for some $c_4 >0$, where in in the last inequality we use \thref{lem:integrability}, the choice of $R_k$, and the fact that $\alpha < 2$.
    The result follows by combining the bounds in \eqref{eq:localised_giant_union_bound} through \eqref{localised_giant_arm_path}.
\end{proof}

A union bound similar to that used in the proof of \thref{prop:giant_determines_c_infty} also gives \thref{prop:o_in_kinfty}.

\begin{proof}[Proof of \thref{prop:o_in_kinfty}]   
    Let $r \in \N$ and let $k^{\min(2-\alpha,1)} = c_1 \log(\# B(r))$ with $c_1$ is as in \thref{prop:second_largest}.
    By a union bound we can write 
    \begin{equation}
        \label{eq:bounding_o_to_inf_1}
        \P_{\beta}
        \left(
            o 
            \not\in 
            \giantone{r},
            o 
            \in
            \kinfty
        \right)
        \leq 
        \P_{\beta}
        \left(
            \# 
            \cluster_r
            > 
            k,
            o 
            \not\in 
            \giantone{r}
        \right)
        + 
        \P_{\beta}
        \left(
            \# 
            \cluster_r
            \leq
            k,
            \cluster_r
            \leftrightarrow
            B(r)^c
        \right).
    \end{equation}
    If $\# \cluster_r > k$ and $o \not\in \giantone{r}$ then $\# \gianttwo{r} > k$ and by \thref{prop:second_largest} and by the choice of $k$ we have that
    \begin{equation}
        \label{eq:bounding_o_to_inf_2}
        \P_{\beta}
        \left(
            \# \cluster_r 
            > 
            k,
            o 
            \not\in 
            \giantone{r}
        \right)
        \leq 
        \P_{\beta}
        \left(
            \gianttwo{r} 
            > 
            k
        \right)
        \leq 
        \exp
        \left(
            -
            c_2
            k^{\min(2-\alpha,1)}
        \right)
    \end{equation}
    for some $c_2 = c_2(G) > 0$.
    If $\# \cluster_r > k$ and $\cluster_r \leftrightarrow B(r)^c$ then by the pigeonhole principle at least one of the edges in $\cluster$ has length at least $r/k$.
    Just as in \eqref{localised_giant_arm_path}, it follows that
    \begin{align}
        \label{eq:bounding_o_to_inf_3}
        \P_{\beta}
        \left(
            \# 
            \cluster_r
            \leq 
            k,
            \cluster_r
            \leftrightarrow
            B(r)^c
        \right)
        \leq
        k
        \sum_{L = r/k}^{\infty}
        \E_{\beta}
        \left[
            \deg(o,L)
        \right]
        \leq 
        \exp
        \left(
            -
            c_4
            k^{\min(2-\alpha,1)}
        \right)
    \end{align}
    for some $c_4 > 0$.
    By the bounds above and by the choice of $k$ we have that
    \begin{equation}
        \P_{\beta}
        \left(
            o 
            \not\in 
            \giantone{r},
            o 
            \in 
            \kinfty
        \right)
        \leq 
        \exp
        \left(
            -c_5
            k^{2-\alpha}
        \right)
        \leq 
        \exp
        \left(
            -c_6
            \log(\#B(r))
        \right)
    \end{equation}
    for some $c_5,c_6 > 0$.
    Finally, by union bound we can write
    \begin{equation}
        \label{eq:rewrite_theta}
        \theta(\beta)
        =
        \P_{\beta}
        \left(
            o 
            \in 
            \kinfty
        \right)
        \leq 
        \P_{\beta}
        \left(
            o 
            \in 
            \giantone{r}
        \right)
        +
        \P_{\beta}
        \left(
            o 
            \not\in 
            \giantone{r},
            o 
            \in 
            \kinfty
        \right)
    \end{equation}
    and the result follows by rearranging.
\end{proof}

\begin{proof}[Proof of \thref{thm:cluster_size_decay}]
    Let $k \in \N$.
    By \thref{leaving_giant} and by the Borel-Cantelli lemma we have that $\giantone{r} \subseteq \kinfty$ for all $r$ sufficiently large almost surely.
    By \thref{prop:o_in_kinfty} and by the Borel-Cantelli lemma we have that $o \in \kinfty$ implies that $o \in \giantone{r}$ for all $r$ sufficiently large.
    It follows that the indicator $\charf \left(\# \cluster_r > k, o \not\in \giantone{r}, \giantone{r} \subseteq \kinfty \right)$ converges to the indicator $\charf \left( k < \# \cluster < \infty \right)$ almost surely.
    Indeed: if $k < \# \cluster < \infty$ then eventually $\# \cluster_r > k$ and $o \not\in \giantone{r}$; if $\# \cluster \leq k$ then $\# \cluster_r \leq k$ for all $r$; and if $\# \cluster = \infty$ then eventually $o \in \giantone{r}$.
    By the dominated convergence theorem
    \begin{equation}
        \lim_{r \to \infty}
        \P_{\beta}
        \left(
            \# 
            \cluster_r > k, 
            o \not\in \giantone{r}, 
            \giantone{r} \subseteq \kinfty
        \right)
        =
        \P_{\beta}
        \left(
            k
            <
            \#
            \cluster
            <
            \infty
        \right),
    \end{equation}
    so that in particular 
    \begin{equation}
        \P_{\beta}
        \left(
            k
            <
            \#
            \cluster
            <
            \infty
        \right)
        \leq 
        \P_{\beta}
        \left(
            \# 
            \cluster_r > k, 
            o \not\in \giantone{r}
        \right)
    \end{equation}
    for all $r \in \N$.
    By \thref{prop:giant_determines_c_infty} there exists $c > 0$ such that for all $r \in \N$
    \begin{equation}
        \P_{\beta}
        \left(
            \# 
            \cluster_r > k, 
            o \not\in \giantone{r}
        \right)
        \leq
        \exp
        \left(
            -c
            k^{\min(2-\alpha,1)}
        \right).
    \end{equation}
    \vskip-1em
\end{proof}

\subsection{Anchored isoperimetric dimension of the infinite cluster}
In this section we prove \thref{cor:anchored_isop_dim} and \thref{thm:anch_dim_upper} on the anchored isoperimetric dimension of the infinite cluster $\kinfty$.

\begin{proof}[Proof of \thref{cor:anchored_isop_dim}]
    This follows from \thref{thm:cluster_size_decay} and \cite[Theorem 1.5]{alonso_supercritical_2026}.
\end{proof}

We prove \thref{thm:anch_dim_upper} by combining the sphere calculus, the local existence-and-uniqueness of the giant in \thref{prop:biskup,prop:second_largest}, and the \textit{lazy proof} of the one-arm event with polylogarithmic corrections from our companion paper \cite[Theorem 1.8]{alonso_supercritical_2026}.
Recall that $\cluster_r$ denotes the cluster of the origin $o$ in $B(r)$, and now we assume $J(x,y)=\Theta(d_G(x,y)^{-d\alpha})$ with $\alpha\in(1,2)$.

\begin{proof}[Proof of \thref{thm:anch_dim_upper}]
    We write $\gamma = \max \left(d(2-\alpha),d - 1 \right)$. By \thref{lem:sphere_calculus} there exists $c_1 > 0$ such that for each $\ell \in \N$ we can choose some $n \in [2^{\ell-1},2^{\ell}]$ with 
    \begin{equation}
        \label{eq:good_kernel_bound}
        J
        \left(
            B(n), 
            B(n)^c
        \right)
        \leq 
        c_1
        n^{\gamma}
        \log(n)^{
            \charf
            \left(
                \alpha=1+1/d
            \right)
        },
    \end{equation}
    and we let $I$ denote this set. 
    We let $\nu$ be as in \thref{prop:biskup}. 
    Set $\xi(n) = n^{\gamma} \log(n)^{\charf \left( \alpha=1+1/d \right)}$, and let $A_n$ denote the event that $\# \giantone{n} \geq \nu \# B(n)$ and $\# \{(u,v) \in \giantone{n} \times B(n)^c : u \sim v \} \leq 2 \beta c_1 \xi(n)$, and let $B_n$ denote the event that $o \in \giantone{n}$.
    We show that on the event that $o \in \kinfty$ the event $A_n \cap B_n$ happens for all $n \in I$ sufficiently large almost surely.
    If this is the case, then for $\epsilon > 0$,
    \begin{equation}
        \liminf_{n \to \infty}
        \frac{
            \# \partial 
            \giantone{n}
        }{
            (\# \giantone{n})^{\gamma/d + \epsilon}
        }
        \leq
        \liminf_{n \to \infty}
        \frac{
            c_1 \log(n)^{\charf \left( \alpha=1+1/d \right)}
        }{
            n^{\epsilon}  
        }  
        =
        0
    \end{equation}
    and in particular since $\giantone{n} \subseteq \kinfty$ for all $n \in I$ sufficiently large the infinite cluster has anchored isoperimetric dimension at most $\max( 1/(\alpha-1), d)$ almost surely.
    By a union bound
    \begin{equation}
        \label{eq:nice_union_bound}
        \P_{\beta}
        \left(
            \left(
                A_n 
                \cap 
                B_n
            \right)^c
            \mid
            o 
            \in 
            \kinfty
        \right)
        \leq 
        \frac{
            \P_{\beta}
            \left(
                A_n^c
            \right)
            +
            \P_{\beta}
            \left(
                B_n^c
                \cap 
                \{o 
                \in 
                \kinfty\}
            \right)
        }{
            \P_{\beta}
            \left(
                o 
                \in
                \kinfty
            \right)
        }
    \end{equation}
    and we bound each term in turn.
    In the first instance we have by a union bound that
    \begin{equation}
        \label{eq:an_comp}
        \P_{\beta}
        \left(
            A_n^c
        \right)
        \leq
        \P_{\beta}
        \left(
            \# 
            \{
                (u,v)
                \in 
                \giantone{n} 
                \times
                B(n)^c
                :
                u \sim v
            \}
            >
            2
            \beta
            c_1
            \xi(n)
        \right)
        +
        \P_{\beta}
        \left(
            \# 
            \giantone{n}
            <
            \nu 
            \# 
            B(n)
        \right).
    \end{equation} 
    Since $n$ was chosen so as to satisfy \eqref{eq:good_kernel_bound} we compute
    \begin{equation}
        \E_{\beta}
        \left[
            \# 
            \{
                (u,v)
                \in 
                B(n)
                \times
                B(n)^c
                :
                u \sim v
            \}
        \right]
        \leq 
        \beta
        J(B(n),B(n)^c)
        \leq 
        \beta
        c_1
        \xi(n).
    \end{equation}
    As $\giantone{n} \subseteq B(n)$ and the edges are independent Bernoulli variables, Chernoff's bound gives that
    \begin{equation}
        \P_{\beta}
        \left(
            \# 
            \{
                (u,v)
                \in 
                \cluster_n 
                \times
                B(n)^c
                :
                u \sim v
            \}
            >
            2
            \beta
            c_1
            \xi(n)
        \right)
        \leq
        \exp
        \Big(  
            -
            \frac{
                \beta 
                c_1
                \xi(n)
            }{
                3
            }
        \Big).
    \end{equation}
    By \thref{prop:biskup} there exists $c_2 > 0$ such that    for all $n \in \N$
    \begin{equation}
        \P_{\beta}
        \left(
            \#\giantone{n} < \nu \# B(n)
        \right)
        \leq
        \exp
        (
            -
            c_2
            \beta 
            (\# B(n))^{2 - \alpha}
        ).
    \end{equation}
    The event that $o \in \kinfty$ implies that $o \leftrightarrow B(n)^c$ for all $n \in \N$ and hence
    \begin{equation}
        \P_{\beta}
        \left(
            B_n^c
            \cap 
            \{o 
            \in 
            \kinfty\}
        \right)
        =
        \P_{\beta}
        \left(
            o 
            \not\in 
            \giantone{n},
            o 
            \in 
            \kinfty
        \right)
        \leq
        \P_{\beta}
        \left(
             o
            \not\in 
            \giantone{n},
            o 
            \leftrightarrow
            B(n)^c
        \right).
    \end{equation}      
    For $k \in \N$, by a union bound and using that $o \not\in \giantone{n}$ and $\# K_n \geq k$ imply $\# \gianttwo{n} \geq k$ we have
    \begin{align}
        \P_{\beta}
        \left(
             o
            \not\in 
            \giantone{n},
            o 
            \leftrightarrow
            B(n)^c
        \right)
        \leq
        \P_{\beta}
        \left(
            \#
            \gianttwo{n}
            \geq 
            k
        \right)
        +
        \P_{\beta}
        \left(
            o 
            \leftrightarrow
            B(n)^c,
            \# 
            \cluster_n 
            < 
            k
        \right).
    \end{align}
    Setting $k(n) = \left((2d/c_3)\log(n)\right)^{1/(2-\alpha)}$ with $c_3$ as in \thref{prop:second_largest}, this same theorem gives
    \begin{equation}
        \P_{\beta}
        \left(
            \gianttwo{n}
            \geq 
            k(n)
        \right)
        \leq
        \exp
        \left(
            -
            c_3
            k(n)^{2-\alpha}
            /
            2
        \right).
    \end{equation}
    Further, it follows from \cite[Theorem 1.8]{alonso_supercritical_2026} that there exists $c_4 > 0$ such that
    \begin{equation}
        \P_{\beta}
        \left(
            o 
            \leftrightarrow
            B(n)^c,
            \# 
            \cluster_n 
            < 
            k
        \right)
        \leq
        c_4
        k(n)
        (n/k(n))^{d(1-\alpha)}.
    \end{equation}
    Since $\sum_{n \in I} \P_{\beta} \left((A_n \cap B_n)^c \mid o \in \kinfty \right) < \infty$ the result follows from the Borel-Cantelli lemma.
\end{proof}

\subsection{Law of large numbers}
\label{sec:lln}
\thref{prop:second_largest} allows us to verify the conditions of \cite[Theorem 2.2]{van_der_hofstad_giant_2023} to obtain the law of large numbers for \LRP{}. Here $u\leftrightarrow v$ means that $u, v$ are in the same component.
\begin{proposition}
    \thlabel{prop:condition_check}
    Let $G$ be a transitive graph of polynomial growth with $d \geq 1$, and suppose that $J : V \times V \to \R_+$ is a transitive kernel satisfying $J(x,y) = \Omega(d_G(x,y)^{-d \alpha})$ with $\alpha \in (0,2)$. 
    Let $\beta > \beta_c$.
    Let $(\GG_r)_{r \in \N}$ be the sequence of random graphs sampled from $\P_{\beta}$ restricted to $B(r)$.
    Then
    \begin{equation}
        \label{eq:lln_conditions}
        \lim_{k \to \infty}
        \limsup_{r \to \infty}
        \frac{1}{\# B(r)^2}
        \E_{\beta}
        \left[
            \sum_{u,v \in V(\GG_r)}
            \charf
            \left(
                \# 
                \cluster_r(u) 
                ,
                \# 
                \cluster_r(v) 
                > 
                k,
                u 
                \not\leftrightarrow
                v
            \right)
        \right]
        =
        0.
    \end{equation}
\end{proposition}

\begin{proof}
    For $u,v \in V(\GG_r)$, the event that $\# \cluster_r(u) > k$ and $\# \cluster_r(v) > k$ with $u \not\sim v$ implies that at least one of $\cluster_r(u)$ or $\cluster_r(v)$ is \textit{not} the largest component in $B(r)$. 
    Let $R_k$ be as in \thref{prop:giant_determines_c_infty}. 
    Summing over $v$, decomposing the sum over $u \in V(\GG_r)$ according to the distance between $u$ and $\partial B(r)$, and applying \thref{prop:giant_determines_c_infty}, we arrive at
    \begin{align}   
        \label{eq:lln_first_bound}
        \begin{split}
            \frac{1}{\# B(r)^2}
            &
            \E_{\beta}
            \Biggl[
                \sum_{u,v \in V(\GG_r)}
                \charf
                \left(
                    \# 
                    \cluster_r(u) 
                    ,
                    \# 
                    \cluster_r(v) 
                    > 
                    k,
                    u 
                    \not\leftrightarrow
                    v
                \right)
            \Biggr]
            \\
            & \leq 
            \frac{2}{\# B(r)}
            \E_{\beta}
            \left[
                \sum_{\substack{
                    u \in V(\GG_r)
                    \\
                    d_G(u,\partial B(r)) \leq R_k
                }}
                1
                +
                \sum_{\substack{
                    u \in V(\GG_r)
                    \\
                    d_G(u,\partial B(r)) > R_k
                }}
                \charf
                \left(
                    \cluster_r(u)>k, \cluster_r(u) \neq
                    K_r^{\scriptscriptstyle{(1)}}
                \right) 
            \right]
        \end{split}
        \\
        \label{eq:lln_second_bound}
        & \leq
        \frac{
            2 
            \#(B(r) \setminus B(r-R_k))
        }{
            \# B(r)
        }
        + 
        2
        \exp
        \left(
            -c_2 
            k^{\min(2-\alpha,1)}
        \right)
    \end{align}
    for some $c_2 > 0$, where in \eqref{eq:lln_second_bound} we use that $\# B(r-R_k - 1) \leq \#B(r)$.
    By the amenability of $G$
    \begin{equation}
        \lim_{r \to \infty}
        \frac{
            2 
            \#(B(r) \setminus B(r-R_k))
        }{
            \# B(r)
        }
        =
        0
    \end{equation}
    and the result follows from \eqref{eq:lln_second_bound} after taking limits.
\end{proof}

\begin{proof}[Proof of \thref{thm:lln}]
    The result follows from \thref{prop:condition_check} and \cite[Theorem 2.2]{van_der_hofstad_giant_2023}.
\end{proof}

\subsection{Lower bound on the cluster-size decay }
\label{sec:volume_lower_bound}
In this section we prove \thref{thm:cluster_size_lower_bound} on the lower bound for the distribution of finite clusters.

\begin{proof}[Proof of \thref{thm:cluster_size_lower_bound}]
    We let $\nu > 0$ be as in \thref{prop:biskup}, we write $\gamma = \max \left(d(2-\alpha),d - 1 \right) $, and we let $q \geq 4$ be minimal such that $k \leq \nu \# B(q/2)$.
    By \thref{lem:averaged_ball_connection}, we can choose $r \in [q,2q]$ with 
    \begin{equation}
        \label{eq:averaged_kernel_bound}
        J
        \left(
            B(r), 
            B(r)^c
        \right)
        \leq 
        c_1
        r^{\gamma}
        \log(r)^{
            \charf
            \left(
                \alpha=1+1/d
            \right)
        }
    \end{equation}
    for some $c_1 > 0$.
    Since $r \geq q$, we have $k \leq \nu \# B(q/2) \leq \nu \# B(r/2)$.
    We also have that there exists $c_2 > 0$ such that $k / \# B(r/2) \geq c_2$. 
    Indeed, by the minimality of $q$ we have $k > \nu \# B((q-1)/2) \geq \# B(q/2) / C_G$ where $C_G$ is the doubling constant of $G$, since $r \leq 2q$ we have $\# B(r/2) \leq \# B(q) \leq C_G \# B(q/2)$, and the bound follows.
    By a union bound and by the independence of edges we have
    \begin{align}
    \label{eq:lower_bound_start}
        \P_{\beta}
        \left(      
            k 
            \leq 
            \# 
            \cluster
            < 
            \infty
        \right) 
        \geq 
        \P_{\beta}
        \left(
            \# 
            \cluster_{r}
            \geq 
            k,
            B(r)
            \not\sim
            B(r)^c
        \right)
        \geq
        \P_{\beta}
        \left(
            \# 
            \cluster_{r}
            \geq 
            k
        \right)
        \P_{\beta}
        \left(
            B(r)
            \not\sim
            B(r)^c
        \right).
    \end{align} 
    By the kernel bound in \eqref{eq:averaged_kernel_bound} and since $k \leq \nu \# B(r/2)$ we have
    \begin{align}
    \label{eq:good_connection_bound_averaged}
        \P_{\beta}
        \left(
            B(r)
            \not\sim
            B(r)^c
        \right)
        \geq 
        \exp
        \left(
            - 
            c_2
            \beta
            k^{\gamma/d} 
            \log(k)^{
                \charf
                \left(
                    \alpha=1+1/d
                \right)
            }
        \right)
    \end{align}
    for some $c_2 = c_2(G,J) > 0$.
    Let $Z_k = \sum_{x \in B(r/2)} \charf \left(\# \cluster(x,r/2) \geq k \right)$, where $\cluster(x,r/2)$ denotes the cluster of $x$ in $B(r/2)$, and note that $\giantone{r/2} \geq k$ implies that $Z_k \geq k$.
    We can write
    \begin{align}
        \label{eq:markov_comp_o}
        \P_{\beta}
        \left(
            \giantone{r/2} \geq k
        \right)
        & \leq 
        \P_{\beta}
        \left(
            Z_k 
            \geq k
        \right)
        \leq 
        \frac{1}{k}
        \sum_{x \in B(r/2)}
        \P_{\beta}
        \left(
            \# 
            \cluster(x,B(o,r/2))
            \geq 
            k
        \right)
        \\
        \label{eq:trans_inv_comp_o_1}
        & \leq 
        \frac{1}{k}
        \sum_{x \in B(r/2)}
        \P_{\beta}
        \left(
            \# 
            \cluster(x,B(x,r))
            \geq 
            k
        \right)
        =
        \frac{\# B(r/2)}{k}
        \P_{\beta}
        \left(
            \# 
            \cluster_{r}
            \geq 
            k
        \right)
    \end{align}
    where the second inequality in \eqref{eq:markov_comp_o} follows from Markov's inequality, the inequality in \eqref{eq:trans_inv_comp_o_1} follows from the observation that $K(x,B(o,r/2)) \subseteq K(x,B(x,r))$, and the equality in \eqref{eq:trans_inv_comp_o_1} follows from translation invariance.
    Rearranging and since $k / \# B(r/2) \geq c_2$ we have
    \begin{equation}
        \label{eq:first_kernel_size_bound}
        \P_{\beta}
        \left(
            \# 
            \cluster_{r}
            \geq 
            k
        \right)
        \geq
        c_2
        \P_{\beta}
        \left( 
            \# K_{r/2}^{(1)}
            \geq 
            k
        \right).
    \end{equation}    
    Since $k \leq \nu \# B(r/2)$ and by \thref{prop:biskup} we have that there exists $c_3 > 0$ such that
    \begin{equation}
        \label{eq:yet-another-unlabelled-eq-2}
        \P_{\beta}
        \left(
            \giantone{r/2} \geq k
        \right)
        \geq 
        \P_{\beta}
        \left(
            \giantone{r/2} 
            \geq
            \nu 
            \# 
            B(r/2)
        \right)
        \geq 
        1
        -
        \exp
        \left(
            -
            c_3
            \#
            B(r/2)^{2 - \alpha}
        \right).
    \end{equation}
    Together with \eqref{eq:first_kernel_size_bound} and \eqref{eq:yet-another-unlabelled-eq-2} we have shown that $\P_{\beta} \left( \cluster_r \geq k \right) \geq c_4$ for some $c_4 = c_4(G,J) > 0$, and together with \eqref{eq:lower_bound_start} and \eqref{eq:good_connection_bound_averaged} we have shown that
    \begin{equation}
        \exp
        (
            -  
            c_5
            \beta k^{\gamma/d}(\log k)^{
                \charf
                \left(
                    \alpha=1+1/d
                \right)}
        )
        \leq
        \P_{\beta}
        \left(      
            k 
            \leq 
            \# 
            \cluster
            < 
            \infty
        \right) 
    \end{equation}
    for some $c_5 = c_5(G,J) > 0$. 
    This concludes the proof.
\end{proof}
\section{\texorpdfstring{Recurrence with $d = 2$ and $\alpha \geq 2$}{Recurrence with d = 2 and α ≥ 2}}
\label{sec:recurrence}
In this section we prove \thref{thm:two_dim_recurrence} on the recurrence of the infinite cluster $\kinfty$ when $d = 2$ and $\alpha \geq 2$.
We do this by first proving a general result on the recurrence of random electrical networks on transitive graphs of polynomial growth satisfying a certain condition on the distribution of edge conductances, and then constructing such a random electrical network from the infinite cluster $\kinfty$.
We follow ideas developed for $\Z^2$  by Berger \cite[Theorem 3.9]{berger_transience_2002}, and we use a mass-transport argument and the sphere calculus from Section \ref{section:sphere_calculus} to extend the argument to transitive graphs of polynomial growth.
It is a standard fact that a transitive graph of polynomial growth $G$ satisfies the \textbf{mass-transport principle}, which states that for every $F : V^2 \to [0,\infty]$ that is diagonally invariant in the sense that $F(\gamma u, \gamma v) = F(u,v)$ for every $u,v \in V$ and $\gamma \in \aut(G)$, we have that $\sum_{v \in V} F(u,v) = \sum_{v \in V} F(v,u)$ for any $u \in V$, see \cite[Chapter 8]{lyons_probability_2016} for further background.

We start by proving the general recurrence result.
We say that a non-negative random variable $X$ is \textbf{dominated by a Cauchy tail} if there exists $c > 0$ such that $\P(X > n) < c / n$ for all $n \in \N$.

\begin{proposition}
    \thlabel{prop:rand_elec_net_rec}
    Let $G$ be a transitive graph of polynomial growth with $d = 2$, and let $\mathcal E$ be a random electrical network on $G$ such that the conductance of each edge is identically distributed (possibly dependent) random variable dominated by a Cauchy tail.
    Then $\mathcal E$ is recurrent almost surely.
\end{proposition}

We use the following result on identically distributed but possibly dependent nonnegative random variables dominated by a Cauchy tail.
This is \cite[Lemma 4.1]{berger_transience_2002} verbatim.

\begin{lemma}
    \thlabel{lem:cauchy_tail}
    Let $(X_i)_{i \in \N}$ be identically distributed (possibly dependent) nonnegative random variables that are dominated by a Cauchy tail. Then for all $\epsilon >0$ there exists $c > 0$ and $n_{\varepsilon} > 0$ such that for all $n \geq n_\varepsilon$
    \begin{equation}
        \P
        \Big(
            \sum_{i = 0}^n
            X_i 
            > 
            c n
            \log 
            n
        \Big)
        > 1-\epsilon.
    \end{equation}
\end{lemma}

Recall that $S_E(n)$ denotes the \emph{edge sphere} of radius $n$.

\begin{proof}[Proof of \thref{prop:rand_elec_net_rec}]
    Writing $s_n$ for $\# S_E(n)$, the lower bounds for the volume of spheres gives that $s_n \geq c_G n$ for some $c_G > 0$.
    For an arbitrary labelling of the edges in $S_E(n)$, let $e_n(i)$ denote the $i$-th edge in $S_E (n)$.
    We write $C(e_n(i))$ for the conductance of the edge $e_n(i)$  and $C(S_E(n))$ for the total conductance of the set $S_E(n)$. 
    Since $S_E(n)$ is an edge cutset from $o$ to $\infty$, $C(S_E(n))=\sum_{i = 1}^{s_n} C(e_n(i))$.
    Writing $A_n$ for the event that $C(S_E(n)) \leq c s_n \log s_n$ and letting $a_n = c s_n \log s_n$,  the effective resistance between $o$ and $\infty$ can be bounded from below as   
    \begin{equation}
        \label{eq:conductance_bound}
       R_{\mathrm{eff}}(o, \infty)\ge  \sum_{n \in \N}
        \frac{1}{C(S_E(n))}
        \geq
        \sum_{n \in \N}
        \frac{\charf
        \left(
            A_n
        \right)}{a_n}
        .
    \end{equation}
    For $\epsilon > 0$, we show that $\P (\sum_{n \in \N} \charf(A_n) / a_n = \infty) \geq 1 - \epsilon$. 
    Suppose to the contrary that there exists $M \in \N$ such that $\P(\widetilde A_M) > \epsilon$ where $\widetilde A_M$ denotes the event that $\sum_{n \in \N} \charf(A_n) / a_n < M$.
    Since the conductances $C(e_n(i))$ are identically distributed and dominated by a Cauchy tail, by \thref{lem:cauchy_tail} there exist $c>0$ and $n_{\epsilon/2}$ such that $\P \left(A_n \right) > 1- \epsilon/2$ for all $n \geq n_{\epsilon/2}$, and in particular $\P(A_n \mid \widetilde A_M)\ge \P(A_n) -\P(\widetilde A_M^c) \geq  \epsilon/2 $ for all $n \geq n_{\epsilon/2}$. 
    By definition, $a_n=c s_n \log s_n$ and since $d = 2$, the sphere calculus \thref{lem:log_divergence} implies that $\sum_{n \geq N} 1/a_n = \infty$.
    Applying a  conditional Markov inequality yields that
    \begin{equation}
        \E
        \left[
            \sum_{n \in \N}
            \frac{
                \charf(A_n)
            }{
                a_n
            }
            \mid 
            \widetilde A_M
        \right]
        =
        \sum_{n \in \N}
        \frac{
            \P
            \big(
                A_n
                \mid
                \widetilde A_M
        \big)
        }{
            a_n
        }
        \geq
        \sum_{n \geq n_{\varepsilon/2}}
        \frac{
            \P
            \big(
                A_n
                \mid
                \widetilde A_M
        \big)
        }{
            a_n
        }
        \geq 
        \varepsilon/2 
        \sum_{n \geq n_{\varepsilon/2}}
        \frac{1}{a_n}
        =
        \infty,
    \end{equation}      
    contradicting the definition of $\widetilde A_M$.
    Combining this with \eqref{eq:conductance_bound} gives
    \begin{equation}
        \P
        \left(
            R_{\mathrm{eff}}(o,\infty)=\infty
        \right)
        \geq
        \P_{\beta}
        \Bigg(
            \sum_{n \geq n_{\varepsilon/2}}
            \frac{\charf
            \left(
                A_n
            \right)}{a_n}
            =
            \infty
        \Bigg)
        \geq 
        1 
        -
        \epsilon.
    \end{equation}
    As $\varepsilon>0$ was arbitrary, this shows that $R_{\mathrm{eff}}(o,\infty)$ is infinite almost surely and by the Nash-Williams criterion for recurrence \cite[Chapter 2.5]{lyons_probability_2016} $\mathcal E$ is recurrent almost surely.
\end{proof}

\begin{remark}
    In the proof of \thref{prop:rand_elec_net_rec} we crucially use that $G$ has dimension $2$ to obtain that $\sum_n 1 / (\#S_E(n) \log \#S_E(n))$ diverges by the sphere calculus in \thref{lem:log_divergence}.
\end{remark}

We construct a random electrical network satisfying the assumptions of \thref{prop:rand_elec_net_rec} from the infinite cluster $\kinfty$.
We define the random electrical network $\mathcal E_1$ on $\kinfty$ by setting the conductance of each edge in $\kinfty$ to $1$.
We define another random electrical network $\mathcal E_2$ on $G$ associated to $\kinfty$ as follows.
We begin by setting the conductance of each edge to $0$.
For each open \emph{nearest-neighbour} edge in $\kinfty$, we assign conductance $1$ to that edge.
For each open edge $(x,y)$ in $\kinfty$ with $d_G(x,y) \geq 2$, we choose one of the finitely many geodesics from $x$ to $y$ in $G$ \emph{uniformly at random}, denoted by $\gamma_{x,y}$, and we assign conductance $d_G(x,y)$ to each edge of $G$ in this geodesic. We say that we project the edge $(x,y)\in K_\infty$ on each edge $(u,v)\in\gamma_{x,y}$.
The final conductance $C_{u,v}$ of an edge $(u,v)$ in $\mathcal E_2$ is then the sum of the conductances assigned to the given edge.
A mass-transport argument gives the following lemma on the distribution of edge conductances in the random electrical network $\mathcal E_2$.

\begin{lemma}
    \thlabel{lem:projected_network}
    Let $G$ be a transitive graph of polynomial growth with $d = 2$, let $J : V \times V \to \R_+$ be transitive and integrable, and suppose that $J(x,y) = O (d_G(x,y)^{-2 \alpha})$ with $\alpha > 1$.
    Then: 
    the distribution of edge conductances in $\mathcal E_2$ is translation invariant;
    the effective conductance of $\mathcal E_1$ is greater than or equal to that of $\mathcal E_2$;
    if $\alpha > 3/2 = 1 + 1/d$ then each edge in the random electrical network $\mathcal E_2$ has finite conductance almost surely;
    if $\alpha > 2$ then the conductance of each edge in $\mathcal E_2$ is in $L^1$;
    and if $\alpha = 2$ then the edge conductances in $\mathcal E_2$ are dominated by a Cauchy tail.
\end{lemma}

\begin{proof}
    That the distribution of edge conductances in $\mathcal E_2$ is translation invariant is immediate from the construction.
    To see that the effective conductance of $\mathcal E_1$ is greater than or equal to the effective conductance of $\mathcal E_2$, suppose that $(x,y)$ is an edge in $\mathcal E_1$ with $d_G(x,y) = r$. 
    This edge has conductance equal to $1$ in $\mathcal E_1$, and by the series law this is equal to the conductance across a sequence of $r$ edges each with conductance equal to $r$.
    It follows that the assignment procedure of $\mathcal E_2$ can only increase the effective conductance of the resulting network.
    We now show that each edge in the random electrical network $\mathcal E_2$ has finite conductance almost surely if $\alpha>3/2$.
    Let $u,v \in V$ be nearest-neighbours, so that $(u,v)$ is an edge in $G$.
    For $x,y \in V$, let $f_k(x,u,v,z)$ be the proportion of geodesic paths from $x$ to $z$ where the $k$-th vertex is $u$ and the $k+1$-st vertex is $v$, and let $f_k(x,u,z)$ be the proportion of geodesic paths from $x$ to $z$ where the $k$-th vertex is $u$.
    It is immediate that $f_k(x,u,v,z) \leq f_k(x,u,z)$.
    For $r \in \N$ let $F_{r,k}(x,u) = \sum_{z \in S(x,r)} f_k(x,u,z)$, where $S(x,r)$ is the (vertex-) sphere of radius $r$ around $x\in V$.
    Since $\sum_{u \in V} f_k(x,u,z) = 1$, it follows that $\sum_{u\in V}F_{r,k}(x,u) = \# S(x,r) = \# S(r)$ for every $x \in V$ and $r \in \N$.
    Since $F_{r,k}$ is invariant under the diagonal action of $\Aut(G)$, the mass-transport principle gives for every $x,u \in V$ and $r \in \N$ that
    \begin{equation}
        \label{eq:mtp_applied}
        \sum_{x \in V}
        F_{r,k}(x,u)
        =
        \sum_{x \in V}
        F_{r,k}(u,x)
        =
        \sum_{u \in V}
        F_{r,k}(x,u)
        =
        \# 
        S(r).
    \end{equation}
    As $f_k(x,u,v,z)$ is the conditional probability that the edge $(x,z)\in K_\infty$ is projected to the edge $(u,v)\in G$, we can bound the expected number of edges projected on $(u,v)$ as
    \begin{align}\label{eq:expected-projected}
        \E_{\beta}
        \left[
            \# 
            \text{edges projected on }
            (u,v)
        \right]
        & \le
        \sum_{r \in \N}
        \sum_{x \in V}
        \sum_{z \in S(x,r)}
        \sum_{k=0}^{r-1}
        f_k(x,u,v,z)
        \P_{\beta}
        \left(
            x 
            \sim
            z
        \right)
    \end{align}
    Note that $\sum_{k=0}^{r-1} f_k(x,u,v,z) \leq \sum_{k=0}^{r} f_k(x,u,z)$, and for $r \in \N$ with $r \geq R_J$ we have that
    \begin{align}\label{eq:r-measure}
        \sum_{x \in V}
        \sum_{z \in S(x,r)}
        \sum_{k=0}^{r}
        f_k(x,u,z)
        \P_{\beta}
        \left(
            x 
            \sim
            z
        \right)
        \leq 
        c
        r^{-2 \alpha}
        \sum_{k=0}^{r}
        \sum_{x \in V}
        F_{r,k}(x,u)
        r^{1 - 2 \alpha}
        \#
        S(r)
    \end{align}
    where the inequality follows from the kernel bounds, re-ordering the sums, and the definition of $F_{r,k}(x,u)$, and the equality follows from applying \eqref{eq:mtp_applied}.
    Using this bound together with the sphere calculus in \thref{lem:sphere_calculus} we find with $d=2$ that
    \begin{equation}
        \label{eq:finite_projected_edges}
        \E_{\beta}
        \left[
            \# 
            \text{edges projected on }
            (u,v)
        \right]
        \leq
        c
        \sum_{r \in \N}
        r^{1 - 2 \alpha}
        \#
        S(r) \le C \int_1^{\infty} t^{1-2\alpha + d-1} \mathrm dt,
    \end{equation}
    which is finite when $\alpha > 3/2$.
    For these values of $\alpha$, the number of edges projected on $(u,v)$ is finite almost surely, and hence the conductance of the edge $(u,v)$ is finite almost surely.
    Since an edge of length $r$ assigns conductance $r$ to each edge of its associated geodesic, similar calculations yield
    \begin{equation}\label{eq:length-jump}
        \E_{\beta}
        \left[
            C_{u,v}
        \right]
        =
        \sum_{r \in \N}
        \sum_{x \in V}
        \sum_{z \in S(x,r)}
        \sum_{k=0}^{r-2}
        r
        f_k(x,u,v,z)
        \P_{\beta}
        \left(
            x 
            \sim
            z
        \right)
        \leq 
        c
        \sum_{r \in \N}
        r^{2-2\alpha}
        \# 
        S(r),
    \end{equation}
    which again by the sphere calculus is finite when $\alpha > 2$.
    We now prove that the distribution of edge conductances has a Cauchy tail when $\alpha = 2$.
    The conductance of an edge $(u,v)$ in $G$ is given in the first row, and we directly define an upper bound 
    \begin{equation}
    \begin{aligned}
        C_{u,v}
        &=
        \sum_{x,z \in V}
        d_G(x,z)
        \charf
        \left(
            x 
            \sim 
            z
        \right) \charf
        \left(
            (x,z) \in \kinfty
        \right)
        \charf
        \left(
            (u,v)
            \in
            \gamma_{x,z}
        \right)
    \\
     \widetilde{C}_{u,v}
        &=
        \sum_{x,z \in V}
        d_G(x,z)
        \charf
        \left(
            x 
            \sim 
            z
        \right)
        \charf
        \left(
            (u,v)
            \in
            \gamma_{x,z}
        \right),
    \end{aligned}
    \end{equation}
    and then $C_{u,v} \leq \widetilde{C}_{u,v}$ almost surely.
    For $x,z \in V$, we let $I_{x,z} = \charf(x \sim z) \charf((u,v) \in \gamma_{x,z})$.
    Since the events $\{x \sim z\}$ and $\{(u,v) \in \gamma_{x,z}\}$ are independent  for any pair of vertices $x,y \in V$, the random variables $I_{x,z}$ are independent Bernoulli random variables with parameter $p_{x,z} = \P_{\beta}(x \sim z) \P((u,v) \in \gamma_{x,z})$.
    Let $\lambda_{x,z} = - \log(1 - p_{x,z})$, and let $Z_{x,z} \sim \mathrm{Poi}(\lambda_{x,z})$ be independent Poisson random variables, and define $\widehat{C}_{u,v} = \sum_{x,z \in V} d_G(x,z) Z_{x,z}$.
    It may be verified that $\widetilde{C}_{u,v}$ is stochastically dominated by $\widehat{C}_{u,v}$ and hence it is sufficient to prove that for some $c > 0$
    \begin{equation}
        \label{eq:domination_chain}
        \P_{\beta}
        \left(
            C_{u,v}
            \geq 
            t
        \right)
        \leq 
        \P_{\beta}
        \left(
            \widetilde{C}_{u,v}
            \geq 
            t
        \right)
        \leq 
        \P_{\beta}
        \left(
            \widehat{C}_{u,v}
            \geq 
            t
        \right)
        \leq c/t.
    \end{equation}
    For $A \subseteq (0,\infty)$ we define the measure $\nu(A) = \sum_{x,z \in V} \charf \left(d_G(x,z) \in A\right) \lambda_{x,z}$. It may be verified that 
    $\Pi(A) = \sum_{x,z \in V} \charf \left(d_G(x,z) \in A \right) Z_{x,z}$ is a Poisson process on $(0,\infty)$ with 
    intensity measure $\nu$ and hence that $\widehat{C}_{u,v}$ is a compound Poisson process on $(0,\infty)$.
    An analogous calculation to \eqref{eq:finite_projected_edges} gives that $\lambda = \sum_{x,z} \lambda_{x,z} = \nu((0,\infty)) < \infty$.
    Hence $N = \sum_{x,z \in V} Z_{x,z}$ is a Poisson random variable with mean $\lambda$, and the distribution of the length of each jump $J$ satisfies for any $A \subseteq (0,\infty)$ 
    \begin{equation}
    \label{eq:jump_distribution}
        \P
        \left(
            J \in A
        \right)
        =
        \sum_{x,z \in V}
            \lambda_{x,z}
            \charf
            \left(
                d_G(x,z) \in A
            \right)/
            \lambda
        =
        \nu(A)/\lambda.
    \end{equation}
    By summing \eqref{eq:r-measure} for all $r\ge t$ for some $t \in \N$, we obtain that $\P \left(J \geq t \right) = \nu([t,\infty))/\lambda \leq c/(\lambda t)$ for some $c > 0$.
    Conditioned on the total number of jumps $N = n \in \N$, it is a standard fact \cite[Proposition 3.8]{last_lectures_2017} that $\widehat{C}_{u,v} = J_1 + \ldots + J_n$ where the $J_i$'s are i.i.d.\ as in \eqref{eq:jump_distribution}.
    If $J_1 + \ldots + J_n \geq t$ for $t \in \N$ then at least one of the jumps satisfies $J_i \geq t / n$ and by a union bound
    \begin{equation}
        \P
        \left(      
            \widehat{C}_{u,v}
            \geq 
            t
            \mid
            N 
            =
            n
        \right)
        =
        \P
        \left(      
            J_1 
            + 
            \ldots 
            + 
            J_N
            \geq 
            t
            \mid
            N 
            =
            n
        \right)
        \leq
        n
        \P
        \left(
            J
            \geq t/n
        \right)
        \leq 
        \frac{cn^2}{\lambda t}.
    \end{equation}  
    As $N \sim \mathrm{Poi}(\lambda)$ with $\lambda$ finite, the law of total probability and \eqref{eq:domination_chain} concludes the proof 
    \begin{equation}
        \P
        \left(      
            \widehat{C}_{u,v}
            \geq 
            t
        \right)
        =
        \E
        \left[
            \P
            \left(      
                J_1 
                + 
                \ldots 
                + 
                J_N
                \geq 
                t
                \mid
                N 
                =
                n
            \right)
        \right]
        \leq 
        \frac{c}{\lambda t}
        \E[N^2]
        \leq 
        \frac{c(1 + \lambda)}{t}.
    \end{equation}
    %
\vskip-1em 
\end{proof}   

\begin{proof}[Proof of \thref{thm:two_dim_recurrence}]
    It follows from \thref{prop:rand_elec_net_rec} that the distribution of edge-conductances in $\mathcal E_2$ is dominated by a Cauchy tail for $\alpha \geq 2$, which by \thref{lem:projected_network} implies that $\mathcal E_2$ is recurrent almost surely.
    By \thref{prop:rand_elec_net_rec} the effective conductance of $\mathcal E_1$ is greater than or equal to that of $\mathcal E_2$, so $\mathcal E_1$ is also recurrent almost surely, and in particular $\kinfty$ is recurrent almost surely.
\end{proof}
\vskip-1em
    
\vskip-1em
\section*{Acknowledgements} 
We are grateful to G\'abor Pete for a useful discussion on the anchored isoperimetric dimension. 
\vskip-1em
\newcommand{\etalchar}[1]{$^{#1}$}

\end{document}